\documentclass[11pt]{article}
\usepackage{authblk}

\usepackage{amsmath,amsthm,amsfonts,mathtools,bm,hyperref}
\usepackage{amssymb}
\usepackage{dsfont}
\usepackage{mathrsfs}
\usepackage{multirow}
\usepackage{tikz,pgfplots}
\usetikzlibrary{patterns}
\usepackage{subcaption}
\pgfplotsset{compat=newest}
\usepackage{float,indentfirst,algorithm}
\usepackage{pgfplots}
\usepackage{booktabs}
\usepackage{hyperref}
\usepackage{multicol}
\usepackage{nomencl}
\usepackage[margin=1in]{geometry}
\usetikzlibrary{patterns}
\usepackage{cite}

\renewcommand{\cal}{\mathcal}

\newcommand{\cN}{{\cal N}}

\newcommand{\bma}{{\bm{a}}}

\newcommand{\bme}{{\bm{e}}}

\newcommand{\bmj}{{\bm{j}}}

\newcommand{\bmu}{{\bm{u}}}
\newcommand{\bmv}{{\bm{v}}}
\newcommand{\bmw}{{\bm{w}}}

\newcommand{\bmy}{{\bm{y}}}

\newcommand{\bmZ}{\bm Z}
\newcommand{\bmX}{\bm X}

\newcommand{\bmxi}{{\bm \xi}}
\newcommand{\bmtau}{{\bm \tau}}

\newcommand{\rd}{{\rm d}}

\newcommand{\bE}{\mathbb{E}}

\newcommand{\bP}{\mathbb{P}}

\newcommand{\bR}{{\mathbb R}}
\newcommand{\bS}{\mathbb S}

\newcommand{\al}{\alpha}

\DeclareMathOperator{\sgn}{sgn}
\DeclareMathOperator{\OO}{O}
\DeclareMathOperator{\oo}{o}
\DeclareMathOperator{\argmax}{argmax}

\newcommand{\deq}{\mathrel{\mathop:}=} 
 
\renewcommand{\leq}{\leqslant}
\renewcommand{\geq}{\geqslant}

\newcommand{\bmI}{\bm I}

\newcommand{\qq}[1]{[\![{#1}]\!]}

\newcommand{\beq}{\begin{equation}}
\newcommand{\eeq}{\end{equation}}

\theoremstyle{plain} 
\newtheorem{theorem}{Theorem}[section]
\newtheorem*{theorem*}{Theorem}
\newtheorem{lemma}[theorem]{Lemma}
\newtheorem*{lemma*}{Lemma}
\newtheorem{corollary}[theorem]{Corollary}
\newtheorem*{corollary*}{Corollary}

\newtheorem*{proposition*}{Proposition}
\newtheorem{assumption}[theorem]{Assumption}
\newtheorem*{assumption*}{Assumption}
\newtheorem{claim}[theorem]{Claim}

\newtheorem*{definition*}{Definition}

\newtheorem*{example*}{Example}

\newtheorem*{remark*}{Remark}
\newtheorem*{remarks*}{Remarks}

  \title{Power Iteration for Tensor PCA}

  \author[1]{Jiaoyang~Huang\thanks{jh4427@nyu.edu}}
  \author[2]{Daniel Z.~Huang\thanks{dzhuang@caltech.edu}}
  
    \author[3]{Qing~Yang\thanks{yangq@ustc.edu.cn}}
  
    \author[4]{Guang~Cheng\thanks{chengg@purdue.edu}}
  
\affil[1]{New York University, New York, NY}
\affil[2]{California Institute of Technology, Pasadena, CA}
\affil[3]{University of Science and Technology of China, China}
\affil[4]{Purdue University, West Lafayette, IN}

\date{}

\begin{document}



\maketitle

\begin{abstract}
In this paper, we study the power iteration algorithm for the spiked tensor model, as introduced in \cite{NIPS2014_5616}. We give necessary and sufficient conditions for the convergence of the power iteration algorithm. When the power iteration algorithm converges, for the rank one spiked tensor model, we show the estimators for the spike strength and linear functionals of the signal are asymptotically Gaussian; for the multi-rank spiked tensor model, we show the estimators are asymptotically mixtures of Gaussian. This new phenomenon is different from the spiked matrix model.
Using these asymptotic results of our estimators, we construct valid and efficient confidence intervals for spike strengths and linear functionals of the signals.

\end{abstract}

\section{Introduction}

Modern real scientific data call for more advanced  structures  than
matrices. High order arrays, or tensors have been actively considered in neuroimaging analysis, topic modeling, signal processing and recommendation system \cite{frolov2017tensor, Comon,Hack,Kar,Rendle,zhou, simony2016dynamic,cichocki2015tensor,sidiropoulos2017tensor}.
 Setting the stage, imagine that the signal is in the form of a large symmetric low-rank  $k$-th order tensor
\begin{align}\label{e:rankr}
\bmX^*=\sum_{j=1}^r\beta_j \bmv_j^{\otimes k}\in \otimes^k \bR^n,
\end{align}
where $r$ $(r\ll n)$ represents the rank and $\beta_j$ are the strength of the signals.
Such low-rank tensor components appear in various applications, e.g. community detection\cite{Ana1}, moments estimation for latent variable models \cite{Hsu,Ana2} and hypergraph matching \cite{Duch}.
Suppose that we do not have access to perfect measurements about the
entries of this signal tensor. The observations $\bmX=\bmX^*+\bmZ$
are contaminated by a substantial amount of random noise (reflected by the random tensor $\bmZ$ which has  i.i.d. Gaussian entries with mean $0$ and variance $1/n$.). The aim is
to perform reliable estimation and inference on the unseen signal tensor $\bmX^*$. In literature, this is the spiked tensor model, introduced in \cite{NIPS2014_5616}.

In the special case, when $k=2$, the above model reduces to the well-known ``spiked matrix model" \cite{johnstone2001distribution}. In this setting it is known that there is an order $1$ critical signal-to-noise ratio $\beta_{\rm c}$, such that below $\beta_{\rm c}$, it is information-theoretical impossible to detect the spikes, and above $\beta_{\rm c}$, it is possible to detect the
spikes by Principal Component Analysis (PCA). A body of work has quantified the behavior of PCA in this setting \cite{johnstone2001distribution, baik2005phase,baik2006eigenvalues,paul2007asymptotics, benaych2012singular,bai2012sample, johnstone2009consistency, birnbaum2013minimax,cai2013sparse,ma2013sparse,vu2013minimax,cai2015optimal,el2008spectrum,ledoit2012nonlinear,donoho2018optimal}.
We refer readers to the review articles \cite{johnstone2018pca}
for more discussion and references to this and related lines of work

Tensor problems are far more than an extension of matrices. Not only the more involved structures and  high-dimensionality, many concepts are not well defined \cite{Kolda}, e.g. eigenvalues and eigenvectors, and most tensor problems are NP-hard \cite{Hill}. Despite a large body of work tackling the spiked tensor model, there are several fundamental yet unaddressed challenges that deserve further attention.

\textbf{Computational Hardness. } The same as the spiked matrix model, for spiked tensor model, 
there is an order $1$ critical signal-to-noise ratio $\beta_{k}$ (depending on the order $k$), such that below $\beta_{k}$, it is information-theoretical impossible to detect the spikes, and above $\beta_{k}$, the maximum likelihood estimator is a distinguishing statistics \cite{chen2019phase,chen2018phase,lesieur2017statistical,perry2020statistical,jagannath2018statistical}.
In the matrix setting the
maximum likelihood estimator is the top eigenvector, which can be computed in polynomial time
by, e.g., power iteration. However, for order $k\geq 3$ tensor, computing the maximum likelihood estimator is NP-hard in generic setting. In this setting, it is widely believed that there is a regime of signal-to-noise
ratios for which it is information theoretically possible to recover the signal but there
is no known algorithm to efficiently approximate it.
In the pioneer work \cite{NIPS2014_5616}, the algorithmic aspects of this model has been studied under the special setting when the rank $r=1$.  They showed that  tensor power iteration with random initialization recovers the signal provided $\beta\gg n^{(k-1)/2}$, and tensor unfolding recovers the signal provided $\beta\gg n^{(\lceil k/2\rceil-1)/2}$. Based on heuristic arguments, they predicted that the necessary and sufficient 
condition for power iteration to succeed is $\beta\gg n^{(k-2)/2}$, and for tensor unfolding  is $\beta\gg n^{(k-2)/4}$.
 Langevin dynamics and gradient descent were studied in \cite{arous2020algorithmic}, and shown to recover the signal provided $\beta\gg n^{(k-2)/2}$. Later the sharp threshold $\beta\gg n^{(k-2)/4}$ is achieved using Sum-of-Squares algorithms \cite{hopkins2015tensor, hopkins2016fast,kim2017community} and sophisticated iteration algorithms \cite{luo2020sharp, zhang2018tensor}. 
The necessary part of this threshold still remains open, and its relation with hypergraphic planted clique problem was discussed in \cite{luo2020open}.

\textbf{Statistical inferences. }
In many applications, it is often the
case that the ultimate goal is not to characterize the  
$L_2$ or ``bulk" behavior (e.g. the mean squared estimation
error) of the signals, but rather to reason about the signals along a few preconceived yet
important directions. 
In the example of community detecting for hypergraphs, the entries of the vector $\bmv$ can represent different community memberships. The testing of whether any two nodes belong to the same community is reduced to the hypothesis testing problem of whether the corresponding entries of $\bmv$ are equal.
 These problems can
be formulated as estimation and inference for linear functionals of a signal, namely, quantities of the
form $\langle \bma, \bmv_j\rangle$, $1\leq j\leq r$ with a  prescribed vector $\bma$.
A natural starting point is to plug in an estimator $\widehat\bmv_j$ of $\bmv_j$, i.e. the estimator $\langle \bma, \widehat \bmv_j\rangle$. However, a most prior works \cite{NIPS2014_5616, hopkins2015tensor, hopkins2016fast,kim2017community,luo2020sharp, zhang2018tensor} on spiked tensor models focuses on the $L_2$ risk analysis, which
is often too coarse to give tight uncertainty bound for the plug-in estimator.
To further complicate matters, there is often a  bias
issue surrounding the plug-in estimator. Addressing
these issues calls for refined risk analysis of the algorithms.

\subsection{Our Contributions}

We consider the power iteration algorithm given by the following recursion
\begin{align}\label{e:power}
\bmu_0=\bmu,\quad \bmu_{t+1}=\frac{\bmX[\bmu_{t}^{\otimes (k-1)}]}{\|\bmX[\bmu_{t}^{\otimes (k-1)}]\|_2}
\end{align}
where $\bmu\in \bR^n$ with $\|\bmu\|_2=1$ is the initial vector, and $\bmX[\bmv^{\otimes(k-1)}]\in \bR^n$ is the vector with $i$-th entry given by $\langle \bmX, \bme_i\otimes \bmv^{\otimes(k-1)}\rangle$.
The estimators are given by
\begin{align}
\widehat \bmv= \bmu_T,
\quad
\widehat \beta= \langle \bmX, \widehat\bmv^{\otimes k}\rangle.
\end{align}
for some large $T$. 
Although in a worst case scenario, i.e. with random initialization, power iteration algorithm underperforms tensor unfolding.  However, if extra information about the signals $\bmv_j$ is available, power iteration algorithm with a warm start can be used to obtain a much better estimator. In fact this approach is commonly used to obtain refined estimators. In this paper, we study the convergence and statistical inference aspects of  the power iteration algorithm. The main contributions of this paper are summarized below,

\textbf{Convergence criterion}.
We give necessary and sufficient conditions for the convergence of the power iteration algorithm. In the rank one case $r=1$, we show that the power iteration algorithm converges to the true signal $\bmv$, provided $|\beta\langle \bmu, \bmv \rangle^{k-2}|\gg 1 $ where $\bmu$ is the initialization vector. In the complementary setting, if $|\beta\langle \bmu, \bmv \rangle^{k-2}|\ll 1$, the output of the power iteration algorithm behaves like random Gaussian vectors, and has no correlations with the signal.  With random initialization, i.e. $\bmu$ is a uniformly random vector on the unit sphere, our results assert that the power iteration algorithm converges in finite time, if and only if $\beta\gg n^{(k-2)/2}$, which verifies the prediction in \cite{NIPS2014_5616}.
This is analogous to the PCA of spiked matrix model, where power iteration recovers the top eigenvalue.  However,  the multi-rank spiked tensor model, i.e. $r\geq 2$, is  different from multi-rank spiked matrix.  The power iteration algorithm for multi-rank spiked tensor model is more sensitive to the initialization, i.e.
the power iteration algorithm converges if $\max_j |\beta_j\langle \bmu, \bmv_j \rangle^{k-2}|\gg 1$. In this case, it converges to $\bmv_{j_*}$ with $j_*=\argmax_j |\beta_j\langle \bmu, \bmv_j \rangle^{k-2}|$.

\textbf{Statistical inference}
We consider the statistical inference problem for the spiked tensor model.
We develop the limiting distributions of the above power iteration estimators. In the rank one case, above the threshold $|\beta\langle \bmu, \bmv \rangle^{k-2}|\gg 1$, we show that our estimator $\langle   \bma, \widehat \bmv\rangle$ (modulo some global sign) admits the following first order approximation
\begin{align*}
\langle   \bma, \widehat \bmv\rangle\approx \left(1-\frac{1}{2\beta^2}\right)\langle\bma, \bmv\rangle+\frac{\langle \bma^\perp, \bmxi\rangle}{\beta},
\end{align*}
where $\bma^\perp=\bma-\langle \bma, \bmv\rangle\bmv$, and $\bmxi=\bmZ[\bmv^{\otimes(k-1)}]$, is an $n$-dim vector, with each entry i.i.d. $\cN(0,1/n)$ Gaussian random variables. For multi-rank spiked tensor model, 
the output of power iteration algorithm depends on the angle between the initialization $\bmu$ and the signals $\bmv_j$. We consider the case the initialization $\bmu$ is a uniformly random vector on the unit sphere. For such initialization, very interestingly, our estimator $\langle   \bma, \widehat \bmv\rangle$ is asymptotically a mixture of
Gaussian, with modes at $\langle \bma, \bmv_j\rangle$ and mixture weights depending on the signal strength $\beta_j$. 
Using these asymptotic results of our estimators, we
construct valid
and efficient confidence intervals for the linear functionals $\langle \bma, \bmv_j\rangle$.

\subsection{Notations:}
For a vector $\bmv\in \bR^n$, we denote its $i$-th coordinate as $\bmv(i)$.
We equate $k$-th order tensors in $\otimes^k \bR^n$ with  vectors of dimension $n^k$, i.e. $\bm\tau=(\bm\tau_{i_1 i_2\cdots i_k})_{1\leq i_1, i_2,\cdots, i_k\leq n}$. 
For any two $k$-th order tensors $\bm\tau, \bm\eta\in \otimes^{k}\bR^n$, we denote their inner product as $\langle \bmtau, \bm\eta\rangle\deq\sum_{1\leq i_1,i_2,\cdots, i_k\leq n}\bm\tau_{i_1 i_2\cdots i_k}\bm\eta_{i_1 i_2\cdots i_k}$.
A $k$-th order tensor can act on a $(k-1)$-th order tensor, and return a vector: $\bm\tau\in \otimes^k \bR^n$ and $\bm\eta\in \otimes^{k-1} \bR^n$
\begin{align}
\bm\tau[\bm\eta]\in \bR^n, \quad \bm\tau[\bm\eta](i)=\langle \bm\tau,\bme_i\otimes \bm\eta\rangle=\sum_{1\leq i_1,\cdots, i_{k-1}\leq n}\bm\tau_{i i_1 i_2\cdots i_{k-1}}\bm\eta_{i_1 i_2\cdots i_{k-1}}.
\end{align}
We denote the $L_2$ norm of a vector $\bmv$ as $\|\bmv\|$. 
We use $\stackrel{\text{d}}{=}$ for the equality in law, and $\xrightarrow{\text{d}}$ for the convergence in law.
We denote the index sets $\qq{a,b}=\{a,a+1,a+2,\cdots, b\}$ and $\qq{n}=\{1,2,3,\cdots, n\}$.
We use $C$ to represent large universal constant, and $c$ a small universal constant,
which
may be different from line by line. We write that $X = \OO(Y )$ if there exists some universal constant such
that $|X| \leq C Y$ . We write $X = \oo(Y )$ if the ratio $|X|/Y\rightarrow \infty$ as $n$ goes to infinity. We write
$X\asymp Y$ if there exist universal constants such that $c Y \leq |X| \leq  CY$.
We say an event holds with high probability, if for there exists $c>0$, and $n$ large
enough, the event holds with probability at least $1-n^{-c\log n}$.

An outline of the paper is given as follows. In Section \ref{se:rank1}, we state our main results for the rank-one spiked tensor model. In particular,  with general initialization a distributional result for the power iteration
algorithm is developed. Section \ref{se:rankr} investigates the general rank-$r$ spiked tensor model. A similar distributional result is established with general initialization as in Section \ref{se:rank1}. While with uniformly distributed
initialization over the unit sphere, we obtain a multinoimal distribution which yields a mixture Gaussian.  
Numerical simulations are presented in Section \ref{s:numerical}.
All proofs and technical details are deferred to the
appendix.

%
%
%
%
%
%
%

\noindent\textbf{Acknowledgement.}
The research collaboration was initiated when both G.C. and J.H. were warmly hosted by IAS in the special year of Deep Learning Theory.
The research of J.H. is supported by the Simons Foundation as a Junior Fellow
at the Simons Society of Fellows.

\section{Main Results}
\subsection{Rank one spiked tensor model}\label{se:rank1}

In this section, we state our main results for the rank-one spiked tensor model (corresponding to $r=1$ in \eqref{e:rankr}):
\begin{align}
\bmX=\beta \bmv^{\otimes k}+\bmZ,
\end{align}
where
\begin{itemize}
\item $\bmX\in \otimes^k \bR^n$ is the $k$-th order tensor observation.
\item $\bmZ\in \otimes^k \bR^n$ is a noise tensor. The entries of $\bmZ$ are i.i.d. standard $\cN(0,1/n)$ Gaussian random variables.
\item $\beta\in \bR$ is the signal size.
\item $\bmv\in \bR^n$ is an unknown unit vector to be recovered.
\end{itemize}

We obtain a distributional result for  the power iteration algorithm \eqref{e:power} with general initialization $\bmu$: when $|\beta|$ is above certain threshold, $\bmu_t$ converges to $\bmv$, and the error is asymptotically Gaussian; when $|\beta|$ is below the same threshold, the algorithm does not converge. 

\begin{theorem}\label{t:main}
Fix the initialization $\bmu\in \bR^n$ with $\|\bmu\|_2=1$ and $\langle \bmu, \bmv\rangle \gtrsim 1/\sqrt n$.
If $|\beta\langle \bmu, \bmv\rangle^{k-2}|\geq n^{\varepsilon}$ with arbitrarily small $\varepsilon>0$, the behavior of the power iteration algorithm depends on the parity of $k$ and the sign of $\beta$ in the following sense:
\begin{enumerate}
\item \label{c:case1} If $k$ is odd, and $\beta>0$  then $(\bmX[\bmu_t^{\otimes k}], \bmu_{t})$ converges to  $(\beta, \bmv)$;
\item \label{c:case2}If $k$ is odd, and $\beta<0$ then $(\bmX[\bmu_t^{\otimes k}], \bmu_{t})$ converges to  $(-\beta, -\bmv)$;
\item \label{c:case3}If $k$ is even, and $\beta>0$, then $(\bmX[\bmu_t^{\otimes k}], \bmu_{t})$ converges to $(\beta, \sgn(\langle\bmu, \bmv\rangle)\bmv)$ depending on the initialization $\bmu$;
\item \label{c:case4}If $k$ is even, and $\beta<0$, then $(\bmX[\bmu_t^{\otimes k}], \bmu_{t})$ does not converge, but instead alternates between $(\beta, \sgn(\langle\bmu, \bmv\rangle)\bmv)$ and $(\beta, -\sgn(\langle\bmu, \bmv\rangle)\bmv)$.
\end{enumerate}

In Case \ref{c:case1}, for any fixed unit vector $\bma\in \bR^n$, and
\begin{align}
T\geq 1+\frac{1}{\varepsilon}\left(\frac{1}{2}+\frac{2\log|\beta|}{\log n}\right),
\end{align}
with probability $1-\OO(n^{-c(\log n)^2})$, the estimators $\widehat \bmv= \bmu_T$,
and 
$\widehat \beta= \bmX[\widehat\bmv^{\otimes k}]$ satisfies
\begin{align}\begin{split}\label{e:clt}
\langle \bma, \widehat\bmv\rangle=\langle \bma, \bmu_{T}\rangle&=\left(1-\frac{1}{2\beta^2}\right)\langle\bma, \bmv\rangle+\frac{\langle \bma, \bmxi\rangle-\langle \bma, \bmv\rangle\langle \bmv, \bmxi\rangle}{\beta}\\
&+\OO\left(\frac{\log n}{\beta^2\sqrt n}+\frac{(\log n)^{3/2}}{\beta^{3/2}n^{3/4}}+\frac{|\langle \bma, \bmv\rangle|}{\beta^4}\right),
\end{split}\end{align}
where $\bmxi=\bmZ[\bmv^{\otimes(k-1)}]$, is an $n$-dim vector, with each entry i.i.d. $\cN(0,1/n)$ Gaussian random variable. And
\begin{align}\label{e:betaclt}
\widehat \beta=\bmX[\bmu_T^{\otimes k}]=
\beta+\langle \bm\xi, \bmv\rangle -\frac{k/2-1}{\beta}+\OO\left(\frac{\log n}{|\beta| \sqrt n}+\frac{(\log n)^{3/2}}{|\beta|^{1/2}n^{3/4}}+\frac{1}{|\beta|^3}\right).
\end{align}
Under the same assumption, we have similar results for Cases \ref{c:case2}, \ref{c:case3}, \ref{c:case4}, by simply changing $(\beta, \bmv)$
in the righthand side of \eqref{e:clt} and \eqref{e:betaclt} to the corresponding limit.
\end{theorem}

\begin{theorem} \label{t:diverge}
Fix the initialization $\bmu\in \bR^n$ with $\|\bmu\|_2=1$.
If $|\beta|\geq n^{\varepsilon}$ and $|\beta\langle \bmu, \bmv\rangle^{k-2}|\leq n^{-\varepsilon}$ with arbitrarily small $\varepsilon>0$, then $\bmu_{t}$ does not converge to $\pm\bmv$, and $\bmu_t$ behaves like a random Gaussian vector. For
\begin{align}
T\geq 1+\frac{1}{\varepsilon}\left(\frac{1}{2}-\frac{\log|\beta|}{(k-2)\log n}\right)
\end{align}
with probability $1-\OO(n^{-c(\log n)^2})$, it holds
\begin{align}
\widehat\bmv=\bmu_T=\frac{\tilde {\bm\xi}}{\|\tilde {\bm\xi}\|_2}+\OO\left(|\beta|\left(\frac{\log n}{\sqrt n}\right)^{k-1}\right),
\end{align}
where $\tilde{\bm\xi}$ is the standard Gaussian vector in $\bR^n$, the error term is a vector of length bounded by $|\beta|(\log n/\sqrt n)^{k-1}$.

\end{theorem}

%
%

In Theorem \ref{t:main}, we assume that $\langle \bmu, \bmv\rangle \gtrsim 1/\sqrt n$, which is generic and is true for a random $\bmu$. Moreover, if the initial vector $\bmu$ is random, then $|\langle \bmu, \bmv\rangle|\asymp n^{-1/2}$. Notably, Theorems \ref{t:main} and \ref{t:diverge} together state that power iteration recovers $\bmv$ if $|\beta|\gg n^{(k-2)/2}$ and fails if $|\beta|\ll n^{(k-2)/2}$. This gives a rigorous proof of the prediction in \cite{NIPS2014_5616} that the necessary and sufficient condition for the convergence is given by  $|\beta|\gtrsim n^{(k-2)/2}$. In practice, it may be possible to use domain
knowledge to choose better initialization points. For example, in the classical topic modeling applications \cite{Ana2}, the unknown vectors $\bmv$ are related to the topic word distributions, and
many documents may be primarily composed of words from just single topic. Therefore,
good initialization points can be derived from these single-topic documents.

The special case for $k=2$, i.e. the spiked matrix model,  has been intensively studied since the pioneer work of Johnstone \cite{johnstone2001distribution}.
In this setting it is known [30] that
there is an order $\OO(1)$ critical signal-to-noise ratio, such that below the threshold, it is information-theoretically
impossible to recover $\bmv$, and above the threshold, the PCA (partially) recovers the unseen eigenvector $\bmv$
\cite{peche2006largest,abbe2020entrywise, o2018random,vu2011singular, zhong2017eigenvector,chen2018asymmetry,zhang2018heteroskedastic, CYC20}. The special case of our results Theorem \ref{t:main} recovers some abovementioned results.

As a consequence of Theorem \ref{t:main}, we have the following central limit theorem for our estimators.
\begin{corollary}\label{coro1}(Central Limit Theorem)
Fix the initialization $\bmu\in \bR^n$ with $\|\bmu\|_2=1$ and $|\langle \bmu, \bmv\rangle| \gtrsim 1/\sqrt n$.
If $|\beta\langle \bmu, \bmv\rangle^{k-2}|\geq n^{\varepsilon}$ with arbitrarily small $\varepsilon>0$, in Case \ref{c:case1} of Theorem \ref{t:main}, for any fixed unit vector $\bma\in \bR^n$ obeying
\begin{align}\label{e:avbound}
|\langle\bma, \bmv \rangle|=\oo\left(\frac{\beta^3}{\sqrt n}\right),
\end{align}
 and time
\begin{align}
T\geq 1+\frac{1}{\varepsilon}\left(\frac{1}{2}+\frac{2\log|\beta|}{\log n}\right).
\end{align}
the estimators $\widehat \bmv= \bmu_T$,
and
$\widehat \beta= \bmX[\widehat\bmv^{\otimes k}]$ satisfies
\begin{equation}\label{coro:clt1}
\frac{\sqrt{n}\widehat\beta}{\sqrt{\langle \bma, (\bmI_n-\widehat\bmv \widehat\bmv^\top)\bma\rangle}}\left[\big(1-\frac{1}{2\widehat\beta^2}\big)^{-1}\langle \bma, \widehat\bmv\rangle-\langle\bma, \bmv\rangle\right]
\xrightarrow{d} \cN(0,1),
\end{equation}
as $n$ tends to infinity. We have similar results for Cases \ref{c:case2}, \ref{c:case3}, \ref{c:case4}, by simply changing $(\beta, \bmv)$
in \eqref{coro:clt1} to the corresponding limit.
\end{corollary}
We remark that in Corollary \ref{coro1}, we assume that $|\langle\bma, \bmv \rangle|=\oo\left(\beta^3/\sqrt n\right)$, which is generic. For example, if $\bmv$ is delocalized, and $\bma$ is supported on finitely many entries, we will have that $|\langle\bma, \bmv \rangle|\lesssim 1/\sqrt n$, and \eqref{e:avbound} is satisfied.

With the central limit theorem for our estimators in Corollary \ref{coro1}, we can easily write down the confidence interval for our estimators.
\begin{corollary}\label{coro2}(Prediction Interval)
Given the asymptotic significance level $\alpha$, and let $z_\al=\Phi(1-\alpha/2)$ where $\Phi(\cdot)$ is the CDF of a standard Gaussian. If $|\beta\langle \bmu, \bmv\rangle^{k-2}|\geq n^{\varepsilon}$ with arbitrarily small $\varepsilon>0$, in Case \ref{c:case1} of Theorem \ref{t:main}, for any fixed unit vector $\bma\in \bR^n$ obeying
\begin{align}\label{e:avbound11}
|\langle\bma, \bmv \rangle|=\oo\left(\frac{\beta^3}{\sqrt n}\right),
\end{align}
 and time
\begin{align}
T\geq 1+\frac{1}{\varepsilon}\left(\frac{1}{2}+\frac{2\log|\beta|}{\log n}\right),
\end{align}
let $\widehat \bmv= \bmu_T$,
and
$\widehat \beta= \bmX[\widehat\bmv^{\otimes k}]$.
The asymptotic confidence interval of $\langle\bma, \bmv\rangle$ is given by 
\begin{align}\label{coro1:clt1}
\frac{1}{1-1/(2\widehat\beta^2)}\left[\langle \bma,\widehat\bmv \rangle-z_{\alpha}\frac{\sqrt{\langle\bma, (\bmI_n-\widehat\bmv\widehat\bmv^\top)\bma\rangle}}{\sqrt{n}\widehat\beta},\
\langle\bma, \widehat\bmv\rangle+z_{\alpha}\frac{\sqrt{\langle \bma, (\bmI_n-\widehat\bmv\widehat\bmv^\top)\bma\rangle}}{\sqrt{n}\widehat\beta}
\right].
\end{align}
We have similar results for Cases \ref{c:case2}, \ref{c:case3}, \ref{c:case4}, by simply changing $(\beta, \bmv)$
in \eqref{coro1:clt1} to the corresponding limit.
\end{corollary}

\subsection{General Results: rank-$r$ spiked tensor model }\label{se:rankr}
In this section, we state our main results for the general case, the rank-$r$ spiked tensor model (\ref{e:rankr}). 
Before stating our main results, we need to introduce some more notations and assumptions.
\begin{assumption}\label{a:sumpr}
We assume that the initialization does not distinguish $\bmv_1, \bmv_2,\cdots, \bmv_r$, such that there exists some large constant $\kappa>0$
\begin{align}
1/\kappa\leq \left|\frac{\langle \bmu, \bmv_i\rangle}{\langle \bmu, \bmv_j\rangle}\right|\leq \kappa,
\end{align}
for all $1\leq i,j\leq r$.
\end{assumption}
If we take the uniform initialization, i.e.
$\bmu_0=\bmu$ is a uniformly distributed vector in $\bS^{n-1}$. Then with probability  $1-\OO(r/\sqrt{\kappa})$
we will have $1/\sqrt {\kappa n}\leq |\langle\bmu, \bmv_i \rangle|\leq \sqrt {\kappa/n}$ for $1\leq i \leq r$, and Assumption \ref{a:sumpr} holds.

The same as in the rank-1 case, the quantities $|\beta_j\langle \bmu, \bmv_j\rangle^{k-2}|$ play a crucial role in our power iteration algorithm. We need to make the following technical assumption:
\begin{assumption}\label{a:sumpr2}
Let $j_*=\argmax_j |\beta_j\langle \bmu, \bmv_j\rangle^{k-2}|$.
We assume that there exists some large constant $\kappa>0$
\begin{align}
 (1-1/\kappa)|\beta_{j_*}\langle \bmu, \bmv_{j_*}\rangle^{k-2}|\geq |\beta_j\langle \bmu, \bmv_j\rangle^{k-2}|,
\end{align}
for all $1\leq j\leq r$ and $j\neq j_*$.
\end{assumption}

It turns out under Assumptions \ref{a:sumpr} and \ref{a:sumpr2}, the power iteration converges to $\bmv_{j_*}$. Moreover,  if we simply take the uniform initialization, i.e.
$\bmu_0=\bmu$ is a uniformly distributed vector in $\bS^{n-1}$.  Assumption \ref{a:sumpr2} holds for some $1\leq j_*\leq r$ with probability $1-\OO(1/\kappa)$.

\begin{theorem}\label{t:mainr}
Fix the initialization $\bmu\in \bR^n$ with $\|\bmu\|_2=1$ and $|\langle \bmu, \bmv_j\rangle| \gtrsim 1/\sqrt n$, for $1\leq j\leq r$.
Let $j_*=\argmax_j |\beta_j\langle \bmu, \bmv_j\rangle^{k-2}|$. Under Assumptions \ref{a:sumpr} and \ref{a:sumpr2}, if $|\beta_{j_*}\langle \bmu, \bmv_{j_*}\rangle^{k-2}|\geq n^\varepsilon$ with arbitrarily small $\varepsilon>0$, the behavior of the power iteration algorithm depends on the parity of $k$ and the sign of $\beta_{j_*}$:
\begin{enumerate}
\item \label{c:case1r} If $k$ is odd, and $\beta_{j_*}>0$ then $(\bmX[\bmu_t^{\otimes k}], \bmu_{t})$ converges to  $(\beta_{j_*}, \bmv_{j_*})$;
\item \label{c:case2r}If $k$ is odd, and $\beta_{j_*}<0$ then $(\bmX[\bmu_t^{\otimes k}], \bmu_{t})$ converges to  $(-\beta_{j_*}, -\bmv_{j_*})$;
\item \label{c:case3r}If $k$ is even, and $\beta_{j_*}>0$, then $(\bmX[\bmu_t^{\otimes k}], \bmu_{t})$ converges to $(\beta_{j_*}, \sgn(\langle\bmu, \bmv_{j_*}\rangle)\bmv_{j_*})$ depending on the initialization $\bmu$;
\item \label{c:case4r}If $k$ is even, and $\beta_{j_*}<0$, then $(\bmX[\bmu_t^{\otimes k}], \bmu_{t})$ does not converge, but instead alternating between $(\beta_{j_*}, \sgn(\langle\bmu, \bmv_{j_*}\rangle)\bmv_{j_*})$ and $(\beta_{j_*}, -\sgn(\langle\bmu, \bmv_{j_*}\rangle)\bmv_{j_*})$.
\end{enumerate}

In Case \ref{c:case1r}, for any fixed unit vector $\bma\in \bR^n$, and
\begin{align}
T\geq 1+\frac{1}{\varepsilon}\left(\frac{1}{2}+\frac{2\log|\beta_1|}{\log n}\right)+\frac{\log\log(\sqrt n|\beta_1|)}{\log(k-1)},
\end{align}
the estimators $\widehat \bmv= \bmu_T$,
and 
$\widehat \beta= \bmX[\widehat\bmv^{\otimes k}]$ satisfies
\begin{align}\begin{split}\label{e:cltr}
\langle \bma, \widehat\bmv\rangle=\langle \bma, \bmu_{T}\rangle&=\left(1-\frac{1}{2\beta_{j_*}^2}\right)\langle\bma, \bmv_{j_*}\rangle+\frac{\langle \bma, \bmxi\rangle-\langle \bma, \bmv_{j_*}\rangle\langle \bmv_{j_*}, \bmxi\rangle}{\beta_{j_*}}\\
&+\OO_\bP\left(\frac{\log n}{\sqrt n}\left(\frac{\log n}{\sqrt n |\beta_1|}\right)^{k-1}+\frac{\log n}{|\beta_1|^2\sqrt n}+\frac{(\log n)^{3/2}}{|\beta_1|^{3/2}n^{3/4}}+\frac{1}{|\beta_1|^4}\right),
\end{split}\end{align}
where $\bmxi=\bmZ[\bmv_{j_*}^{\otimes(k-1)}]$, is an $n$-dim vector, with each entry i.i.d. $\cN(0,1/n)$ Gaussian random variable. And
\begin{align}\begin{split}\label{e:betacltr}
\widehat \beta
&=\bmX[\bmu_T^{\otimes k}]=\beta_{j_*}+\langle \bm\xi, \bmv_{j_*}\rangle -\frac{k/2-1}{\beta_{j_*}}\\
&+\OO_\bP\left(\frac{\log n}{\sqrt n}\left(\frac{\log n}{\sqrt n |\beta_1|}\right)^{k-1}+\frac{\log n}{|\beta_1| \sqrt n}+\frac{(\log n)^{3/2}}{|\beta_1|^{1/2}n^{3/4}}+\frac{1}{|\beta_1|^3}\right).
\end{split}\end{align}
Under the same assumption, we have similar results for Cases \ref{c:case2r}, \ref{c:case3r}, \ref{c:case4r}, by simply changing $(\beta_{j_*}, \bmv_{j_*})$ in the righthand side of \eqref{e:cltr} and \eqref{e:betacltr} to the corresponding limit.
\end{theorem}
In Theorem \ref{t:mainr}, we assume that $|\langle \bmu, \bmv_j\rangle| \gtrsim 1/\sqrt n$ for $1\leq j\leq r$. This is generic and is true for a random initialization $\bmu$.

We want to remark that for multi-rank spiked tensor model, the senarios for $k=2$, i.e. the spiked matrix model, and $k\geq 3$ are very different. For the spiked matrix model, in Theorem \ref{t:mainr}, we always have that $j_*=\argmax_j|\beta_j|=1$, and power iteration algorithm always converges to the eigenvector corresponding to the largest eigenvalue. However, for rank $k\geq 3$, the power iteration algorithm may converge to any vector $\bmv_j$ provided that the initialization $\bmu$ is sufficiently close to $\bmv_j$.
As a consequence of Theorem \ref{t:mainr}, we have the following central limit theorem for our estimators.
\begin{corollary}\label{coro3}
Fix the initialization $\bmu\in \bR^n$ with $\|\bmu\|_2=1$ and $|\langle \bmu, \bmv_j\rangle| \gtrsim 1/\sqrt n$ for $1\leq j\leq r$.
We assume $|\beta\langle \bmu, \bmv_{j_*}\rangle^{k-2}|\geq n^{\varepsilon}$ with arbitrarily small $\varepsilon>0$, and Assumptions \ref{a:sumpr} and \ref{a:sumpr2}. In Case \ref{c:case1r} of Theorem \ref{t:mainr}, for any fixed unit vector $\bma\in \bR^n$, 
for any fixed unit vector $\bma\in \bR^n$ obeying
\begin{align}\label{e:avboundk}
|\langle\bma, \bmv_{j_*} \rangle|=\oo\left(\frac{|\beta_1|^3}{\sqrt n}\right),
\end{align}
and time
\begin{align}
T\geq 1+\frac{1}{\varepsilon}\left(\frac{1}{2}+\frac{2\log|\beta_1|}{\log n}\right),
\end{align}
the estimators $\widehat \bmv= \bmu_T$,
and
$\widehat \beta=\bmX[\bmu_T^{\otimes k}]$
satisfy
\begin{equation}\label{coro:clt3}
\frac{\sqrt{n}\widehat\beta_{j_*}}{\sqrt{\langle\bma, (\bmI_n-\widehat\bmv \widehat\bmv^\top)\bma\rangle}}
\left[\big(1-\frac{1}{2\widehat\beta_{j_*}^2}\big)^{-1}\langle\bma,\widehat\bmv\rangle-\langle\bma,\bmv_{j_*}\rangle\right]
\xrightarrow{d} \cN(0,1).
\end{equation}
We have similar results for Cases \ref{c:case2r}, \ref{c:case3r}, \ref{c:case4r}, by simply changing $(\beta_{j_*}, \bmv_{j_*})$
in \eqref{coro:clt3} to the corresponding limit.
\end{corollary}

In the following we take $\bmu$ to be a random vector uniformly distributed over the unit sphere. The power iteration algorithm can be easily understood in this setting, thanks to Theorem \ref{t:mainr}. More precisely if $j_*=\argmax_j |\beta_j\langle \bmu, \bmv_j\rangle^{k-2}|$ and the initialization $\bmu$ satisfies Assumptions \ref{a:sumpr} and \ref{a:sumpr2}, then the power iteration estimator $(\widehat\bmv, \widehat\beta)$ recovers $(\bmv_{j_*}, \beta_{j_*})$. From the discussions below, for a random vector $\bmu$ uniformly distributed over the unit sphere, Assumptiosn \ref{a:sumpr} and \ref{a:sumpr2} holds with probability $1-\OO(1/\sqrt\kappa)$. We can compute explicitly the probability that index $i$ achieves $\argmax_j |\beta_j\langle \bmu, \bmv_j\rangle^{k-2}|$:
\begin{align}\begin{split}\label{e:defp}
p_i
&\deq\bP(i=\argmax_j |\beta_j\langle \bmu, \bmv_j\rangle^{k-2}|)\\
&=\int_0^\infty \sqrt{\frac{2}{\pi}}e^{-x^2/2}\left(\prod_{\ell\neq i}\int_0^{\left(\frac{|\beta_i|}{|\beta_\ell|}\right)^{\frac{1}{k-2}}x}\sqrt{\frac{2}{\pi}}e^{-y^2/2}\rd y\right)\rd x,
\end{split}\end{align}
for any $1\leq i\leq r$. For spiked matrix model, i.e. $k=2$, we always have 
$1=\argmax_j |\beta_j\langle \bmu, \bmv_j\rangle^{k-2}|$, and
 $p_1=1, p_2=p_3=\cdots=0$. For spiked tensor models with $k\geq 3$, all those $p_i$ are nonnegative and $p_1\geq p_2\geq p_3\geq \cdots>0$.

\begin{theorem}\label{t:randominit}
Fix large $\kappa>0$ and recall $p_i$ as defined \eqref{e:defp}. If $\bmu$ is uniformly distributed over the unit sphere, and  $|\beta_1|\geq n^{(k-2)/2+\varepsilon}$ with arbitrarily small $\varepsilon>0$, then for any $1\leq i\leq r$:
\begin{enumerate}
\item \label{c:case1rr} If $k$ is odd, and $\beta_i>0$  then with probability $p_i+\OO(1/\sqrt{\kappa})$,
$(\bmX[\bmu_t^{\otimes k}], \bmu_{t})$ converges to  $(\beta_i, \bmv_{i})$;
\item \label{c:case2rr}If $k$ is odd, and $\beta_i<0$ then with probability $p_i+\OO(1/\sqrt{\kappa})$,
 $(\bmX[\bmu_t^{\otimes k}], \bmu_{t})$ converges to  $(-\beta_i, -\bmv_{i})$;
\item \label{c:case3rr}If $k$ is even, and $\beta_i>0$, then  with probability $p_i/2+\OO(1/\sqrt{\kappa})$,
$(\bmX[\bmu_t^{\otimes k}], \bmu_{t})$ converges to $(\beta_i, + \bmv_{i})$, and with probability $p_i/2+\OO(1/\sqrt{\kappa})$,
$(\bmX[\bmu_t^{\otimes k}], \bmu_{t})$ converges to $(\beta_i, -\bmv_{i})$.
\item \label{c:case4rr}If $k$ is even, and $\beta_i<0$, then with probability $p_i+\OO(1/\sqrt{\kappa})$,
 $(\bmX[\bmu_t^{\otimes k}], \bmu_{t})$ alternates between $(\beta_i, \bmv_{i})$ and $(\beta_i, -\bmv_{i})$  .
\end{enumerate}

In Case \ref{c:case1rr}, for any fixed unit vector $\bma\in \bR^n$, and
\begin{align}
T\geq 1+\frac{1}{\varepsilon}\left(\frac{1}{2}+\frac{2\log|\beta_1|}{\log n}\right)+\frac{\log\log(\sqrt n|\beta_1|)}{\log(k-1)},
\end{align}
with probability $p_i+\OO(1/\sqrt{\kappa})$,
the estimators $\widehat \bmv= \bmu_T$,
and 
$\widehat \beta=\bmX[\bmu_T^{\otimes k}]$
satisfy
\begin{align}\begin{split}\label{e:cltrr}
\langle \bma, \widehat\bmv\rangle=\langle \bma, \bmu_{T}\rangle&=\left(1-\frac{1}{2\beta_{i}^2}\right)\langle\bma, \bmv_{i}\rangle+\frac{\langle \bma, \bmxi\rangle-\langle \bma, \bmv_{i}\rangle\langle \bmv_{i}, \bmxi\rangle}{\beta_{i}}\\
&+\OO_\bP\left(\frac{\log n}{\sqrt n}\left(\frac{\log n}{\sqrt n |\beta_1|}\right)^{k-1}+\frac{\log n}{|\beta_1|^2\sqrt n}+\frac{(\log n)^{3/2}}{|\beta_1|^{3/2}n^{3/4}}+\frac{1}{|\beta_1|^4}\right),
\end{split}\end{align}
where $\bmxi=\bmZ[\bmv_{i}^{\otimes(k-1)}]$, is an $n$-dim vector, with each entry i.i.d. $\cN(0,1/n)$ Gaussian random variable. And
\begin{align}\begin{split}\label{e:betacltrr}
\widehat \beta
&=\bmX[\bmu_T^{\otimes k}]=\beta_{i}+\langle \bm\xi, \bmv_{i}\rangle -\frac{k/2-1}{\beta_{i}}\\
&+\OO_\bP\left(\frac{\log n}{\sqrt n}\left(\frac{\log n}{\sqrt n |\beta_1|}\right)^{k-1}+\frac{\log n}{|\beta_1| \sqrt n}+\frac{(\log n)^{3/2}}{|\beta_1|^{1/2}n^{3/4}}+\frac{1}{|\beta_1|^3}\right).
\end{split}\end{align}
Under the same assumption, we have similar results for Cases \ref{c:case2rr}, \ref{c:case3rr}, \ref{c:case4rr}, by simply changing $(\beta_i, \bmv_{i})$ in the righthand side of \eqref{e:cltrr} and \eqref{e:betacltrr} to the corresponding limit.
\end{theorem}

We want to emphasize here that the senarios for $k=2$, i.e. the spiked matrix model, and $k\geq 3$ are very different. For spiked matrix model, i.e. $k=2$, we always have that $p_1=0, p_2=p_3=\cdots=0$. The power iteration algorithm always converges to the eigenvector corresponding to the largest eigenvalue. We can only recover $(\beta_1, \bmv_1)$ no matter how many times we repeat the algorithm. However, for spiked tensor models with $k\geq 3$, all those $p_i$ are nonnegative, $p_1\geq p_2\geq p_3\geq \cdots>0$. By repeating the power iteration algorithm for sufficiently many times, it  recovers $(\beta_i, \bmv_i)$ with probability roughly $p_i$.

Similar to the rank one case in Section \ref{se:rank1}, we are also able to establish the asymptotic distribution and confidence interval for multi-rank spiked tensor model with uniformly distributed initialization $\bmu$. 
\begin{corollary}\label{coro4}
Fix $k\geq 3$, assume $\bmu$ to be a random vector uniformly distributed over the unit sphere and  $|\beta_1|\geq n^{(k-2)/2+\varepsilon}$ with arbitrarily small $\varepsilon>0$.
In Case \ref{c:case1rr} of Theorem \ref{t:randominit}, for any fixed unit vector $\bma\in \bR^n$, 
and time
\begin{align*}
T\geq 1+\frac{1}{\varepsilon}\left(\frac{1}{2}+\frac{2\log|\beta_1|}{\log n}\right)+\frac{\log\log(\sqrt n|\beta_1|)}{\log(k-1)},
\end{align*}
for any $1\leq i\leq r$, with probability $p_i+\OO(1/\sqrt{\kappa})$,
the estimators $\widehat \bmv= \bmu_T$
and
$\widehat \beta=\bmX[\bmu_T^{\otimes k}]$
satisfy
\begin{equation}\label{e:bba}
\frac{\sqrt{n}\widehat\beta}{\sqrt{\langle\bma, (\bmI_n-\widehat\bmv \widehat\bmv^\top)\bma\rangle}}
\left[\langle\bma,\widehat\bmv\rangle-\big(1-\frac{1}{2\widehat\beta^2}\big)\langle\bma,\bmv_{i}\rangle\right]
\xrightarrow{d} \cN(0,1).
\end{equation}
And
\begin{align}\label{e:bbc}
\sqrt n \left(\beta_i-\widehat\beta-\frac{k/2-1}{\widehat\beta}\right)\xrightarrow{d} \cN(0,1).
\end{align}
We have similar results for Cases \ref{c:case2r}, \ref{c:case3r}, \ref{c:case4r}, by simply changing $(\beta_{i}, \bmv_{i})$
above to the corresponding limit.
\end{corollary}
We want to emphasize the difference between Corollary \ref{coro1} and Corollary \ref{coro4}. In the rank one case, the estimators $\widehat \beta$ and $\langle \bma, \widehat \bmv\rangle$ are asymptotically Gaussian. In the multi-rank spiked tensor model with $k\geq 3$, those estimators $\widehat \beta$ and $\langle \bma, \widehat \bmv\rangle$ are no longer Gaussian. Instead, they are asymptotically a mixture Gaussian with mixture weights $p_1\geq p_2\geq p_3\geq\cdots$.

\begin{corollary}\label{coro5}
Given the asymptotic significance level $\alpha$, and let $z_\al=\Phi(1-\alpha/2)$ where $\Phi(\cdot)$ is the CDF of a standard Gaussian.
Under the conditions in Corollary \ref{coro4}, 
in Case \ref{c:case1rr} of Theorem \ref{t:randominit}, we can find the asymptotic confidence interval of $\langle\bma, \bmv_i\rangle$  as
\begin{align*}
\frac{1}{1-1/(2\widehat\beta^2)}\left[\langle \bma,\widehat\bmv \rangle-z_{\alpha}\frac{\sqrt{\langle\bma, (\bmI_n-\widehat\bmv\widehat\bmv^\top)\bma\rangle}}{\sqrt{n}\widehat\beta},\
\langle\bma, \widehat\bmv\rangle+z_{\alpha}\frac{\sqrt{\langle \bma, (\bmI_n-\widehat\bmv\widehat\bmv^\top)\bma\rangle}}{\sqrt{n}\widehat\beta}
\right]
\end{align*}
and the asymptotic confidence interval of $\beta_i$ as
\begin{align*}
\left[\widehat\beta+\frac{k/2-1}{\widehat\beta}-\frac{z_\al}{\sqrt n}, \quad
\widehat\beta+\frac{k/2-1}{\widehat\beta}+\frac{z_\al}{\sqrt n}\right].
\end{align*}
We have similar results for Cases \ref{c:case2r}, \ref{c:case3r}, \ref{c:case4r}, by changing $(\beta_{i}, \bmv_{i})$
above to the corresponding limit.
\end{corollary}

\section{Numerical Study}\label{s:numerical}

In this section, we conduct numerical experiments on synthetic data to demonstrate our distributional results provided in Sections \ref{se:rank1} and \ref{se:rankr}. We fix the dimension $n=600$ and rank $k=3$.

\subsection{Rank one spiked tensor model}
We begin with numerical experiments on rank one case.  
This section is devoted to numerically studying the efficiency of our estimators for the strength of signals and linear functionals of the
signals. 
We take the signal $\bmv$ a random vector sampled from the unit sphere in $\bR^n$, and the vector 
\begin{align}
\bma=\frac{1}{\sqrt 3}(\bme_{n/3}+\bme_{2n/3}+\bme_{n})
\end{align}
 For the setting without prior information of the signal, we take the initialization of our power iteration algorithm $\bmu$ a random vector sampled from the unit sphere in $\bR^n$, and the strength of signal $\beta=n^{(k-2)/2}\approx 24.495$. We plot in Figure \ref{f:rank1_1} our estimators for the strength of signals after normalization
 \begin{align}\label{e:nbeta}
\widehat\beta+\frac{k/2-1}{\widehat\beta}-\beta
 \end{align}
 and our estimators for the linear functionals of the
signals
\begin{align}\label{e:nav}
\frac{\sqrt{n}\widehat\beta}{\sqrt{\langle \bma, (\bmI_n-\widehat\bmv \widehat\bmv^\top)\bma\rangle}}\left[\big(1-\frac{1}{2\widehat\beta^2}\big)^{-1}\langle \bma, \widehat\bmv\rangle-\langle\bma, \bmv\rangle\right]
\end{align}
as in Corollary \ref{coro1}.
\begin{figure}
\begin{center}
 \includegraphics[scale=0.8,trim={0 0 0 0},clip]{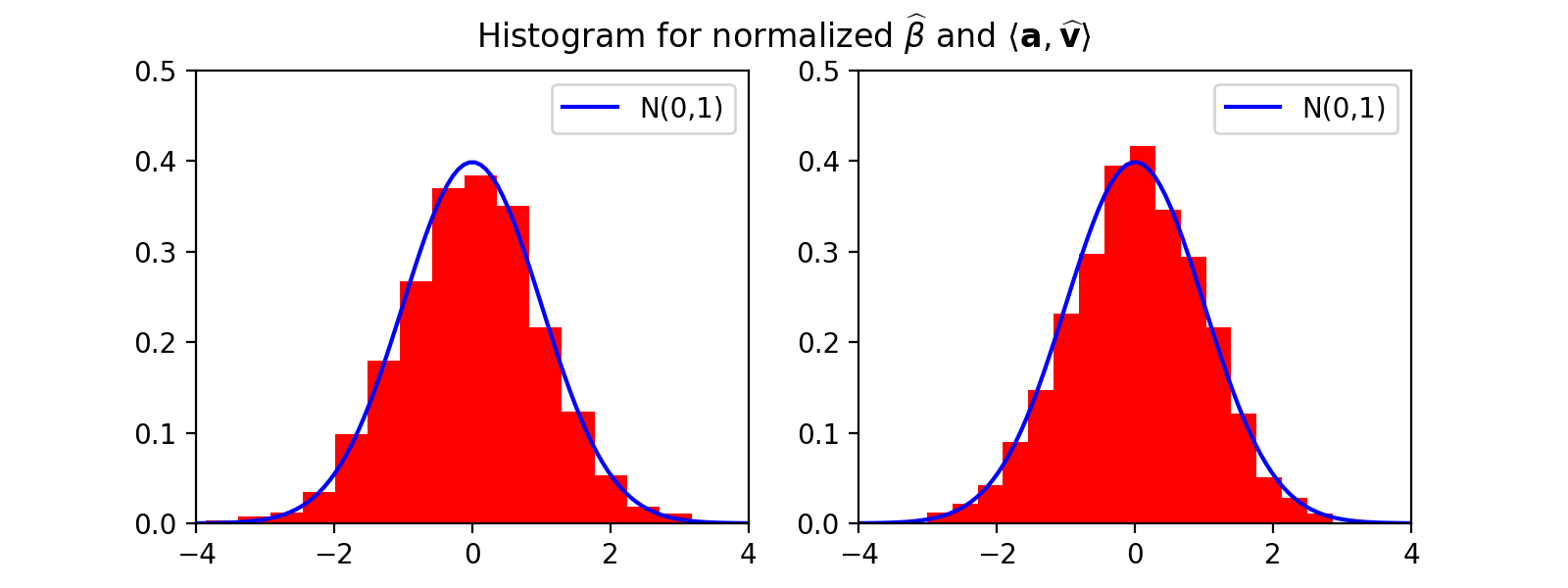}
 \caption{The empirical density of normalized $\widehat \beta$ as in \eqref{e:nbeta} (left panel), and normalized $\langle \bma, \widehat \bmv\rangle$ as in \eqref{e:nav}.  The results are reported over $2000$ independent trials where the initialization of our power iteration algorithm $\bmu$ a random vector sampled from the unit sphere in $\bR^n$, and the strength of signal $\beta=n^{(k-2)/2}\approx 24.495$.}
 \label{f:rank1_1}
 \end{center}
 \end{figure}
 
For the setting that there is prior information of the signal, we take the initilization of our power iteration algorithm $\bmu=(\bmv+\bmw)/\|\bmv+\bmw\|_2$, where $\bmv$ is a random vector sampled from the unit sphere in $\bR^n$.
We plot  our estimators for the strength of signals after normalization \eqref{e:nbeta}
 and our estimators for the linear functionals of the
signals \eqref{e:nav} for $\beta=5$ in Figure \ref{f:rank1_2}, and for $\beta=10$ in Figure \ref{f:rank1_3}.
\begin{figure}
\begin{center}
 \includegraphics[scale=0.8,trim={0 0 0 0},clip]{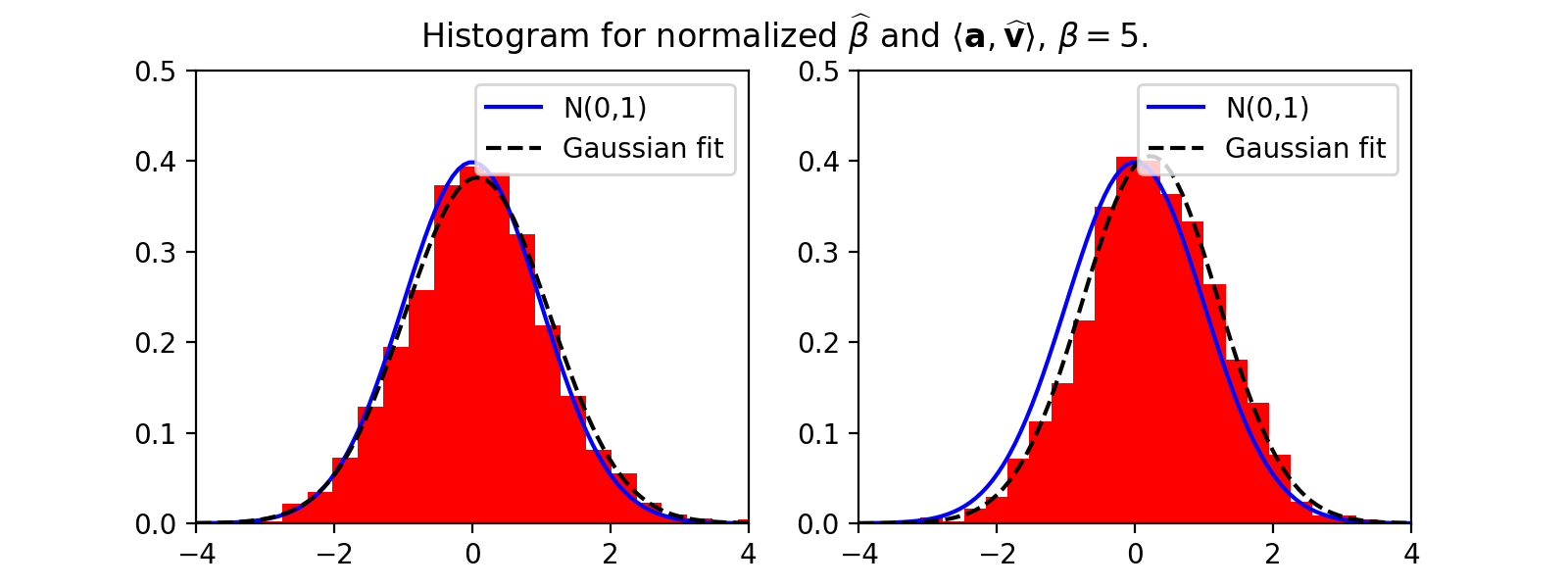}
 \caption{The empirical density of normalized $\widehat \beta$ as in \eqref{e:nbeta} (left panel), and normalized $\langle \bma, \widehat \bmv\rangle$ as in \eqref{e:nav}.  The results are reported over $2000$ independent trials where the initialization of our power iteration algorithm $\bmu$ a random vector sampled from the unit sphere in $\bR^n$, and the strength of signal $\beta=5$.}
 \label{f:rank1_2}
 \end{center}
 \end{figure}
 Although our Theorem \ref{t:main} and Corollary \ref{coro1} requires $|\beta \langle \bmu, \bmv \rangle^{k-2}|\geq n^{\varepsilon}\gg 1$, Figures \ref{f:rank1_2} and \ref{f:rank1_3} indicate that our estimators $\widehat \beta$ and $\langle \bma, \widehat\bmv\rangle$ are asymptotically Gaussian even with small $\beta$, i.e. $\beta=5,10$. Theorem \ref{t:main} also indicates that error term in Corollary \eqref{coro1}, i.e. the error term in \eqref{coro:clt1}, is of order $1/|\beta|$. This matches with our simulation. In Figures \ref{f:rank1_2} and \ref{f:rank1_3}, the the difference between the Gaussian fit of our empirical density and the density of $\cN(0,1)$ decreases as $\beta$ increases from $5$ to $10$.
 \begin{figure}
\begin{center}
 \includegraphics[scale=0.8,trim={0 0 0 0},clip]{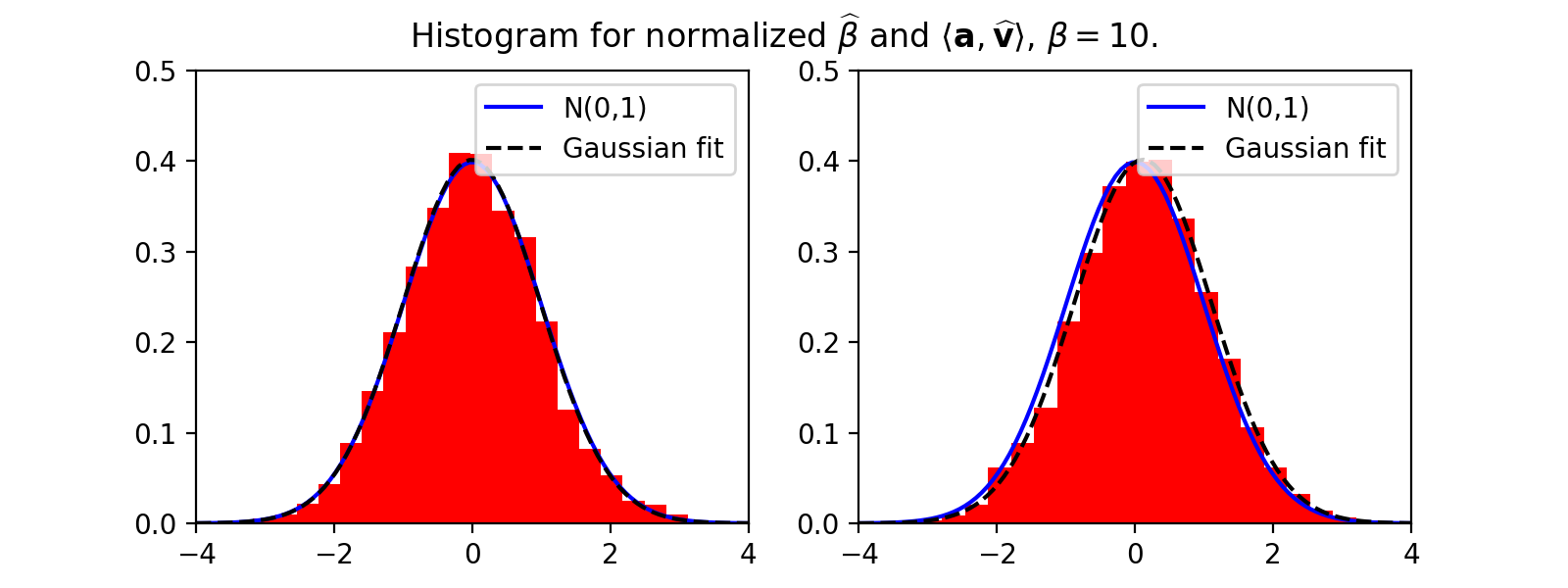}
 \caption{The empirical density of normalized $\widehat \beta$ as in \eqref{e:nbeta} (left panel), and normalized $\langle \bma, \widehat \bmv\rangle$ as in \eqref{e:nav}.  The results are reported over $2000$ independent trials where the initialization of our power iteration algorithm $\bmu$ a random vector sampled from the unit sphere in $\bR^n$, and the strength of signal $\beta=10$.}
 \label{f:rank1_3}
 \end{center}
 \end{figure}

 In Figure \ref{f:rank1_4}, we test the threshold signal-to-noise ratio  for the power iteration algorithm.
Our Theorems \ref{t:main} and \ref{t:diverge} state that for $|\beta\langle \bmu_0,\bmv\rangle^{k-2}|\gg1$ 
 tensor power iteration recovers the signal $\bmv$, and fails when  $|\beta\langle \bmu_0,\bmv\rangle^{k-2}|\ll1$. Especially for random initialization, we have that $|\langle \bmu_0,\bmv\rangle|\asymp 1/\sqrt n$. Our Theorems state that for $|\beta|\gg n^{(k-2)/2}$ 
 tensor power iteration recovers the signal $\bmv$, and fails when  $|\beta|\ll n^{(k-2)/2}$.
Take $k=3$.  In the left panel of Figure \ref{f:rank1_4}, we test tensor power iteration with random initialization for various dimensions $n\in\{200,300,400,500,600\}$ and signal strength $\beta/\sqrt n\in (0,2]$. 
  In the right panel of Figure \ref{f:rank1_4}, we test tensor power iteration with fixed small $\beta=3$ and informative initialization $\beta\langle \bmu_0,\bmv\rangle\in(0,2]$ for various dimensions $n\in\{200,300,400,500,600\}$.  The outputs $\langle \widehat\bmv, \bmv\rangle$ are averaged over $60$ independent trials.

  \begin{figure}
\begin{center}
 \includegraphics[scale=0.7,trim={0 0 0 0},clip]{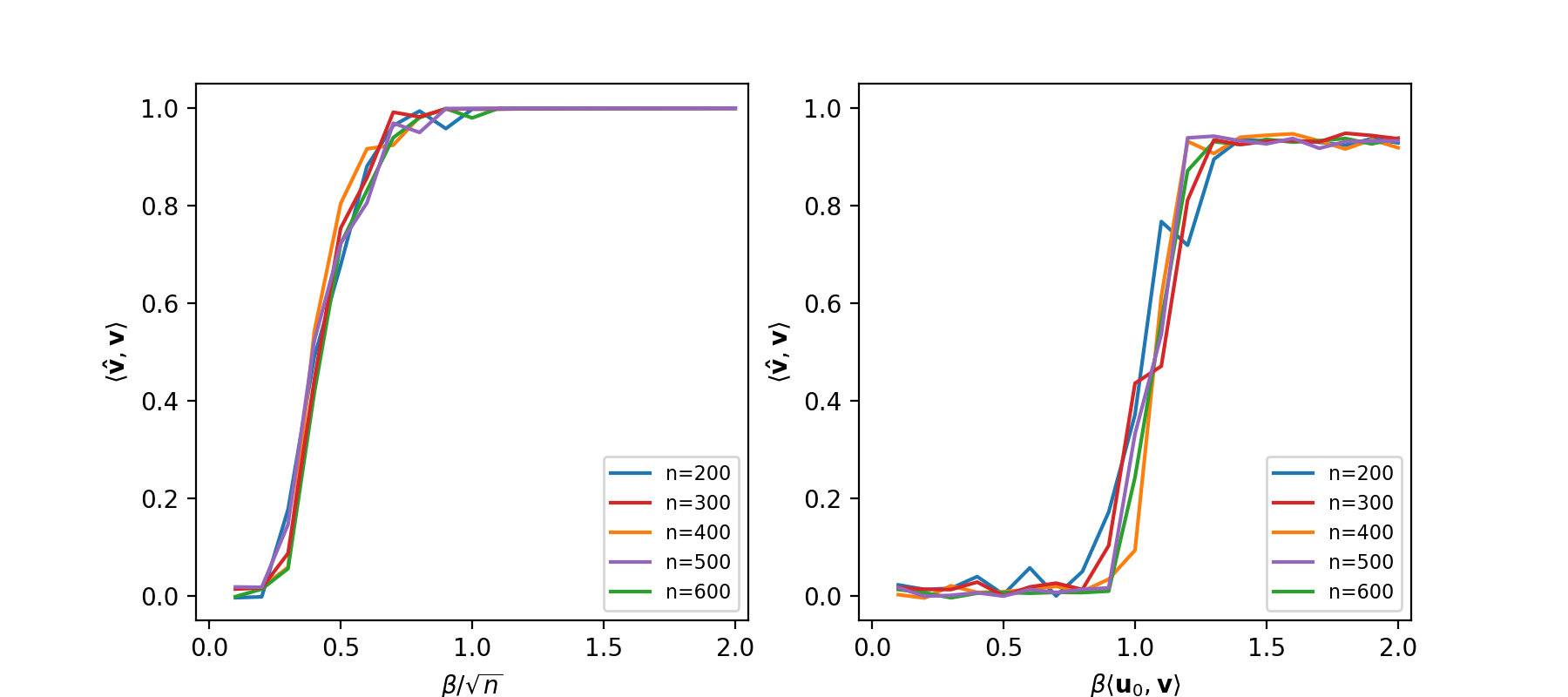}
 \caption{Output of tensor power iteration with random initialization for various signal strength $\beta/\sqrt n\in (0,2]$ (left panel), and  tensor power iteration with fixed small $\beta=3$ and informative initialization $\beta\langle \bmu_0,\bmv\rangle\in(0,2]$.}
 \label{f:rank1_4}
 \end{center}
 \end{figure}

\subsection{Rank-$r$ spiked tensor model}

In this section, we conduct numerical experiments to demonstrate our distributional results for the multi-rank spiked tensor model. We consider the simplest case that there are two spikes with signals $\bmv_1, \bmv_2$, such that they are uniformly sampled from the unit sphere in $\bR^n$ and orthogonal to each other $\langle\bmv_1,\bmv_2\rangle=0$,  and the vector 
\begin{align}
\bma=\frac{1}{\sqrt 3}(\bme_{n/3}+\bme_{2n/3}+\bme_{n}).
\end{align}

 \begin{figure}
\begin{center}
 \includegraphics[scale=0.7,trim={0 0 0 0},clip]{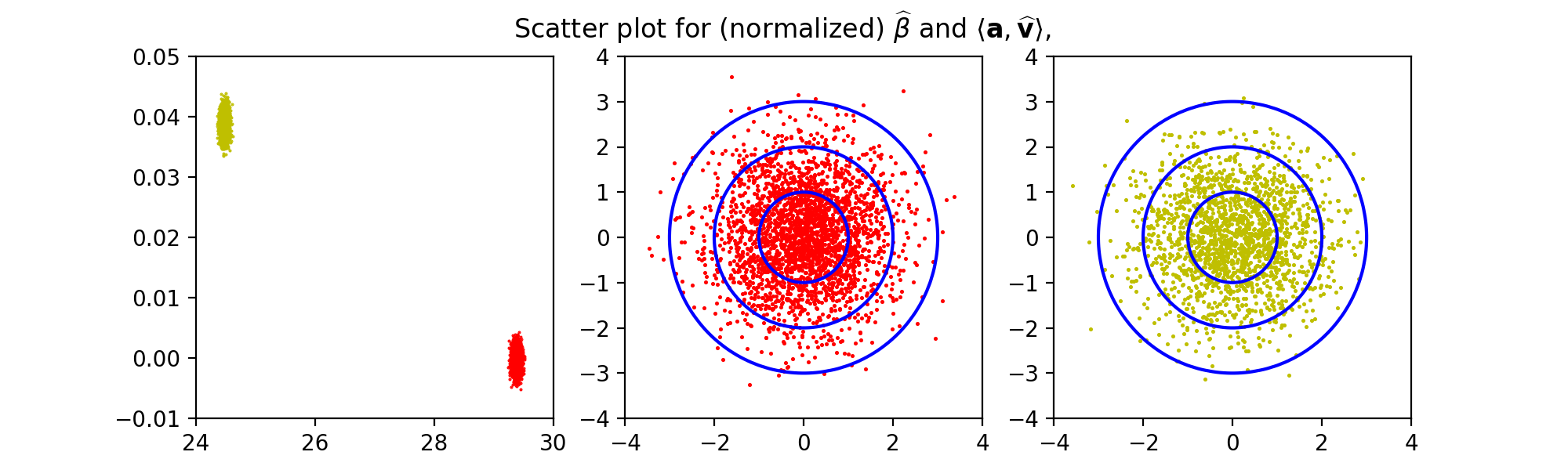}
 \caption{Scatter plot of $(\widehat \beta,\langle \bma, \widehat \bmv\rangle)$ (first panel),
 the normalized $(\widehat \beta,\langle \bma, \widehat \bmv\rangle)$ as in \eqref{e:firstcluster} for the cluster corresponding to $(\beta_1, \langle \bma, \bmv_1\rangle)$ (second panel),  the normalized $(\widehat \beta,\langle \bma, \widehat \bmv\rangle)$ as in \eqref{e:secondcluster} for the cluster corresponding to $(\beta_2, \langle \bma, \bmv_2\rangle)$. The contour plot is a standard $2$-dim Gaussian distribution, at $1,2,3$ standard deviation.
  The results are reported over $5000$ independent trials where the initialization of our power iteration algorithm $\bmu$ a random vector sampled from the unit sphere in $\bR^n$.}
 \label{f:rank1_4}
 \end{center}
 \end{figure}
 
We test the setting that there is no prior information of the signal. We take the strength of signals $\beta_1=1.2\times n^{(k-2)/2}\approx 29.394$ and $\beta_2= n^{(k-2)/2}\approx 24.495$ and the initialization of our power iteration algorithm $\bmu$ a random vector sampled from the unit sphere in $\bR^n$.  We scatter plot in Figure \ref{f:rank1_4} our estimator $\widehat\beta$ for the strength of signals, and  our estimator $\langle \bma, \widehat \bmv\rangle$ for the linear functionals of the
signals over $5000$ independent trials. As seen in the first panel of Figure \ref{f:rank1_4}, our estimators $(\widehat\beta, \langle\bma, \widehat\bmv \rangle)$ form two clusters, centered around $(\beta_1,  \langle\bma, \widehat\bmv_1\rangle)\approx(29.394, 0.000)$ and $(\beta_2,  \langle\bma, \widehat\bmv_2\rangle)\approx(24.495, 0.039)$. In the second and third panels, we zoom in, and scatter plot for the cluster corresponding to $(\beta_1,  \langle\bma, \widehat\bmv_1\rangle)\approx(29.394, 0.000)$
 \begin{align}\label{e:firstcluster}
\widehat\beta+\frac{k/2-1}{\widehat\beta}-\beta, \quad \frac{\sqrt{n}\widehat\beta_1}{\sqrt{\langle \bma, (\bmI_n-\widehat\bmv \widehat\bmv^\top)\bma\rangle}}\left[\big(1-\frac{1}{2\widehat\beta^2}\big)^{-1}\langle \bma, \widehat\bmv\rangle-\langle\bma, \bmv_1\rangle\right],
 \end{align}
 and scatter plot for the cluster corresponding to $(\beta_2,  \langle\bma, \widehat\bmv_2\rangle)\approx(24.495, 0.039)$
\begin{align}\label{e:secondcluster}
\widehat\beta+\frac{k/2-1}{\widehat\beta}-\beta_2,\quad \frac{\sqrt{n}\widehat\beta}{\sqrt{\langle \bma, (\bmI_n-\widehat\bmv \widehat\bmv^\top)\bma\rangle}}\left[\big(1-\frac{1}{2\widehat\beta^2}\big)^{-1}\langle \bma, \widehat\bmv\rangle-\langle\bma, \bmv_2\rangle\right].
\end{align}
As predicted by our Theorem \ref{t:randominit}, both clusters are asymptotically Gaussian, and the normalized estimators matches pretty well with the contour plot of standard $2$-dim Gaussian distribution, at $1,2,3$ standard deviation.

 We plot in Figure \ref{f:rank1_5} our estimators for the strength of signals and the linear functionals of the
signals after normalization, for the first cluster \eqref{e:firstcluster}, and for the second cluster \eqref{e:secondcluster}.

  \begin{figure}
\begin{center}
 \includegraphics[scale=0.8,trim={0 0 0 0},clip]{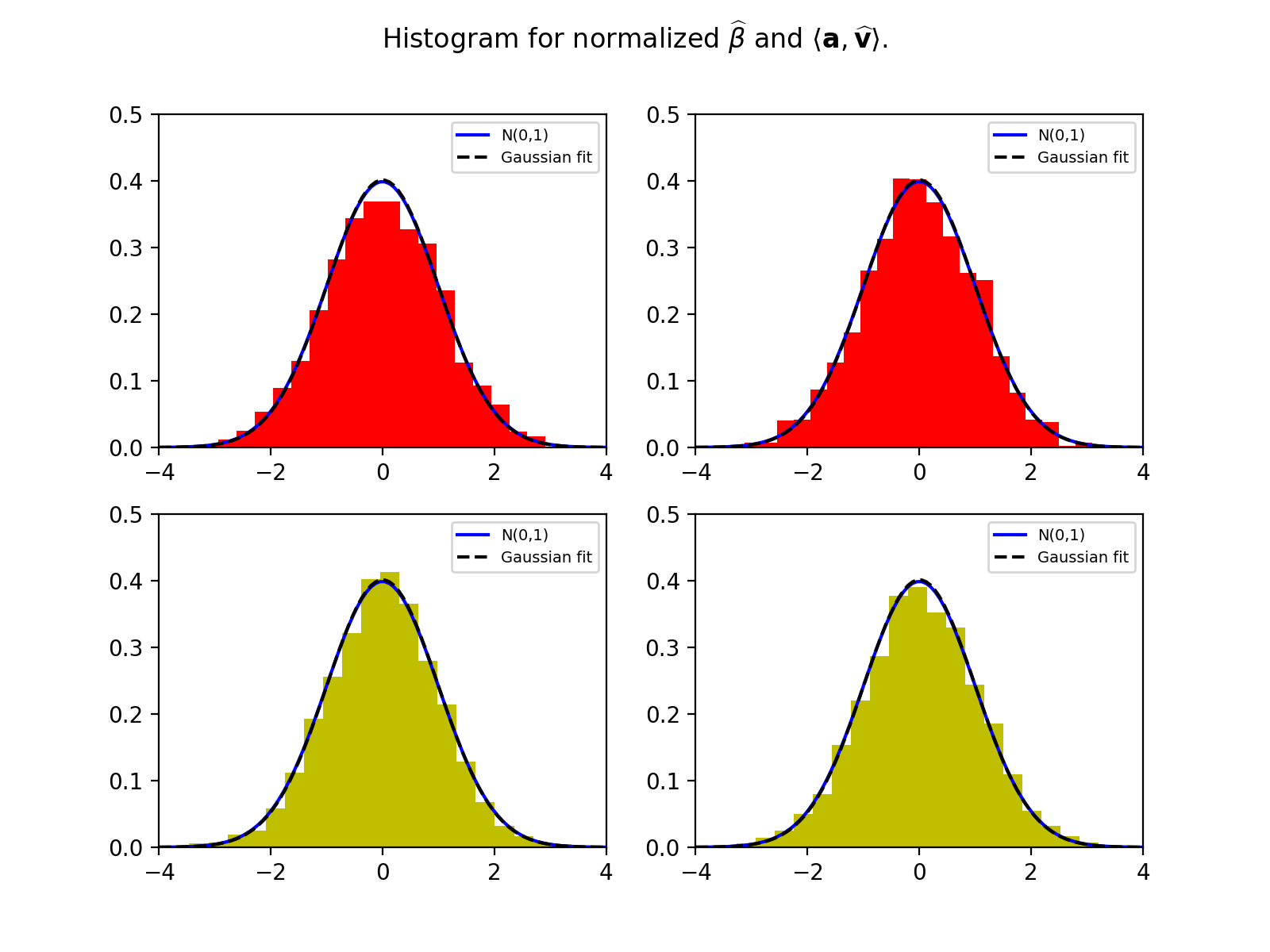}
 \caption{The empirical density of 
 the normalized $(\widehat \beta,\langle \bma, \widehat \bmv\rangle)$ as in \eqref{e:firstcluster} for the cluster corresponding to $(\beta_1, \langle \bma, \bmv_1\rangle)$ (second panel),  the normalized $(\widehat \beta,\langle \bma, \widehat \bmv\rangle)$ as in \eqref{e:secondcluster} for the cluster corresponding to $(\beta_2, \langle \bma, \bmv_2\rangle)$. 
  The results are reported over $5000$ independent trials where the initialization of our power iteration algorithm $\bmu$ a random vector sampled from the unit sphere in $\bR^n$.}
 \label{f:rank1_5}
 \end{center}
 \end{figure}

 In Table \eqref{t:table},  for each $n\in \{50,100,200,400,600,800\}$ and $k=3$, we take the strength of signals $\beta_1= n^{(k-2)/2}$ and $\beta_2= 1.2\times n^{(k-2)/2}$. Over  $1000$ independent trials for power iteration with random initialization for each $n$, we estimate the percentage $\widehat p_1$ of estimators converging to $\beta_1$, and the percentage $\widehat p_2$ of estimators converging to $\beta_2$. Our theoretical values are 
 \begin{align*}
 p_1=\bP(|\beta_1\langle \bmu, \bmv_1\rangle|>|\beta_2\langle \bmu, \bmv_2\rangle|)\approx 0.44,\\
  p_2=\bP(|\beta_1\langle \bmu, \bmv_1\rangle|<|\beta_2\langle \bmu, \bmv_2\rangle|)\approx 0.56.\\
 \end{align*}
 We also exam the numerical coverage rates for our $95\%$ confidence intervals over $1000$ independent trials.

\begin{table}
\centering
\begin{tabular}{ |c|c|c|c|c|c|c| } 
 \hline
  & $n=50$ & $n=100$& $n=200$ & $n=400$ & $n=600$ & $n=800$ \\ 
  \hline
  $\widehat p_1$ & 0.405& 0.399& 0.421 &  0.381 & 0.422 & 0.401   \\
  \hline
  $\widehat p_2$& 0.595& 0.579 & 0.601  & 0.619 & 0.578 &  0.599  \\
  \hline
  signal $\beta_1$& 0.9136 & 0.9223& 0.9596& 0.9291 & 0.9313 &   0.9551\\
  \hline
  linear form $\langle \bma, \bmv_1\rangle$ &0.9680 & 0.9499& 0.9572 &  0.9580&   0.9668&   0.9526\\
  \hline
  signal $\beta_2$& 0.9462 &0.9334& 0.9430 &   0.9612 &  0.9602 &   0.9599\\
  \hline
  linear form $\langle \bma, \bmv_2\rangle$ &  0.9445&  0.9434  & 0.94819 & 0.9677 &  0.9533 &  0.9549    \\
 \hline
 \end{tabular}
\caption{Estimated $\widehat p_1, \widehat p_2$ over $1000$ independent trials for dimension $n\in \{50,100, 200,400,600,800\}$ (top two rows), and numerical coverage rates for our $95\%$ confidence intervals over $1000$ independent trials for dimension $n\in \{50,100,200,400,600,800\}$ (last four rows).}
\label{t:table}
\end{table}

\section{Proof of main theorems}\label{se:appendix}
\subsection{Proof of Theorems \ref{t:main} and \ref{t:diverge}}
The following lemma on the conditioning of Gaussian tensors will be repeatedly use in the remaining of this section.

\begin{lemma}\label{l:decompose}
Let $\bmZ\in \otimes^k \bR^n$ be a random Gaussian tensor. The entries of $\bmZ$ are i.i.d. standard $\cN(0,1/n)$ Gaussian random variables. Fix $\bm\tau_1, \bm\tau_2,\cdots, \bm\tau_t\in \otimes^{k-1}\bR^n$ orthonormal $(k-1)$-th order tensors, i.e. $\langle\bm\tau_i,\bm\tau_j\rangle=\delta_{ij}$, and vectors $\bm\xi_1, \bm\xi_2, \cdots, \bm\xi_t\in \bR^n$. Then the distribution of $\bmZ[\bm\tau]$ conditioned on $ \bmZ[\bm\tau_s]=\bm\xi_s$ for $1\leq s\leq t$ is
\begin{align*}
\bmZ[\bm\tau]\stackrel{\text{d}}{=}\sum_{s=1}^t \langle \bm\tau_s, \bm\tau\rangle\bm\xi_s+\tilde \bmZ\left[\bm\tau-\sum_{s=1}^t \langle \bm\tau_s, \bm\tau\rangle\bm\tau_s\right],
\end{align*}
where $\tilde \bmZ$ is an independent copy of $\bmZ$.
\end{lemma}
\begin{proof}[Proof of Lemma \ref{l:decompose}]
For any $(k-1)$-th order tensor $\bmtau$, viewed as a vector in $\bR^{n^{k-1}}$, we can decompose it as the projection on the span of $\bmtau_1, \bmtau_2,\cdots, \bmtau_t$ and the orthogonal part
\begin{align}\label{e:decompose}
    \bmtau=\sum_{s=1}^t \langle \bm\tau_s, \bm\tau\rangle\bmtau_s+ \left(\bm\tau-\sum_{s=1}^t \langle \bm\tau_s, \bm\tau\rangle\bm\tau_s\right).
\end{align}
Using the above decomposition and $\bmZ[\bmtau_s]=\bmxi_s$, we can write $\bmZ[\bm\tau]$ as
\begin{align}\label{e:decomp}
  \bmZ[\bm\tau]\stackrel{\text{d}}{=}\sum_{s=1}^t \langle \bm\tau_s, \bm\tau\rangle\bm\xi_s+ \bmZ\left[\bm\tau-\sum_{s=1}^t \langle \bm\tau_s, \bm\tau\rangle\bm\tau_s\right],
\end{align}
and the first sum and the second term on the righthand side of \eqref{e:decomp} are independent. The claim \eqref{e:decompose} follows.
\end{proof}

\begin{proof}[Proof of Theorem \ref{t:main}]
We define an auxiliary iteration, $\bmy_0=\bmu$ and
\begin{align}\label{e:recur}
\bmy_{t+1}=\bmX[\bmy_t^{\otimes(k-1)}].
\end{align}
Then with $\bmy_t$, our original power iteration \eqref{e:power} is given by $\bmu_t=\bmy_t/\|\bmy_t\|_2$.

Let $\bmxi=\bmZ[\bmv^{\otimes(k-1)}]\in \bR^n$. Then the entries of $\bmxi$ are given by
\begin{align}
\bmxi(i)=\bmZ[ \bmv^{\otimes(k-1)}](i)=\langle \bmZ,\bme_i\otimes \bmv^{\otimes(k-1)}\rangle=\sum_{i_1,i_2,\cdots, i_{k-1}\in\qq{1, n}}\bmZ_{i i_1i_2\cdots i_{k-1}}\bmv(i_1)\bmv(i_2)\cdots \bmv(i_{k-1}).
\end{align}
From the expression, $\bmxi(i)$ is
a linear combination of Gaussian random variables, itself is also a Gaussian. Moreover, 
these entries $\bmxi(i)$ are i.i.d. Gaussian variables with mean zero and variance $1/n$:
\begin{align}
\bE[\bm\xi(i)^2]=\sum_{i_1,i_2,\cdots, i_{k-1}\in\qq{1, n}}\bE[\bmZ^2_{i i_1i_2\cdots i_{k-1}}]\bmv(i_1)^2 \bmv(i_2)^2\cdots \bmv(i_{k-1})^2=\frac{1}{n}.
\end{align}

We can compute $\bmy_t$ iteratively: $\bmy_1$ is given by
\begin{align}\label{e:y1}
\bmy_1=\bmX[\bmy_0^{\otimes(k-1)}]=\beta\langle\bmy_0,\bmv\rangle^{k-1}\bmv+\bmZ[\bmy_0^{\otimes(k-1)}].
\end{align}
For the last term on the righthand side of \eqref{e:y1}, we can decompose $\bmy_0^{\otimes(k-1)}$ as a projection on $\bmv^{\otimes(k-1)}$ and its orthogonal part:
\begin{align}
\bmy_0^{\otimes(k-1)}=\langle \bmy_0, \bmv\rangle^{k-1}\bmv^{\otimes(k-1)}
+\sqrt{1-\langle \bmy_0, \bmv\rangle^{2(k-1)}}\bmtau_0,
\end{align}
where $\bmtau_0\in \otimes^{(k-1)}\bR^n$ and $\langle \bmv^{\otimes (k-1)}, \bm\tau_0\rangle=0$, $\langle \bmtau_0, \bm\tau_0\rangle=1$. Thanks to Lemma \ref{l:decompose}, conditioning on $\bmxi=\bmZ[\bmv^{\otimes(k-1)}]$, $\bmxi_1=\bmZ[\bmtau_0]$ has the same law as $\tilde \bmZ[\bmtau_0]$, where $\tilde \bmZ$ is an independent copy of $\bmZ$. Since $\langle \bmtau_0, \bm\tau_0\rangle=1$, $\bm\xi_1$ is a Gaussian vector with each entry $\cN(0,1/n)$. With those notations we can rewrite the expression \eqref{e:y1} of $\bmy_1$  as
\begin{align}\label{e:y1c}
\bmy_1=\beta\langle\bmy_0,\bmv\rangle^{k-1}\bmv
+\langle \bmy_0, \bmv\rangle^{k-1}\bmxi
+\sqrt{1-\langle \bmy_0, \bmv\rangle^{2(k-1)}}\bmxi_1.
\end{align}

In the following we show that: 

\begin{claim}\label{c:expression}
We can compute $\bmy_1, \bmy_2, \bmy_3, \cdots, \bmy_t$ inductively. The Gram-Schmidt orthonormalization procedure gives an orthogonal base of
$\bmv^{\otimes (k-1)}, \bmy_0^{\otimes (k-1)}, \bmy_1^{\otimes (k-1)}, \cdots, \bmy_{t-1}^{\otimes (k-1)}$ as:
\begin{align}\label{e:base}
\bmv^{\otimes (k-1)}, \bmtau_0, \bmtau_1, \cdots, \bmtau_{t-1}.
\end{align}
Let $\bmxi_{s+1}=\bmZ[\bmtau_s]$ for $0\leq s\leq t-1$. 
Conditioning on $\bmxi=\bmZ[\bmv^{\otimes(k-1)}]$ and $\bmxi_{s+1}=\bmZ[\bmtau_s]$ for $0\leq s\leq t-2$, $\bmxi_{t}=\bmZ[\bmtau_{t-1}]$ is an independent Gaussian vector, with each entry $\cN(0,1/n)$.
Then  $\bmy_t$ is in the following form
\begin{align}\label{e:btforms}
\bmy_t=a_t \bmv+b_t \bmw_t+c_t\bmxi_t, \quad b_t \bmw_t=b_{t0}\bmxi+b_{t1}\bmxi_1+\cdots+b_{tt-1}\bmxi_{t-1},
\end{align}
where $\|\bmw_t\|_2=1$.  
\end{claim}

\begin{proof}[Proof of Claim \ref{c:expression}]
The Claim  \ref{c:expression} for $t=1$ follows from \eqref{e:y1c}.  In the following, assuming Claim \ref{c:expression} holds for $t$, we prove it for $t+1$.

Let $\bmv^{\otimes (k-1)}, \bmtau_0, \bmtau_1, \cdots, \bmtau_t$ be an orthogonal base for $\bmv^{\otimes (k-1)}, \bmy_0^{\otimes (k-1)}, \bmy_1^{\otimes (k-1)}, \cdots, \bmy_t^{\otimes (k-1)}$, obtained by the Gram-Schmidt orthonormalization procedure. More precisely, given those tensors $\bmv^{\otimes (k-1)}, \bmtau_0, \bmtau_1, \cdots, \bmtau_{t-1}$, we denote 
\begin{align}\begin{split}
&b_{(t+1) 0}=\langle\bmy_t^{\otimes (k-1)},\bmv^{\otimes (k-1)}\rangle,  \quad
c_{t+1}=\langle\bmy_t^{\otimes (k-1)},  \bm\tau_{t}\rangle,\\
&
b_{(t+1)(s+1)}=\langle\bmy_t^{\otimes (k-1)},  \bm\tau_{s}\rangle, \quad
0\leq s\leq t-1.
\end{split}\end{align}
then
$b_{(t+1) 0}\bmv^{\otimes (k-1)}+b_{(t+1)1}\bm\tau_0+b_{(t+1)2}\bm\tau_1+\cdots b_{(t+1)t}\bm\tau_{t-1}$ is the projection of $\bmy_t^{\otimes (k-1)}$ on the span of  $\bmv^{\otimes (k-1)}, \bmy_0^{\otimes (k-1)}, \bmy_1^{\otimes (k-1)}, \cdots, \bmy_{t-1}^{\otimes (k-1)}$.
With those notations, we can write $\bmy_t^{\otimes (k-1)}$ as
\begin{align}\label{e:bbcoeff}
\bmy_t^{\otimes (k-1)}
=b_{(t+1) 0}\bmv^{\otimes (k-1)}+b_{(t+1)1}\bm\tau_0+b_{(t+1)2}\bm\tau_1+\cdots b_{(t+1)t}\bm\tau_{t-1}+c_{t+1} \bm\tau_{t},
\end{align}

 Using \eqref{e:btforms} and \eqref{e:bbcoeff}, 
 we notice that 
 \begin{align}
 \langle \beta\bmv^{\otimes k-1},\bmy_t^{\otimes (k-1)}\rangle
 =\beta(a_t+b_t\langle \bmw_t, \bmv\rangle+c_t\langle \bmxi_t, \bmv\rangle )^{k-1}\bmv,
 \end{align}
 and the iteration \eqref{e:recur} implies that
\begin{align}\label{e:yt+1}
\bmy_{t+1}=\beta(a_t+b_t\langle \bmw_t, \bmv\rangle+c_t\langle \bmxi_t, \bmv\rangle )^{k-1}\bmv+ b_{t+1} \bmw_{t+1}+c_{t+1}\bmZ[\bmtau_t],
\end{align}
where
\begin{align}\begin{split}
b_{t+1} \bmw_{t+1}
&=\bmZ[b_{(t+1) 0}\bmv^{\otimes (k-1)}+b_{(t+1)1}\bm\tau_0+b_{(t+1)2}\bm\tau_1+\cdots b_{(t+1)t}\bm\tau_{t-1}]\\
&=b_{(t+1) 0}\bmxi+b_{(t+1)1}\bm\xi_1+b_{(t+1)2}\bmxi_2+\cdots b_{(t+1)t}\bmxi_{t}.
\end{split}\end{align}
Since $\bmtau_t$ is orthogonal to $\bmv^{\otimes (k-1)}, \bmtau_0, \bmtau_1, \cdots, \bmtau_{t-1}$, Lemma \ref{l:decompose} implies that conditioning on $\bmxi=\bmZ[\bmv^{\otimes(k-1)}]$ and $\bmxi_{s+1}=\bmZ[\bmtau_s]$ for $0\leq s\leq t-1$, $\bmxi_{t+1}=\bmZ[\bmtau_t]$ is an independent Gaussian vector, with each entry $\cN(0,1/n)$. The above discussion gives us that
\begin{align}\label{e:at+1s}
\bmy_{t+1}=a_{t+1}\bmv+ b_{t+1} \bmw_{t+1}+c_{t+1}\bmxi_{t+1},\quad a_{t+1}=\beta(a_t+b_t\langle \bmw_t, \bmv\rangle+c_t\langle \bmxi_t, \bmv\rangle )^{k-1}.
\end{align}
In this way, for any $t\geq 0$, $\bmy_t$ is given in the form \eqref{e:btforms}.
\end{proof}

In the following, We study the case that   $\langle \bmu, \bmv\rangle>0$. The case $\langle \bmu, \bmv\rangle<0$ can be proven in exactly the same way, by simply changing $(\beta, \bmv)$ with $((-1)^{k}\beta, -\bmv)$. We prove by induction

\begin{claim}\label{c:coebound} 
For any fixed time $t$, with probability at least $1-\OO(e^{-c(\log N)^2})$ the following holds: for any $s\leq t$,
\begin{align}\begin{split}\label{e:ind}
&|a_s|\gtrsim |\beta|(|b_{s0}|+|b_{s1}|+\cdots+|b_{s(s-1)}|),\\
&|a_s|\gtrsim n^\varepsilon \max\{\bm1(k\geq 3)|c_s/\beta^{1/(k-2)}|, |c_s/\sqrt n|\}.
\end{split}\end{align}
and
\begin{align}\begin{split}\label{e:xibb}
&\|\bm\xi\| , \|\bm\xi_s\|_2 = 1+\OO(\log n/\sqrt n), \quad  |\langle \bmv, \bm\xi\rangle|,|\langle \bma, \bm\xi\rangle|, |\langle \bma, \bm\xi_s\rangle|,\\
&\|{\rm{Proj}}_{{\rm Span}\{\bmv, \bmxi,\bm\xi_1,\cdots, \cdots, \bmxi_{s-1}\}}(\bmxi_s)\|_2 \lesssim \log n/\sqrt n.
\end{split}\end{align}
\end{claim}

\begin{proof}[Proof of Claim \ref{c:coebound}]
From \eqref{e:y1c},  $\bmy_1=\beta\langle\bmu,\bmv\rangle^{k-1}\bmv
+\langle \bmu, \bmv\rangle^{k-1}\bmxi
+\sqrt{1-\langle \bmu, \bmv\rangle^{2(k-1)}}\bmxi_1$. We have $a_1=\beta\langle\bmu,\bmv\rangle^{k-1}$, $b_{10}=\langle \bmu, \bmv\rangle^{k-1}$, $b_1\bmw_1=\langle \bmu, \bmv\rangle^{k-1}\bm\xi$ and $c_1=
\sqrt{1-\langle \bmu, \bmv\rangle^{2(k-1)}}$.
Since $\bm\xi$ is a Gaussian vector with each entry mean zero and variance $1/n$, the concentration for chi-square distribution implies that 
\begin{align}
\|\bm\xi\|_2=\sqrt{\sum_{i=1}^n \bm\xi(i)^2}=1+\OO(\log n/\sqrt n)
\end{align} 
with probability $1-e^{c(\log n)^2}$.
We can check that $|a_1|=|\beta b_{10}|$,  $|\beta ^{1/{(k-2)}}a_1|=|\beta \langle \bmu, \bmv\rangle^{k-2}|^{(k-1)/(k-2)}\gtrsim n^{(k-1)\varepsilon/(k-2)}\geq n^\varepsilon|c_1|$, and $|\sqrt n a_1|=|\beta\langle \bmu, \bmv\rangle^{k-2}| |\sqrt n \langle \bmu, \bmv\rangle |\gtrsim n^\varepsilon\geq n^\varepsilon|c_1|$. Moreover, conditioning on  $\bmZ[\bmv^{\otimes(k-1)}]=\bm\xi$, Lemma \ref{l:decompose} implies that $\bmxi_1=\bmZ[\bmtau_{0}]$  is an independent Gaussian random vector with each entry $\cN(0,1/n)$. By the standard concentration inequality, it holds that  with probability $1-e^{c(\log n)^2}$, $\|\bm\xi_1\|_2=1+\OO(\log n/\sqrt n)$, $|\langle \bma, \bm\xi_1\rangle|$ and the projection of $\bm\xi_1$ on the span of $\{\bmv, \bmxi\}$ is bounded by $\log n/\sqrt n$. So far we have proved that \eqref{e:ind} and \eqref{e:xibb} for $t=1$.

In the following, we assume that \eqref{e:ind} holds for $t$, and prove it for $t+1$. We recall from \eqref{e:btforms} and \eqref{e:at+1s} that
\begin{align}\label{e:tt1}
a_{t+1}=\beta(a_t+b_t\langle \bmw_t, \bmv\rangle+c_t\langle \bmxi_t, \bmv\rangle )^{k-1}, \quad b_t \bmw_t=b_{t0}\bmxi+b_{t1}\bmxi_1+\cdots+b_{tt-1}\bmxi_{t-1}
\end{align}
By our induction hypothesis, we have that
\begin{align}\label{e:tt2}
|b_t\langle \bmw_t, \bmv\rangle|\lesssim |b_{t0}\langle \bmxi, \bmv\rangle|+|b_{t1}\langle \bmxi_1, \bmv\rangle|+\cdots+|b_{t(t-1)}\langle \bmxi_{t-1}, \bmv\rangle|
\lesssim (\log n /\sqrt{n})|a_t|/|\beta|,
\end{align}
and
\begin{align}\label{e:tt3}
|c_t\langle \bmxi_t, \bmv\rangle|\lesssim (\log n/ \sqrt n)|c_t|\lesssim (\log n)|a_t|/n^\varepsilon.
\end{align}
It follows from plugging \eqref{e:tt2} and \eqref{e:tt3} into \eqref{e:tt1}, we get
\begin{align}\label{e:abb}
a_{t+1}=\beta(a_t+\OO(\log n|a_t|/n^\varepsilon))^{k-1}= (1+\OO(\log n/n^\varepsilon)) \beta a_t^{k-1}.
\end{align}

We recall from \eqref{e:bbcoeff}, the coefficients $b_{(t+1) 0}, b_{(t+1)1}, \cdots, b_{(t+1)t}$ are determined from the projection of $\bmy_t^{\otimes (k-1)}$ on $\bmv^{\otimes (k-1)}, \bmtau_0, \bmtau_1, \cdots, \bmtau_{t-1}$
\begin{align}
\bmy_t^{\otimes (k-1)}
=b_{(t+1) 0}\bmv^{\otimes (k-1)}+b_{(t+1)1}\bm\tau_0+b_{(t+1)2}\bm\tau_1+\cdots b_{(t+1)t}\bm\tau_{t-1}+c_{t+1} \bm\tau_{t}.
\end{align}
We also recall that $\bmv^{\otimes (k-1)}, \bmtau_0, \bmtau_1, \cdots, \bmtau_{t-1}$ are obtained from $\bmv^{\otimes (k-1)}, \bmy_0^{\otimes (k-1)}, \bmy_1^{\otimes (k-1)}, \cdots, \bmy_{t-1}^{\otimes (k-1)}$
by the Gram-Schmidt orthonormalization procedure. So we have that the span of  vectors (viewed as vectors)  $\bmv^{\otimes (k-1)}, \bmtau_0, \bmtau_1, \cdots, \bmtau_{t-1}$ is the same as the span of tensors $\bmv^{\otimes (k-1)}, \bmy_0^{\otimes (k-1)}, \bmy_1^{\otimes (k-1)}, \cdots, \bmy_{t-1}^{\otimes (k-1)}$, which is contained in the span of $\{\bmv, \bmw_t,\bmy_0, \cdots, \bmy_{t-1}\}^{\otimes (k-1)}$. Moreover from the relation \eqref{e:btforms}, one can see that the span of $\{\bmv, \bmw_t,\bmy_0, \cdots, \bmy_{t-1}\}$ is the same as the span of
$\{\bmv, \bmxi,\bmxi_1, \cdots, \bmxi_{t-1}\}$.
It follows that
\begin{align}\begin{split}\label{e:sumb}
&\phantom{{}={}}\sqrt{b_{(t+1) 0}^2+b_{(t+1)1}^2+b_{(t+1)2}^2+\cdots +b_{(t+1)t}^2}\\
&=\|{\rm{Proj}}_{{\rm Span}\{\bmv^{\otimes (k-1)}, \bmtau_0, \bmtau_1, \cdots, \bmtau_{t-1}\}}(a_t \bmv+b_t \bmw_t+c_t\bmxi_t)^{\otimes (k-1)}\|_2 \\
&\leq\|{\rm{Proj}}_{{\rm Span}\{\bmv, \bmw_t,\bmy_0, \cdots, \bmy_{t-1}\}^{\otimes (k-1)}}(a_t \bmv+b_t \bmw_t+c_t\bmxi_t)^{\otimes (k-1)}\|_2\\
&\leq\|{\rm{Proj}}_{{\rm Span}\{\bmv, \bmw_t,\bmy_0, \cdots, \bmy_{t-1}\}}(a_t \bmv+b_t \bmw_t+c_t\bmxi_t)\|_2^{k-1}\\
&=\|a_t \bmv+b_t \bmw_t+c_t{\rm{Proj}}_{{\rm Span}\{\bmv, \bmxi,\bmxi_1, \cdots, \bmxi_{t-1}\}}(\bmxi_t)\|_2^{k-1}\\
&\lesssim \left(|a_t|+|b_t|+\frac{\log n|c_t|}{\sqrt n}\right)^{k-1}\lesssim |a_t|^{k-1}\lesssim |a_{t+1}|/|\beta|,
\end{split}\end{align}
where in the last line we used our induction hypothesis that $\|{\rm{Proj}}_{{\rm Span}\{\bmv, \bmxi,\bmxi_1, \cdots, \bmxi_{t-1}\}}(\bmxi_t)\|_2\lesssim \log n/\sqrt n$.

Finally we estimate $c_{t+1}$. We recall from \eqref{e:bbcoeff}, the coefficient $c_{t+1}$ is the remainder of $\bmy_t^{\otimes (k-1)}$ after projecting on $\bmv^{\otimes (k-1)}, \bmtau_0, \bmtau_1, \cdots, \bmtau_{t-1}$. It is bounded by the remainder of $\bmy_t^{\otimes (k-1)}$ after projecting on $\bmv^{\otimes (k-1)}$,
\begin{align}
|c_{t+1}|\leq \|\bmy_t^{\otimes (k-1)}-a_t^{k-1}\bmv^{\otimes (k-1)}\|_2
=\|(a_t \bmv+b_t \bmw_t+c_t\bmxi_t)^{\otimes (k-1)}-a_t^{k-1}\bmv^{\otimes (k-1)}\|_2
.
\end{align}
The difference $(a_t \bmv+b_t \bmw_t+c_t\bmxi_t)^{\otimes (k-1)}-a_t^{k-1}\bmv^{\otimes (k-1)}$ is a sum of terms in the following form,
\begin{align}\label{e:owe}
\bm\eta_1\otimes \bm\eta_2\otimes \cdots\otimes \bm\eta_{k-1},
\end{align}
where vectors $\bm\eta_1, \bm\eta_2, \cdots, \bm\eta_{k-1}\in \{a_t\bmv, b_t\bmw_t+c_t\bm\xi_t\}$, and at least one of them is $b_t\bmw_t+c_t\bm\xi_t$. We notice that by our induction hypothesis, $\|b_t\bmw_t+c_t\bm\xi_t\|_2\lesssim |b_t|\|\bmw_t\|_2+|c_t|\|\bm\xi_t\|_2\lesssim |b_t|+|c_t|$. For the $L_2$ norm of \eqref{e:owe}, each copy of  $a_t\bmv$ contributes $a_t$ and each copy of $b_t\bmw_t+c_t\bm\xi_t$ contributes a factor $|b_t|+|c_t|$. We conclude that
\begin{align}\label{e:ctbound}
|c_{t+1}|\leq
\|(a_t \bmv+b_t \bmw_t+c_t\bmxi_t)^{\otimes (k-1)}-a_t^{k-1}\bmv^{\otimes (k-1)}\|_2
\lesssim \sum_{r=1}^{k-1}|a_t|^{k-1-r}(|b_t|+|c_t|)^r.
\end{align}
Combining the above estimate with \eqref{e:abb} that $|a_{t+1}|\asymp |\beta||a_t|^{k-1}$, we divide both sides of \eqref{e:ctbound} by $|\beta||a_t|^{k-1}$,
\begin{align}\label{e:xtbb}
\frac{|c_{t+1}|}{|a_{t+1}|}
&\lesssim \frac{1}{|\beta|}\sum_{r=1}^{k-1}\left(\frac{|b_t|}{|a_t|}+\frac{|c_t|}{|a_t|}\right)^r
\lesssim
\frac{1}{|\beta|}\sum_{r=1}^{k-1}\left(\frac{1}{|\beta|}+\frac{|c_t|}{|a_t|}\right)^r,
\end{align}
where we used our induction hypothesis that $|a_t|\gtrsim |\beta||b_t|$.
There are three cases:
\begin{enumerate}
\item If $|c_t|/|a_t|\geq 1$, then
\begin{align}
\frac{|c_{t+1}|}{|a_{t+1}|}
\lesssim
\frac{1}{|\beta|}\sum_{r=1}^{k-1}\left(\frac{1}{|\beta|}+\frac{|c_t|}{|a_t|}\right)^r
\lesssim\frac{1}{|\beta|}\left(\frac{|c_t|}{|a_t|}\right)^{k-1}.
\end{align}
If $k=2$, then our assumption $|\beta\langle \bmu, \bmv\rangle^{k-2}|=|\beta|\geq n^{\varepsilon}$, implies that  $|c_{t+1}|/|a_{t+1}|\lesssim (|c_t|/|a_t|)/n^{\varepsilon}$. If $k\geq 2$, by our induction hypothesis $|c_t|/|a_t|\lesssim \beta^{1/(k-2)}/n^\varepsilon$. This implies $(|c_t|/|a_t|)^{k-2}/|\beta|\lesssim 1/n^{\varepsilon}$, and we still get that $|c_{t+1}|/|a_{t+1}|\lesssim (|c_t|/|a_t|)/n^{\varepsilon}$.
\item If $1/|\beta|\lesssim |c_t|/|a_t|\leq 1$, then
\begin{align}
\frac{|c_{t+1}|}{|a_{t+1}|}
\lesssim
\frac{1}{|\beta|}\sum_{r=1}^{k-1}\left(\frac{1}{|\beta|}+\frac{|c_t|}{|a_t|}\right)^r
\lesssim\frac{1}{|\beta|}\left(\frac{|c_t|}{|a_t|}\right)\lesssim \frac{1}{n^\varepsilon}\left(\frac{|c_t|}{|a_t|}\right),
\end{align}
where we used that $|\beta|\geq |\beta\langle \bmu, \bmv\rangle^{k-2}|\geq n^{\varepsilon}$.
\item Finally for $ |c_t|/|a_t|\lesssim 1/|\beta|$, we will have
\begin{align}
\frac{|c_{t+1}|}{|a_{t+1}|}
\lesssim
\frac{1}{|\beta|}\sum_{r=1}^{k-1}\left(\frac{1}{|\beta|}+\frac{|c_t|}{|a_t|}\right)^r
\lesssim\frac{1}{|\beta|}\left(\frac{1}{|\beta|}\right)\lesssim \frac{1}{|\beta|^2}.
\end{align}
\end{enumerate}
In all these cases if $|c_{t}|/|a_{t}|\lesssim\min\{\sqrt n, \bm1(k\geq 3)|\beta|^{1/(k-2)}\}/n^{\varepsilon}$, we have $|c_{t+1}|/|a_{t+1}|\lesssim\min\{\sqrt n, \bm1(k\geq 3)|\beta|^{1/(k-2)}\}/n^{\varepsilon}$.
This finishes the proof of the induction \eqref{e:ind}.

For \eqref{e:xibb}, since $\bmtau_t$ is orthogonal to $\bmv^{\otimes (k-1)}, \bmtau_0, \bmtau_1, \cdots, \bmtau_{t-1}$, Lemma \ref{l:decompose} implies that conditioning on $\bmxi=\bmZ[\bmv^{\otimes(k-1)}]$ and $\bmxi_{s+1}=\bmZ[\bmtau_s]$ for $0\leq s\leq t-1$, $\bmxi_{t+1}=\bmZ[\bmtau_t]$ is an independent Gaussian vector, with each entry $\cN(0,1/n)$. By the standard concentration inequality, it holds that  with probability $1-e^{c(\log n)^2}$, $\|\bm\xi_{t+1}\|_2=1+\OO(\log n/\sqrt n)$, $|\langle \bma, \bm\xi_{t+1}\rangle|$ and the projection of $\bm\xi_{t+1}$ on the span of $\{\bmv, \bmxi, \bmxi_1,\cdots, \bmxi_t\}$ is bounded by $\log n/\sqrt n$. This finishes the proof of the induction \eqref{e:xibb}.
\end{proof}

Next, using \eqref{e:ind} and \eqref{e:xibb} in Claim \ref{c:coebound} as input, we prove that for
\begin{align}\label{e:tbbd}
t\geq 1+\frac{1}{\varepsilon}\left(\frac{1}{2}+\frac{2\log|\beta|}{\log n}\right),
\end{align}
with probability $1-e^{c(\log n)^2}$ we have
\begin{align}\label{e:btform}
\bmy_t=a_t \bmv+b_{t0}\bmxi+b_{t1}\bmxi_1+\cdots+b_{tt-1}\bmxi_{t-1}+c_t\bmxi_t,
\end{align}
such that
\begin{align}\label{e:newbound}
b_{t0}=\frac{a_t}{\beta}+\OO\left(\frac{\log n |a_t|}{|\beta|^2\sqrt n}\right)
\quad |b_{t1}|, |b_{t2}|,\cdots, |b_{t(t-1)}|\lesssim \frac{(\log n)^{1/2}|a_{t}|}{|\beta|^{3/2}n^{1/4}},\quad |c_t|\lesssim |a_t|/\beta^2.
\end{align}

Let $x_t=|c_t/a_t|\ll |\beta|^{1/(k-2)}$, then \eqref{e:xtbb} implies
\begin{align}
x_{t+1}\lesssim \frac{1}{|\beta|}\sum_{r=1}^{k-1}\left(\frac{1}{|\beta|}+x_t\right)^r,
\end{align}
from the discussion after \eqref{e:xtbb}, we have that either
$x_{t+1}\lesssim 1/\beta^2$, or $x_{t+1}\lesssim x_t/n^\varepsilon$.
Since $x_1=|c_1/a_1|\lesssim n^{1/2-\varepsilon}$, we conclude that it holds
\begin{align}\label{e:cbbs}
x_t=|c_t/a_t|\lesssim 1/\beta^2,\quad \text{when }t\geq \frac{1}{\varepsilon}\left(\frac{1}{2}+\frac{2\log|\beta|}{\log n}\right).
\end{align}

To derive the upper bound of $b_{t1}, b_{t2},\cdots, b_{t(t-1)}$, we use \eqref{e:sumb}.
\begin{align}\begin{split}\label{e:sumb2}
&\phantom{{}={}}b_{(t+1) 0}^2+b_{(t+1)1}^2+b_{(t+1)2}^2+\cdots +b_{(t+1)t}^2\\
&\leq
\|a_t \bmv+b_t \bmw_t+c_t{\rm{Proj}}_{{\rm Span}\{\bmv, \bmxi,\bmxi_1, \cdots, \bmxi_{t-1}\}}(\bmxi_t)\|_2^{2(k-1)}\\
&=\left(a_t^2+\OO\left(|a_t|\left(|b_t|+|c_t|\right)\frac{\log n}{\sqrt n}+\left(|b_t|+|c_t|\frac{\log n}{\sqrt n}\right)^2\right)\right)^{k-1},
\end{split}\end{align}
where we used our induction  \eqref{e:xibb} that $|\langle \bmxi, \bmv\rangle|, |\langle \bmxi_1, \bmv\rangle|,\cdots, |\langle \bmxi_t, \bmv\rangle|\lesssim \log n/\sqrt n$ and
the projection $\|{\rm{Proj}}_{{\rm Span}}\{\bmv, \bmxi,\bmxi_1, \cdots, \bmxi_{t-1}\}(\bmxi_t)\|_2\lesssim \log n/\sqrt n$.
Moreover, the first term $b_{(t+1) 0}$ is the projection of $\bmy_t^{\otimes (k-1)}$ on $\bmv^{\otimes (k-1)}$,
\begin{align}\label{e:sumb3}
b_{(t+1) 0}=\langle a_t \bmv+b_t \bmw_t+c_t\bmxi_{t},\bmv\rangle^{k-1}
=\left(a_t+ \OO\left(\frac{\log n(|b_t|+|c_t|)}{\sqrt n}\right)\right)^{k-1},
\end{align}
where we used \eqref{e:xibb} that $|\langle \bmxi, \bmv\rangle|, |\langle \bmxi_1, \bmv\rangle|,\cdots, |\langle \bmxi_t, \bmv\rangle|\lesssim \log n/\sqrt n$.
Now we can take difference of \eqref{e:sumb2} and \eqref{e:sumb3}, and use that $|b_t|\lesssim |a_t|/|\beta|$ from \eqref{e:ind} and $|c_t|\lesssim |a_t|/|\beta|$ from \eqref{e:cbbs},
\begin{align}\label{e:btta}
b_{(t+1)0}=a_t^{k-1}+\OO\left(|a_t|^{k-1}\frac{\log n}{|\beta|\sqrt n}\right), \quad b_{(t+1)1}^2+b_{(t+1)2}^2+\cdots +b_{(t+1)t}^2
\lesssim a_t^{2(k-1)}\frac{\log n}{|\beta|\sqrt n}.
\end{align}
From \eqref{e:tt1} and \eqref{e:abb}, we have that
\begin{align}\label{e:at+1}
a_{t+1}=\beta b_{(t+1) 0}\asymp \beta a_t^{k-1}.
\end{align}
Using the above relation, we can simplify \eqref{e:btta} as
\begin{align}
b_{(t+1)0}=\frac{a_{t+1}}{\beta}+\OO\left(\frac{\log n|a_{t+1}|}{|\beta|^2\sqrt n}\right),
\quad |b_{(t+1)1}|, |b_{(t+1)2}|, \cdots |b_{(t+1)t}|
\lesssim \frac{(\log n)^{1/2}|a_{t+1}|}{|\beta|^{3/2}n^{1/4}}.
\end{align}
This finishes the proof of \eqref{e:newbound}.

With the expression \eqref{e:newbound}, we can process to prove our main results \eqref{e:clt} and \eqref{e:betaclt}.
Thanks to \eqref{e:xibb}, for $t$ satisfies \eqref{e:tbbd}, we have that with probability at least $1-\OO(e^{-c(\log N)^2})$
\begin{align}\label{e:ytnorm}
\|\bmy_t\|_2^2=a_t^2\left(1+\frac{1}{\beta^2}+\frac{2\langle \bmv, \bmxi\rangle}{\beta}+\OO\left(\frac{\log n}{\beta^2\sqrt n}+\frac{(\log n)^{3/2}}{|\beta|^{3/2}n^{3/4}}+\frac{1}{\beta^4}\right)\right).
\end{align}
By rearranging it we get
\begin{align}\label{e:ytnorminv}
1/\|\bmy_t\|_2
=\frac{1}{|a_t|}\left(1-\frac{1}{2\beta^2}-\frac{\langle \bmv, \bmxi\rangle}{\beta}+\OO\left(\frac{\log n}{\beta^2\sqrt n}+\frac{(\log n)^{3/2}}{|\beta|^{3/2}n^{3/4}}+\frac{1}{\beta^4}\right)\right).
\end{align}
We can take the inner product $\langle \bma, \bmy_t\rangle$ using \eqref{e:btform} and \eqref{e:newbound}, and multiply \eqref{e:ytnorminv}
\begin{align}\begin{split}
\langle \bma, \bmu_t\rangle
=\frac{\langle \bma, \bmy_t\rangle}{\|\bmy_t\|_2}
&=\sgn(a_t)\left(\left(1-\frac{1}{2\beta^2}\right)\langle\bma, \bmv\rangle+\frac{\langle \bma, \bmxi\rangle-\langle \bma, \bmv\rangle\langle \bmv, \bmxi\rangle}{\beta}\right)\\
&+\OO_\bP\left(\frac{\log n}{\beta^2\sqrt n}+\frac{(\log n)^{3/2}}{|\beta|^{3/2}n^{3/4}}+\frac{|\langle \bma, \bmv\rangle|}{\beta^4}\right),
\end{split}\end{align}
where we used \eqref{e:xibb} that with high probability $|\langle \bma, \bm\xi\rangle|$, $|\langle \bma, \bm\xi_s\rangle|$ for $1\leq s\leq t$ are bounded by $\log n/\sqrt n$. This finishes the proof of \eqref{e:clt}.
For $\widehat\beta$ in \eqref{e:betaclt}, we have
\begin{align}\label{e:newbeta}
\bmX[\bmu_t^{\otimes k}]
=\frac{\bmX[\bmy_t^{\otimes k}]}{\|\bmy_t\|_2^k}
=\frac{\langle \bmy_t, \bmX[\bmy_t^{\otimes(k-1)}]\rangle}{\|\bmy_t\|_2^k}
=\frac{\langle \bmy_t, \bmy_{t+1}\rangle}{\|\bmy_t\|_2^k}
.
\end{align}

Thanks to \eqref{e:cbbs}, \eqref{e:at+1s} and \eqref{e:xibb}, for $t$ satisfies \eqref{e:tbbd},  with probability at least $1-\OO(e^{-c(\log N)^2})$,  we have
\begin{align}
\bmy_{t+1}=a_{t+1} \bmv+b_{(t+1)0}\bmxi+b_{(t+1)1}\bmxi_1+\cdots+b_{(t+1)t}\bmxi_{t}+c_{t+1}\bmxi_{t+1},
\end{align}
where $|c_{t+1}|\lesssim |a_t|^{k-1}/\beta^2$,
\begin{align}\begin{split}\label{e:atrecur}
a_{t+1}
&=\beta (a_t+b_{t0}\langle \bm\xi, \bmv\rangle+b_{t1}\langle \bm\xi_1, \bmv\rangle+\cdots+
b_{t(t-1)}\langle \bm\xi_{t-1}, \bmv\rangle+c_t\langle \bm\xi_t,\bmv\rangle)^{k-1}\\
&=\beta a_t^{k-1}\left(1+\frac{\langle\bmxi, \bmv \rangle}{\beta}+\OO\left(\frac{\log n}{\beta^2\sqrt n}+\frac{(\log n)^{3/2}}{|\beta|^{3/2}n^{3/4}}\right)\right)^{k-1},\quad
\end{split}\end{align}
and
\begin{align}
b_{(t+1)0}=a_t^{k-1}\left(1+\OO\left(\frac{\log n}{|\beta|\sqrt n}\right)\right), \quad |b_{(t+1)1}|, |b_{(t+1)2}|+\cdots +|b_{(t+1)t}|
\lesssim a_t^{k-1}\frac{(\log n)^{1/2}}{|\beta|^{1/2}n^{1/4}}.
\end{align}
From the discussion above, combining with \eqref{e:btform} and \eqref{e:newbound} with straightforward computation, we have
\begin{align}\label{e:cross}
\langle \bmy_t, \bmy_{t+1}\rangle=
\beta a_{t}^k\left(1+\frac{1}{\beta^2}+\frac{(k+1)\langle \bm\xi, \bmv\rangle}{\beta}+\OO\left(\frac{\log n}{\beta^2 \sqrt n}+\frac{(\log n)^{3/2}}{|\beta|^{3/2}n^{3/4}}\right)\right).
\end{align}
By plugging \eqref{e:ytnorminv} and \eqref{e:cross} into \eqref{e:newbeta}, we get
\begin{align}
\bmX[\bmu_t^{\otimes k}]=\sgn(a_t)^k\left(\beta+\langle \bm\xi, \bmv\rangle -\frac{k/2-1}{\beta}\right)+\OO\left(\frac{\log n}{|\beta| \sqrt n}+\frac{(\log n)^{3/2}}{|\beta|^{1/2}n^{3/4}}+\frac{1}{|\beta|^3}\right)
\end{align}
Since by our assumption, in Case 1 we have that $\beta>0$. Thanks to \eqref{e:atrecur} $a_{t+1}=\beta a_t^{k-1}(1+\oo(1))$, especially $a_{t+1}$ and $a_t$ are of the same sign. In the case $\langle \bmu, \bmv\rangle>0$, we have $a_1=\beta\langle \bmu, \bmv\rangle^{k-1}>0$. We conclude that $a_t>0$. Therefore $\sgn(\bmX[\bmu_t^{\otimes k}])=\sgn(a_t)^k=+$, and it follows that
\begin{align}
\bmX[\bmu_t^{\otimes k}]
=\beta+\langle \bm\xi, \bmv\rangle -\frac{k/2-1}{\beta}+\OO\left(\frac{\log n}{|\beta| \sqrt n}+\frac{(\log n)^{3/2}}{|\beta|^{1/2}n^{3/4}}+\frac{1}{|\beta|^3}\right)
\end{align}
This finishes the proof of \eqref{e:betaclt}.
The Cases \ref{c:case2}, \ref{c:case3}, \ref{c:case4} follow by simply changing $(\beta, \bmv)$
in the righthand side of \eqref{e:clt} and \eqref{e:betaclt} to the corresponding limit.
\end{proof}

\begin{proof}[Proof of Theorem \ref{t:diverge}]

We use the same notations as in the proof of Theorem \ref{t:main}.
If $|\beta|\geq n^{\varepsilon}$ and $|\beta\langle \bmu, \bmv\rangle^{k-2}|\leq n^{-\varepsilon}$, then we first prove by induction that
 for any fixed time $t$, with probability at least $1-\OO(e^{-c(\log N)^2})$ the following holds: for any $s\leq t$,
\begin{align}\begin{split}\label{e:ind2}
&|b_{s0}|, |b_{s1}|,\cdots, |b_{s(s-1)}|\lesssim \max\{|c_s|/|\beta|^{(k-1)/(k-2)}, (\log n)^{k-1}|c_s|/n^{(k-1)/2}\},\\
&|c_s|\geq n^\varepsilon\beta^{1/(k-2)}|a_s|,
\end{split}\end{align}
and
\begin{align}\begin{split}\label{e:xibb2}
&\|\bm\xi\| , \|\bm\xi_s\|_2 = 1+\OO(\log n/\sqrt n), \\
&|\langle \bmv, \bm\xi\rangle|, \|{\rm{Proj}}_{{\rm Span}\{\bmv, \bmxi,\bm\xi_1,\cdots, \cdots, \bmxi_{s-1}\}}(\bmxi_s)\|_2 \lesssim \log n/\sqrt n.
\end{split}\end{align}

From \eqref{e:y1c},  $a_1=\beta\langle\bmu,\bmv\rangle^{k-1}$, $b_{10}=\langle \bmu, \bmv\rangle^{k-1}$ and $c_1=
\sqrt{1-\langle \bmu, \bmv\rangle^{2(k-1)}}$.
Since $|\beta|\geq n^\varepsilon$ and $|\beta\langle \bmu, \bmv\rangle^{k-2}|\leq n^{-\varepsilon}$, we have that $|\langle \bmu, \bmv\rangle|\leq n^{-2\varepsilon/(k-2)}\ll1$ and therefore $|c_1|\asymp 1$.
We can check that $|\beta ^{1/{(k-2)}}a_1|=|\beta\langle \bmu, \bmv\rangle^{k-2}|^{(k-1)/(k-2)}\leq n^{-\varepsilon}\lesssim n^{-\varepsilon}|c_1|$ and $|b_{10}|=|a_1/\beta|\lesssim  n^{-\varepsilon}|c_1/\beta^{(k-1)/(k-2)}|$. Moreover, conditioning on  $\bmZ[\bmv^{\otimes(k-1)}]=\bm\xi$, Lemma \ref{l:decompose} implies that $\bmxi_1=\bmZ[\bmtau_{0}]$  is an independent Gaussian random vector with each entry $\cN(0,1/n)$. By the standard concentration inequality, it holds that  with probability $1-e^{c(\log n)^2}$, $\|\bm\xi_1\|_2=1+\OO(\log n/\sqrt n)$, and the projection of $\bm\xi_1$ on the span of $\{\bmv, \bmxi\}$ is bounded by $\log n/\sqrt n$. So far we have proved \eqref{e:ind2} and \eqref{e:xibb2} for $t=1$.

In the following, assuming the statements \eqref{e:ind2} and \eqref{e:xibb2} hold for $t$, we prove them for $t+1$.
From \eqref{e:tt1}, using \eqref{e:tt1} and \eqref{e:tt2}, we have
\begin{align}\begin{split}\label{e:abound}
&\phantom{{}={}}|a_{t+1}|
=\left|\beta(a_t+b_t\langle \bmw_t, \bmv\rangle+c_t\langle \bmxi_t, \bmv\rangle )^{k-1}\right|\\
&\lesssim |\beta|\left(|a_t|+\frac{\log n (|b_{t0}|+|b_{t1}|+\cdots +|b_{t(t-1)}|)}{\sqrt n}+\frac{\log n|c_t|}{\sqrt n}\right)^{k-1}\\
&\lesssim|\beta|\left(|a_t|+\frac{\log n|c_t|}{\sqrt n}\right)^{k-1}
\lesssim |\beta|\left(\frac{|c_t|}{n^\varepsilon |\beta|^{1/(k-2)}}+\frac{\log n|c_t|}{\sqrt n}\right)^{k-1}\\
&\lesssim |\beta||c_t|^{k-1}\left(\frac{1}{n^{\varepsilon}|\beta|^{1/(k-2)}}+\frac{\log n}{\sqrt n}\right)^{k-1}
\lesssim \frac{|c_t|^{k-1}}{n^{\varepsilon}|\beta|^{1/(k-2)}},
\end{split}\end{align}
where in the third line we used our induction hypothesis that $|b_{t0}|+|b_{t1}|+\cdots +|b_{t(t-1)}|\lesssim |c_t|$, and $n^{-\varepsilon}\geq |\beta||\langle\bmu,\bmv\rangle|^{k-2}\gtrsim |\beta|/n^{(k-2)/2}$.

For $b_{(t+1) 0},b_{(t+1)1},\cdots , b_{(t+1)t}$, from \eqref{e:sumb} we have
\begin{align}\begin{split}\label{e:bbound}
&\phantom{{}={}}\sqrt{b_{(t+1) 0}^2+b_{(t+1)1}^2+b_{(t+1)2}^2+\cdots +b_{(t+1)t}^2}\lesssim
\left(|a_t|+|b_t|+\frac{\log n |c_t|}{\sqrt n}\right)^{k-1}\\
&\lesssim
\left(|a_t|+\frac{\log n |c_t|}{\sqrt n}\right)^{k-1}
\lesssim \left(\frac{|c_t|}{n^\varepsilon |\beta|^{1/(k-2)}}+\frac{\log n|c_t|}{\sqrt n}\right)^{k-1}\\
&\lesssim |c_t|^{k-1}\left(\frac{1}{n^{\varepsilon}|\beta|^{1/(k-2)}}+\frac{\log n}{\sqrt n}\right)^{k-1}.
\end{split}\end{align}

Finally we estimate $c_{t+1}$. We recall from \eqref{e:bbcoeff}, the coefficient $c_{t+1}$ is the remainder of $\bmy_t^{\otimes (k-1)}$ after projecting on $\bmv^{\otimes (k-1)}, \bmtau_0, \bmtau_1, \cdots, \bmtau_{t-1}$.
We have the following lower bound for $c_{t+1}$
\begin{align}\begin{split}\label{e:cbound}
|c_{t+1}|^2
&=\|(a_t \bmv+b_t \bmw_t+c_t\bmxi_t)^{\otimes (k-1)}\|^2_2-(b_{(t+1) 0}^2+b_{(t+1)1}^2+b_{(t+1)2}^2+\cdots +b_{(t+1)t}^2)\\
&\geq \|a_t \bmv+b_t \bmw_t+c_t\bmxi_t\|^{2(k-1)}_2-\OO\left(|c_t|^{2(k-1)}\left(\frac{1}{n^{\varepsilon}|\beta|^{1/(k-2)}}+\frac{\log n}{\sqrt n}\right)^{2(k-1)}\right).
\end{split}\end{align}
For the first term on the righthand side of \eqref{e:cbound}, using our induction hypothesis
\eqref{e:ind2} and \eqref{e:xibb2} that $|a_t|\lesssim |c_t|$, we have
\begin{align}\begin{split}\label{e:firstt}
\|a_t \bmv+b_t \bmw_t+c_t\bmxi_t\|^{2}_2
&=a_t^2 +b_t^2+c_t^2\|\bmxi_t\|^{2}_2+
2a_tb_t\langle \bmv, \bmw_t\rangle+2 a_t c_t\langle \bmv, \bm\xi_t\rangle
+2b_tc_t\langle \bmw_t, \bm\xi_t\rangle\\
&=\left(1+\OO\left(\frac{\log n}{\sqrt n}+\frac{1}{n^{2\varepsilon}\beta^{2/(k-2)}}\right)\right)c_t^2.
\end{split}\end{align}
We get the following lower for $c_{t+1}$ by plugging \eqref{e:firstt} into \eqref{e:cbound}, and rearranging
\begin{align}\label{e:cbbdhaha}
|c_{t+1}|\geq \left(1+\OO\left(\frac{\log n}{\sqrt n}+\frac{1}{n^{2\varepsilon}\beta^{2/(k-2)}}\right)\right)|c_t|^{k-1}
\end{align}
The claim that $|b_{(t+1)0}|, |b_{(t+1)1}|,\cdots, |b_{(t+1)t}|\lesssim \max\{|c_{t+1}|/|\beta|^{(k-1)/(k-2)}, (\log n)^{k-1}|c_{t+1}|/n^{(k-1)/2}\}$ follows from combining \eqref{e:bbound} and \eqref{e:cbbdhaha}.
The claim that $|c_{t+1}|\geq n^\varepsilon\beta^{1/(k-2)}|a_{t+1}|$ follows from combining \eqref{e:abound} and \eqref{e:cbbdhaha}.

For \eqref{e:xibb2}, since $\bmtau_t$ is orthogonal to $\bmv^{\otimes (k-1)}, \bmtau_0, \bmtau_1, \cdots, \bmtau_{t-1}$, Lemma \ref{l:decompose} implies that conditioning on $\bmxi=\bmZ[\bmv^{\otimes(k-1)}]$ and $\bmxi_{s+1}=\bmZ[\bmtau_s]$ for $0\leq s\leq t-1$, $\bmxi_{t+1}=\bmZ[\bmtau_t]$ is an independent Gaussian vector, with each entry $\cN(0,1/n)$. By the standard concentration inequality, it holds that  with probability $1-e^{c(\log n)^2}$, $\|\bm\xi_{t+1}\|_2=1+\OO(\log n/\sqrt n)$,  and the projection of $\bm\xi_{t+1}$ on the span of $\{\bmv, \bmxi, \bmxi_1,\cdots, \bmxi_t\}$ is bounded by $\log n/\sqrt n$. This finishes the proof of the induction \eqref{e:xibb2}.

Next, using \eqref{e:ind} and \eqref{e:xibb} as input, we prove that for
\begin{align}\label{e:tbbd2}
t\geq 1+\frac{1}{\varepsilon}\left(\frac{1}{2}-\frac{\log|\beta|}{(k-2)\log n}\right),
\end{align}
we have
\begin{align}\label{e:btform2}
\bmy_t=a_t \bmv+b_{t0}\bmxi+b_{t1}\bmxi_1+\cdots+b_{t(t-1)}\bmxi_{t-1}+c_t\bmxi_t,
\end{align}
such that
\begin{align}\label{e:newbound2}
\quad |a_t|, |b_{t0}|, |b_{t1}|,\cdots, |b_{t(t-1)}|\lesssim |c_t| |\beta|\left(\frac{\log n}{\sqrt n}\right)^{k-1}.
\end{align}

Let $x_t=|a_t/c_t|$, then \eqref{e:ind2} implies that $x_t\leq 1/(n^{\varepsilon}|\beta|^{1/(k-2)})$. By taking the ratio of \eqref{e:abound} and \eqref{e:cbbdhaha}, we get
\begin{align}\label{e:xtit}
x_{t+1}\lesssim |\beta|\left(\frac{\log n}{\sqrt n}+x_t\right)^{k-1}.
\end{align}
there are two cases,
\begin{enumerate}
\item if $\log n/\sqrt n\lesssim x_t\leq 1/(n^{\varepsilon}|\beta|^{1/(k-2)})$, then
\begin{align}
x_{t+1}\lesssim |\beta| x_t^{k-1}
=x_t(|\beta|^{1/(k-2)} x_t)^{k-2}\leq x_t/n^\varepsilon;
\end{align}
\item If $x_t\lesssim \log n/\sqrt n$, then $|x_{t+1}|\lesssim |\beta|(\log n/\sqrt n)^{k-1}$.
\end{enumerate}
Since $x_1=|a_1/c_1|\lesssim 1/(n^{\varepsilon}|\beta|^{1/(k-2)})$, we conclude that
\begin{align}\label{e:cbb}
x_t=|a_t/c_t|\lesssim |\beta|(\log n/\sqrt n)^{k-1},\quad \text{when }t\geq \frac{1}{\varepsilon}\left(\frac{1}{2}-\frac{\log|\beta|}{(k-2)\log n}\right).
\end{align}
In this regime, \eqref{e:bbound} implies that
\begin{align}\begin{split}
&\phantom{{}={}}|b_{(t+1) 0}|, |b_{(t+1)1}|, |b_{(t+1)2}|, \cdots ,|b_{(t+1)t}|\lesssim |\beta||c_t|^{k-1}\left(\frac{|a_t|}{|c_t|}+\frac{\log n}{\sqrt n}\right)^{k-1}\\
&\lesssim |\beta||c_t|^{k-1}\left(\frac{\log n}{\sqrt n}\right)^{k-1}\lesssim |c_{t+1}||\beta|\left(\frac{\log n}{\sqrt n}\right)^{k-1},
\end{split}\end{align}
where we used \eqref{e:cbbdhaha} in the last inequality. This finishes the proof of \eqref{e:newbound2}.
Using \eqref{e:newbound2}, we can compute $\bmu_t$,
\begin{align}
\bmu_t=\frac{\bmy_t}{\|\bmy_t\|}=\frac{\bm\xi_t}{\|\bm\xi_t\|_2}+\OO_\bP\left(|\beta|\left(\frac{\log n}{\sqrt n}\right)^{k-1}\right),
\end{align}
where the error term is a vector of length bounded by $|\beta|(\log n/\sqrt n)^{k-1}$. This finishes the proof of Theorem \ref{t:main}.

\end{proof}

\subsection{Proof of Corollarys \ref{coro1} and  \ref{coro2}}

\begin{proof}[Proof of Corollary \ref{coro1}]
According to the definition of $\bmxi$ in \eqref{e:clt} of Theorem \ref{t:main}, i.e.
$\bmxi=\bmZ[\bmv^{\otimes(k-1)}]$, is an $n$-dim vector, with each entry i.i.d. $\cN(0,1/n)$ Gaussian random variable. 
We  see that
$$
\langle \bm\xi, \bmv\rangle\stackrel{d}{=} \cN\left(0,1/n\right).
 $$
  Especially with high probability we will have that $|\langle \bm\xi, \bmv\rangle|\lesssim \log n/\sqrt n$.
Then we conclude from \eqref{e:betaclt}, with high probability it holds
\begin{align}\label{e:betabb}
\widehat \beta=\beta+\OO\left(\frac{1}{\beta}+\frac{\log n}{\sqrt n}\right).
\end{align} 
With the bound \eqref{e:betabb}, we can replace $\langle \bma, \bmv\rangle/(2\beta^2)$ on the righthand side of \eqref{e:clt} by $\langle \bma, \bmv\rangle/(2\widehat\beta^2)$, which gives an error \begin{align}
\left|\frac{\langle \bma, \bmv\rangle}{2\beta^2}-\frac{\langle \bma, \bmv\rangle}{2\widehat\beta^2}\right|=\OO\left(|\langle \bma, \bmv\rangle|\left(\frac{1}{|\beta|^4}+\frac{\log n}{|\beta|^3\sqrt n}\right)\right).
\end{align}
Combining the above discussion together, 
we can rewrite \eqref{e:clt} as
\begin{align}\begin{split}\label{e:avvbb}
&\phantom{{}={}}\langle \bma, \widehat\bmv\rangle-\left(1-\frac{1}{2\widehat\beta^2}\right)\langle\bma, \bmv\rangle=\frac{\langle \bma, \bmxi\rangle-\langle \bma, \bmv\rangle\langle \bmv, \bmxi\rangle}{\beta}+\OO\left(\frac{\log n}{\beta^2\sqrt n}+\frac{(\log n)^{3/2}}{\beta^{3/2}n^{3/4}}+\frac{|\langle \bma, \bmv\rangle|}{\beta^4}\right)
\end{split}\end{align}
with high probability.

Again thanks to the definition of $\bmxi$ in \eqref{e:clt} of Theorem \ref{t:main}, i.e.
$\bmxi=\bmZ[\bmv^{\otimes(k-1)}]$, is an $n$-dim vector, with each entry i.i.d. $\cN(0,1/n)$ Gaussian random variable, we see that
\begin{align}
\langle \bma, \bmxi\rangle-\langle \bma, \bmv\rangle\langle \bmv, \bmxi\rangle
=\langle \bma -\langle \bma, \bmv\rangle \bmv, \bm\xi\rangle,
\end{align}
is a Gaussian random variable, with mean zero and variance
\begin{align}
\bE[\langle \bma -\langle \bma, \bmv\rangle \bmv, \bm\xi\rangle^2]
=\frac{1}{n}\|\bma -\langle \bma, \bmv\rangle \bmv\|_2^2=\frac{1}{n}\langle \bma, (\bmI_n-\bmv \bmv^\top)\bma\rangle=\frac{1+\oo(1)}{n}\langle \bma, (\bmI_n-\widehat\bmv \widehat\bmv^\top)\bma\rangle.
\end{align}
 This together with \eqref{e:betabb}, \eqref{e:avvbb} as well as our assumption \eqref{e:avbound} \begin{equation}
\frac{\sqrt{n}\widehat\beta}{\sqrt{\langle \bma, (\bmI_n-\widehat\bmv \widehat\bmv^\top)\bma\rangle}}\left[\big(1-\frac{1}{2\widehat\beta^2}\big)^{-1}\langle \bma, \widehat\bmv\rangle-\langle\bma, \bmv\rangle\right]
\xrightarrow{d} \cN(0,1).
\end{equation}

Under the same assumption, we have similar results for Cases \ref{c:case2}, \ref{c:case3}, \ref{c:case4}, by simply changing $(\beta, \bmv)$
in the righthand side of \eqref{e:clt} and \eqref{e:betaclt} to the corresponding expression.
\end{proof}

\begin{proof}[Proof of Corollary  \ref{coro2}]
Given the significance level $\alpha$, the asymptotic confidence intervals in Corollary \ref{coro2}  can be calculated from Corollary \ref{coro1} by bounding the absolute values of the left hand sides of \eqref{coro:clt1} at $z_{\alpha}$.
\end{proof}

\subsection{Proof of Theorem \ref{t:mainr} }

\begin{proof}[Proof of Theorem \ref{t:mainr}]
We define an auxiliary iteration, $\bmy_0=\bmu$ and
\begin{align}\label{e:recurr}
\bmy_{t+1}=\bmX[\bmy_t^{\otimes(k-1)}].
\end{align}
Then we have that $\bmu_t=\bmy_t/\|\bmy_t\|_2$.

For index $\bmj=(j_1, j_2, \cdots, j_{k-1})\in \qq{1, r}^{k-1}$. Let $\bmxi_\bmj=\bmZ[\bmv_{j_1}\otimes \bmv_{j_2}\otimes \cdots\otimes \bmv_{j_{k-1}}]$. Its entries
\begin{align}
\bmxi_\bmj(i)=\sum_{i_1,i_2,\cdots, i_{k-1}\in\qq{1, n}}\bmZ_{i i_1i_2\cdots i_{k-1}}\bmv_{j_1}(i_1)\bmv_{j_2}({i_2})\cdots \bmv_{j_{k-1}}(i_{k-1}),
\end{align}
are linear combination of Gaussian random variables, which is also Gaussian.
These entries are i.i.d. Gaussian variables with mean zero and variance $1/n$,
\begin{align}
\bE[\bm\xi_{\bmj}(i)^2]=\sum_{i_1,i_2,\cdots, i_{k-1}\in\qq{1, n}}\bE[\bmZ^2_{i i_1i_2\cdots i_{k-1}}]\bmv_{j_1}(i_1)^2\bmv_{j_2}({i_2})^2\cdots \bmv_{j_{k-1}}(i_{k-1})^2=\frac{1}{n}.
\end{align}

We can compute $\bmy_t$ iteratively:
\begin{align}\label{e:y1r}
\bmy_1=\bmX[\bmy_0^{\otimes(k-1)}]=\sum_{j=1}^r\beta_j\langle\bmy_0,\bmv_j\rangle^{k-1}\bmv_j+\bmZ[\bmy_0^{\otimes(k-1)}].
\end{align}
For the last term on the righthand side of \eqref{e:y1r}, we can decompose $\bmy_0^{\otimes(k-1)}$ as a projection on $\bmv_{j_1}\otimes\bmv_{j_2}\cdots\otimes \bmv_{j_{k-1}}$ for $\bmj\in \qq{1,r}^{k-1}$, and its orthogonal part:
\begin{align}
\bmy_0^{\otimes(k-1)}=\sum_{\bmj }\prod_{s=1}^{k-1}\langle \bmy_0, \bmv_{j_s}\rangle \bmv_{j_1}\otimes \bmv_{j_2}\otimes \cdots \otimes \bmv_{j_{k-1}}
+\sqrt{1-\left(\sum_{j=1}^r\langle \bmy_0, \bmv_j\rangle^2\right)^{(k-1)}}\bmtau_0,
\end{align}
where  the sum is over $\bmj\in \qq{1,r}^{k-1}$,  $\bmtau_0\in \otimes^{k}\bR^n$ and $\|\bmtau_0\|_2=1$.
Let $\bmxi_1=\bmZ[\bmtau_0]$. By our construction $\bmv_{j_1}\otimes \bmv_{j_2}\otimes \cdots \otimes \bmv_{j_{k-1}}$ for any $\bmj\in \qq{1,r}^{k-1}$ and $\bmtau_0$ are othorgonal to each other. Thanks to Lemma \ref{l:decompose}, conditioning on $\bmxi_\bmj:=\bmZ[\bmv_{j_1}\otimes \bmv_{j_2}\otimes \cdots\otimes \bmv_{j_{k-1}}]$ for index $\bmj=(j_1, j_2, \cdots, j_{k-1})\in \qq{1, r}^{k-1}$, $\bmxi_1=\bmZ[\bmtau_0]$ has the same law as $\tilde \bmZ[\bmtau_0]$, where $\tilde \bmZ$ is an independent copy of $\bmZ$. Since $\langle \bmtau_0, \bm\tau_0\rangle=1$, $\bm\xi_1$ is a Gaussian vector with each entry $\cN(0,1/n)$.  With those notations we can rewrite $\bmy_1$ as
\begin{align}\label{e:y12r}
\bmy_1=\sum_{j=1}^r\beta_j\langle\bmy_0,\bmv_j\rangle^{k-1}\bmv_j
+\sum_{\bmj }\prod_{s=1}^{k-1}\langle \bmy_0, \bmv_{j_s}\rangle \bmxi_\bmj
+\sqrt{1-\left(\sum_{j=1}^r\langle \bmy_0, \bmv_j\rangle^2\right)^{(k-1)}}\bmxi_1.
\end{align}

In the following we show that: 
\begin{claim}\label{c:bbexp}
We can compute $\bmy_2, \bmy_3, \cdots, \bmy_t$ inductively. The  Gram-Schmidt orthonormalization procedure gives  an orthogonal base of $\bmv_{j_1}\otimes \bmv_{j_2}\otimes \cdots \otimes \bmv_{j_{k-1}}$ for $\bmj\in \qq{1,r}^{k-1}$ and
$ \bmy_0^{\otimes (k-1)}, \bmy_1^{\otimes (k-1)}, \cdots, \bmy_{t-1}^{\otimes (k-1)}$ as:
\begin{align}\label{e:baser}
\{\bmv_{j_1}\otimes \bmv_{j_2}\otimes \cdots \otimes \bmv_{j_{k-1}}\}_{\bmj \in \qq{1,r}^{k-1}}, \bmtau_0, \bmtau_1, \cdots, \bmtau_{t-1}.
\end{align}
Let $\bmxi_{\bmj}=\bmZ[\bmv_{j_1}\otimes\bmv_{j_2}\otimes \cdots \otimes \bmv_{j_{k-1}}]$ for $\bmj=(j_1, j_2,\cdots, j_{k-1})\in \qq{1,r}^{k-1}$, and $\bmxi_{s+1}=\bmZ[\bmtau_s]$ for $0\leq s\leq t-1$. Conditioning on $\bmxi_{\bmj}=\bmZ[\bmv_{j_1}\otimes\bmv_{j_2}\otimes \cdots \otimes \bmv_{j_{k-1}}]$ for $\bmj=(j_1, j_2,\cdots, j_{k-1})\in \qq{1,r}^{k-1}$ and $\bmxi_{s+1}=\bmZ[\bmtau_s]$ for $0\leq s\leq t-2$, $\bmxi_{t}=\bmZ[\bmtau_{t-1}]$ is an independent Gaussian vector, with each entry $\cN(0,1/n)$.
 Then  $\bmy_t$ is in the following form
\begin{align}\label{e:btformr}
\bmy_t=a_t \bmv_t+b_t \bmw_t+c_t\bmxi_t,
\end{align}
where
\begin{align}\label{e:rrccer}
a_t \bmv_t=a_{t1}\bmv_1+a_{t2}\bmv_2+\cdots+a_{tr}\bmv_{r},
\quad b_t \bmw_t=\sum_{\bmj}b_{t\bmj}\bmxi_\bmj+b_{t1}\bmxi_1+\cdots+b_{tt-1}\bmxi_{t-1},
\end{align}
and $\|\bmv_1\|_2,\|\bmv_2\|_2,\cdots, \|\bmv_r\|_2,\|\bmw_t\|_2=1$.
\end{claim}

\begin{proof}[Proof of Claim \ref{c:bbexp}]
The Claim  \ref{c:bbexp} for $t=1$ follows from \eqref{e:y12r}.  In the following, assuming Claim \ref{c:bbexp} holds for $t$, we prove it for $t+1$.

Conditioning on $\bmxi_\bmj=\bmZ[\bmv_{j_1}\otimes \bmv_{j_2}\otimes \cdots\otimes \bmv_{j_{k-1}}]$ for index $\bmj=(j_1, j_2, \cdots, j_{k-1})\in \qq{1, r}^{k-1}$ and $\bmZ[\bmtau_s]=\bmxi_{s+1}$ for $0\leq s\leq t-2$, Lemma \ref{l:decompose} implies that $\bmxi_t=\bmZ[\bmtau_{t-1}]$ has the same law as $\tilde \bmZ[\bmtau_{t-1}]$, where $\tilde \bmZ$ is an independent copy of $\bmZ$. Since $\bm\tau_{t-1}$ is orthogonal to $\bmv_{j_1}\otimes \bmv_{j_2}\otimes \cdots\otimes \bmv_{j_{k-1}}$ for index $\bmj=(j_1, j_2, \cdots, j_{k-1})\in \qq{1, r}^{k-1}$ and $\bmZ[\bmtau_s]=\bmxi_{s+1}$ for $0\leq s\leq t-2$, $\bm\xi_t$ is an independent Gaussian random vector with each entry $\cN(0,1/n)$.

Let $\{\bmv_{j_1}\otimes \bmv_{j_2}\otimes \cdots \otimes \bmv_{j_{k-1}}\}_{\bmj \in \qq{1,r}^{k-1}}, \bmtau_0, \bmtau_1, \cdots, \bmtau_t$ be an orthogonal base for $\bmv_{j_1}\otimes \bmv_{j_2}\otimes \cdots \otimes \bmv_{j_{k-1}}$ for $\bmj\in \qq{1,r}^{k-1}$ and $ \bmy_0^{\otimes (k-1)}, \bmy_1^{\otimes (k-1)}, \cdots, \bmy_t^{\otimes (k-1)}$, obtained by the Gram-Schmidt orthonormalization procedure. More precisely, given those tensors $\{\bmv_{j_1}\otimes \bmv_{j_2}\otimes \cdots \otimes \bmv_{j_{k-1}}\}_{\bmj \in \qq{1,r}^{k-1}}, \bmtau_0, \bmtau_1, \cdots, \bmtau_{t-1}$, we denote 
\begin{align}\begin{split}\label{e:bbcoeffr}
&b_{(t+1)\bmj}
=\langle \bmy_t^{\otimes (k-1)},  \bmv_{j_1}\otimes \bmv_{j_2}\otimes \cdots\otimes \bmv_{j_{k-1}}\rangle,\quad \bmj=(j_1, j_2, \cdots, j_{k-1})\in \qq{1, r}^{k-1},\\
&b_{(t+1)s}
=\langle \bmy_s^{\otimes (k-1)},  \bm\tau_{s-1}\rangle,\quad 1\leq s\leq t,\quad
c_{t+1}=\langle \bmy_t^{\otimes (k-1)},  \bm\tau_{t}\rangle
\end{split}\end{align}
and
$\sum_{\bmj}b_{(t+1) \bmj}\bmv_{j_1}\otimes \bmv_{j_2}\otimes \cdots\otimes \bmv_{j_{k-1}}+b_{(t+1)1}\bm\tau_0+b_{(t+1)2}\bm\tau_1+\cdots b_{(t+1)t}\bm\tau_{t-1}$ is the projection of $\bmy_t^{\otimes (k-1)}$ on the span of  $\{\bmv_{j_1}\otimes \bmv_{j_2}\otimes \cdots \otimes \bmv_{j_{k-1}}\}_{\bmj \in \qq{1,r}^{k-1}}, \bmy_0^{\otimes (k-1)}, \bmy_1^{\otimes (k-1)}, \cdots, \bmy_{t-1}^{\otimes (k-1)}$.  Then we can write $\bmy_t^{\otimes (k-1)}$ in terms of the base \eqref{e:baser}
\begin{align}
\bmy_t^{\otimes (k-1)}
=\sum_{\bmj}b_{(t+1) \bmj}\bmv_{j_1}\otimes \bmv_{j_2}\otimes \cdots\otimes \bmv_{j_{k-1}}+b_{(t+1)1}\bm\tau_0+b_{(t+1)2}\bm\tau_1+\cdots b_{(t+1)t}\bm\tau_{t-1}+c_{t+1} \bm\tau_{t}.
\end{align}

The recursion \eqref{e:recurr} implies that
\begin{align}\label{e:yt+1r}
\bmy_{t+1}=\sum_{j=1}^r\beta_j(a_{tj}+b_t\langle \bmw_t, \bmv_j\rangle+c_t\langle \bmxi_t, \bmv_j\rangle )^{k-1}\bmv_j+ b_{t+1} \bmw_{t+1}+c_{t+1}\bmZ[\bmtau_t]
\end{align}
where
\begin{align}\begin{split}\label{e:defbt}
b_{t+1} \bmw_{t+1}
&=\bmZ[\sum_{\bmj}b_{(t+1) \bmj}\bmv_{j_1}\otimes \bmv_{j_2}\otimes \cdots\otimes \bmv_{j_{k-1}}+b_{(t+1)1}\bm\tau_0+b_{(t+1)2}\bm\tau_1+\cdots b_{(t+1)t}\bm\tau_{t-1}]\\
&=\sum_{\bmj}b_{(t+1) \bmj}\bmxi_\bmj+b_{(t+1)1}\bm\xi_1+b_{(t+1)2}\bmxi_2+\cdots b_{(t+1)t}\bmxi_{t}.
\end{split}\end{align}

Since $\bmtau_t$ is orthogonal to $\{\bmv_{j_1}\otimes \bmv_{j_2}\otimes \cdots \otimes \bmv_{j_{k-1}}\}_{\bmj \in \qq{1,r}^{k-1}}, \bmtau_0, \bmtau_1, \cdots, \bmtau_{t-1}$,
Lemma \ref{l:decompose} implies that conditioning on $\bmxi_\bmj=\bmZ[\bmv_{j_1}\otimes \bmv_{j_2}\otimes \cdots\otimes \bmv_{j_{k-1}}]$ for $\bmj=(j_1, j_2, \cdots, j_{k-1})\in \qq{1,r}^{k-1}$ and $\bmxi_{s+1}=\bmZ[\bmtau_s]$ for $0\leq s\leq t-1$, $\bmxi_{t+1}=\bmZ[\bmtau_t]$ is an independent Gaussian vector, with each entry $\cN(0,1/n)$. The above discussion gives us that
\begin{align}
\bmy_{t+1}=a_{t+1}\bmv_{t+1}+ b_{t+1} \bmw_{t+1}+c_{t+1}\bmxi_{t+1},\quad a_{t+1}=\sqrt{a_{(t+1)1}^2+a_{(t+1)2}^2+\cdots+a_{(t+1)r}^2}.
\end{align}
and
\begin{align}\label{e:arecur}
a_{(t+1) j}=\beta_j(a_{tj}+b_t\langle \bmw_t, \bmv_j\rangle+c_t\langle \bmxi_t, \bmv_j\rangle )^{k-1}, \quad 1\leq j\leq r.
\end{align}
\end{proof}

We recall that by our Assumption \ref{a:sumpr},
that
\begin{align}
1/\kappa\leq \left|\frac{\langle \bmu, \bmv_i\rangle}{\langle \bmu, \bmv_j\rangle}\right|\leq \kappa,
\end{align}
for all $1\leq i,j\leq r$. If $j_*=\argmax_j \beta_j\langle \bmu, \bmv_{j}\rangle^{k-2}$, it is necessary that $\beta_{j_*}\gtrsim \beta_1$, where the implicit constant depends on $\kappa$.

In the following, we study the case that  $\langle \bmu, \bmv_{j_*}\rangle>0$. The case $\langle \bmu, \bmv_{j_*}\rangle<0$ can be proven in exactly the same way, by simply changing $(\beta, \bmv_{j_*})$ with $((-1)^{k}\beta, -\bmv_{j_*})$. We prove by induction

\begin{claim}\label{c:eedy}
For any fixed time $t$, with probability at least $1-\OO(e^{-c(\log n)^2})$ the following holds: for any $s\leq t$,
\begin{align}\begin{split}\label{e:indr}
&|a_{sj_*}|\geq |a_{sj}|,\quad |a_s|\gtrsim |\beta_1|(\sum_{\bmj}|b_{s\bmj}|+|b_{s1}|+\cdots+|b_{s(s-1)}|),\\&
|a_s|\gtrsim n^\varepsilon \max\{\bm1(k\geq 3)|c_s/\beta_1^{1/(k-2)}|, |c_s/\sqrt n|\},
\end{split}\end{align}
and for $\bmj={(j_1, j_2, \cdots, j_{k-1})}\in \qq{1,r}^{k-1}$
\begin{align}\begin{split}\label{e:xibbr}
&\|\bm\xi_\bmj\| , \|\bm\xi_s\|_2 = 1+\OO(\log n/\sqrt n), \quad  |\langle \bmv_j, \bm\xi_\bmj\rangle|,|\langle \bma, \bm\xi_\bmj\rangle|, |\langle \bma, \bm\xi_s\rangle| \lesssim \log n/\sqrt n.\\
&\|{\rm{Proj}}_{{\rm Span}\{\bmv_1,\bmv_2,\cdots,\bmv_r, \{\bmxi_\bmj\}_{\bmj\in \qq{1,r}^{k-1}} ,\bm\xi_1,\cdots, \cdots, \bmxi_{s-1}\}}(\bmxi_s)\|_2 \lesssim \log n/\sqrt n.
\end{split}\end{align}
\end{claim}

\begin{proof}[Proof of Claim \ref{c:eedy}]
From \eqref{e:y12r}, we have
\begin{align}
\bmy_1&=\sum_{j=1}^r\beta_j\langle\bmy_0,\bmv_j\rangle^{k-1}\bmv_j
+\sum_{\bmj }\prod_{s=1}^{k-1}\langle \bmy_0, \bmv_{j_s}\rangle \bmxi_\bmj
+\sqrt{1-\left(\sum_{j=1}^r\langle \bmy_0, \bmv_j\rangle^2\right)^{(k-1)}}\bmxi_1\\
&=\sum_{j=1}^r a_{1j} \bmv_j+\sum_\bmj b_{1\bmj}\bmxi_\bmj+c_1\bmxi_1,
\end{align}
where $a_{1j}=\beta_j\langle\bmu,\bmv_j\rangle^{k-1}$ for $1\leq j\leq r$, $b_{1\bmj}=\prod_{s=1}^{k-1}\langle \bmu, \bmv_{j_s}\rangle$ for any index $\bmj=(j_1, j_2,\cdots, j_{k-1})$ and $c_1=
\sqrt{1-\left(\sum_{j=1}^r\langle \bmu, \bmv_j\rangle^2\right)^{(k-1)}}$.
Since $\bm\xi_\bmj$ are independent Gaussian vectors with each entry mean zero and variance $1/n$, the concentration for chi-square distribution implies that $\|\bm\xi_\bmj\|_2=1+\OO(\log n/\sqrt n)$ with probability $1-e^{c(\log n)^2}$.
Since $j_*=\argmax_j |\beta_j\langle \bmu, \bmv_j\rangle^{k-2}|$, combining with our Assumption \ref{a:sumpr}, it gives that $|a_{1j_*}|\geq |a_{1j}|/\kappa$. As a consequence, we also have that $|a_1|=\sqrt{a_{11}^2+a_{12}^2+\cdots+a_{1r}^2}\asymp |a_{1j_*}|$. Again using our Assumption \ref{a:sumpr}
\begin{align}
\sum_\bmj |b_{1\bmj}|\lesssim \left(\sum_{j=1}^{r}|\langle \bmu, \bmv_{j}\rangle|\right)^{k-1}\lesssim \sum_{j=1}^{r}|\langle \bmu, \bmv_{j}\rangle|^{k-1}\lesssim |a_{1j_*}|/\beta_{j_*}\lesssim |a_1|/\beta_1.
\end{align}
We can check that $|\beta_1 ^{1/{(k-2)}}a_1|\asymp|\beta_{j_*} ^{1/{(k-2)}}a_{1j_*}| =|\beta_{j_*}\langle \bmu, \bmv_{j_*}\rangle^{k-2}|^{(k-1)/(k-2)}\geq n^\varepsilon\geq n^\varepsilon|c_1|$, and $|\sqrt n a_1|\asymp |\sqrt n a_{1j_*}|=|\beta_{j_*}\langle \bmu, \bmv_{j_*}\rangle^{k-2}| |\sqrt n \langle \bmu, \bmv_{j_*}\rangle |\gtrsim n^\varepsilon\geq n^\varepsilon|c_1|$.
Moreover, conditioning on  $\bmxi_\bmj=\bmZ[\bmv_{j_1}\otimes \bmv_{j_2}\otimes \cdots\otimes \bmv_{j_{k-1}}]$ for $\bmj=(j_1, j_2, \cdots, j_{k-1})\in \qq{1,r}^{k-1}$, Lemma \ref{l:decompose} implies that $\bmxi_1=\bmZ[\bmtau_{0}]$  is an independent Gaussian random vector with each entry $\cN(0,1/n)$. By the standard concentration inequality, it holds that  with probability $1-e^{c(\log n)^2}$, $\|\bm\xi_1\|_2=1+\OO(\log n/\sqrt n)$, $|\langle \bma, \bm\xi_1\rangle|$ and the projection of $\bm\xi_1$ on the span of $\{\bmv_1,\bmv_2,\cdots, \bmv_r, \{\bmxi_\bmj\}_{\bmj\in\qq{1,r}^{k-1}}\}$ is bounded by $\log n/\sqrt n$. So far we have proved that \eqref{e:ind} and \eqref{e:xibb} hold for $t=1$.

In the following, we assume that \eqref{e:indr} and \eqref{e:xibbr} hold for $t$, and prove it for $t+1$. We recall from \eqref{e:rrccer} and \eqref{e:arecur} that
\begin{align}\label{e:tt1r}
a_{(t+1) j}=\beta_j(a_{tj}+b_t\langle \bmw_t, \bmv_j\rangle+c_t\langle \bmxi_t, \bmv_j\rangle )^{k-1},\quad b_t \bmw_t=\sum_{\bmj}b_{t\bmj}\bmxi_\bmj+b_{t1}\bmxi_1+\cdots+b_{tt-1}\bmxi_{t-1}.
\end{align}
By our induction hypothesis, we have
\begin{align}\label{e:tt2r}
|b_t\langle \bmw_t, \bmv_j\rangle|\lesssim \sum_{\bmj}|b_{t\bmj}\langle \bmxi_\bmj, \bmv_j\rangle|+|b_{t1}\langle \bmxi_1, \bmv_j\rangle|+\cdots+|b_{t(t-1)}\langle \bmxi_{t-1}, \bmv_j\rangle|
\lesssim (\log n /\sqrt{n})|a_t|/|\beta_1|,
\end{align}
and
\begin{align}\label{e:tt3r}
|c_t\langle \bmxi_t, \bmv_j\rangle|\lesssim (\log n/ \sqrt n)|c_t|\lesssim (\log n)|a_t|/n^\varepsilon.
\end{align}
It follows from plugging \eqref{e:tt2r} and \eqref{e:tt3r} into \eqref{e:tt1r}, we get
\begin{align}
a_{(t+1) j}=\beta_j(a_{tj}+b_t\langle \bmw_t, \bmv_j\rangle+c_t\langle \bmxi_t, \bmv_j\rangle )^{k-1}=\beta_j(a_{tj}+\OO(\log n|a_{t}|/n^\varepsilon))^k\lesssim \beta_j |a_t|^{k-1},
\end{align}
and especially
\begin{align}
a_{(t+1) {j_*}}=\beta_{j_*}(a_{tj_*}+\OO(\log n|a_{tj_*}|/n^\varepsilon))^k= (1+\OO(\log n/n^\varepsilon)) \beta_{j_*} a_{tj_*}^{k-1}.
\end{align}
Therefore, we conclude that
\begin{align}\label{e:abbr}
|a_{t+1j_*}|\asymp |\beta_{j_*} a_{tj_*}^{k-1}|\asymp|\beta  a_{t}^{k-1}|,
\end{align}
and
\begin{align}
|a_{(t+1)j}|
\lesssim \beta_j |a_{t}|^{k-1}\lesssim\beta_{j_*} |a_{tj_*}|^{k-1}\lesssim |a_{t+1j_*}|.
\end{align}

We recall from \eqref{e:bbcoeffr},
$\sum_{\bmj}b_{(t+1) \bmj}\bmv_{j_1}\otimes \bmv_{j_2}\otimes \cdots\otimes \bmv_{j_{k-1}}+b_{(t+1)1}\bm\tau_0+b_{(t+1)2}\bm\tau_1+\cdots b_{(t+1)t}\bm\tau_{t-1}$ is the projection of $\bmy_t^{\otimes (k-1)}$ on the span of  $\{\bmv_{j_1}\otimes \bmv_{j_2}\otimes \cdots \otimes \bmv_{j_{k-1}}\}_{\bmj \in \qq{1,r}^{k-1}}, \bmy_0^{\otimes (k-1)}, \bmy_1^{\otimes (k-1)}, \cdots, \bmy_{t-1}^{\otimes (k-1)}$. We also recall that $\{\bmv_{j_1}\otimes \bmv_{j_2}\otimes \cdots \otimes \bmv_{j_{k-1}}\}_{\bmj \in \qq{1,r}^{k-1}}, \bmtau_0, \bmtau_1, \cdots, \bmtau_{t-1}$ are obtained from $\{\bmv_{j_1}\otimes \bmv_{j_2}\otimes \cdots \otimes \bmv_{j_{k-1}}\}_{\bmj \in \qq{1,r}^{k-1}},, \bmy_0^{\otimes (k-1)}, \bmy_1^{\otimes (k-1)}, \cdots, \bmy_{t-1}^{\otimes (k-1)}$
by the Gram-Schmidt orthonormalization procedure. So we have that the span of  vectors $\{\bmv_{j_1}\otimes \bmv_{j_2}\otimes \cdots \otimes \bmv_{j_{k-1}}\}_{\bmj \in \qq{1,r}^{k-1}}, \bmtau_0, \bmtau_1, \cdots, \bmtau_{t-1}$ is the same as the span of vectors $\{\bmv_{j_1}\otimes \bmv_{j_2}\otimes \cdots \otimes \bmv_{j_{k-1}}\}_{\bmj \in \qq{1,r}^{k-1}},, \bmy_0^{\otimes (k-1)}, \bmy_1^{\otimes (k-1)}, \cdots, \bmy_{t-1}^{\otimes (k-1)}$, which is contained in the span of $\{\bmv_1, \bmv_2,\cdots, \bmv_r, \bmw_t,\bmy_0, \cdots, \bmy_{t-1}\}^{\otimes (k-1)}$. Moreover from the relation  \eqref{e:btformr} and \eqref{e:rrccer}, one can see that the span of $\{\bmv_1, \bmv_2,\cdots, \bmv_r, \bmw_t,\bmy_0, \cdots, \bmy_{t-1}\}$ is the same as the span of
$\{\bmv_1, \bmv_2,\cdots, \bmv_r, \{\bmxi_\bmj\}_{\bmj\in\qq{1,r}^{k-1}},\bmxi_1, \cdots, \bmxi_{t-1}\}$.
It follows that
\begin{align}\begin{split}\label{e:sumbr}
&|b_{t+1}|\lesssim\sqrt{\sum_{\bmj}b_{(t+1) \bmj}^2+b_{(t+1)1}^2+b_{(t+1)2}^2+\cdots +b_{(t+1)t}^2}\\
&=\|{\rm{Proj}}_{{\rm Span}\{\{\bmv_{j_1}\otimes \bmv_{j_2}\otimes \cdots \otimes \bmv_{j_{k-1}}\}_{\bmj \in \qq{1,r}^{k-1}}, \bmtau_0, \bmtau_1, \cdots, \bmtau_{t-1}\}}(a_t \bmv_t+b_t \bmw_t+c_t\bmxi_t)^{\otimes (k-1)}\|_2 \\
&\leq\|{\rm{Proj}}_{{\rm Span}\{\bmv_1, \bmv_2,\cdots, \bmv_r, \bmw_t,\bmy_0, \cdots, \bmy_{t-1}\}^{\otimes (k-1)}}(a_t \bmv_t+b_t \bmw_t+c_t\bmxi_t)^{\otimes (k-1)}\|_2\\
&\leq\|{\rm{Proj}}_{{\rm Span}\{\bmv, \bmw_t,\bmy_0, \cdots, \bmy_{t-1}\}}(a_t \bmv_t+b_t \bmw_t+c_t\bmxi_t)\|_2^{k-1}\\
&=\|a_t \bmv_t+b_t \bmw_t+c_t{\rm{Proj}}_{{\rm Span}\{\bmv_1, \bmv_2,\cdots, \bmv_r, \{\bmxi_\bmj\}_{\bmj\in\qq{1,r}^{k-1}},\bmxi_1, \cdots, \bmxi_{t-1}\}}(\bmxi_t)\|_2^{k-1}\\
&\lesssim \left(|a_t|+|b_t|+\frac{\log n|c_t|}{\sqrt n}\right)^{k-1}\lesssim |a_t|^{k-1}\lesssim |a_{t+1}|/\beta_1,
\end{split}\end{align}
where in the first line we used \eqref{e:defbt}, and in the last line of \eqref{e:sumbr} we used our induction hypothesis that $\|{\rm{Proj}}_{{\rm Span}\{\bmv_1,\bmv_2,\cdots,\bmv_r, \{\bmxi_\bmj\}_{\bmj\in \qq{1,r}^{k-1}},\bmxi_1, \cdots, \bmxi_{t-1}\}}(\bmxi_t)\|_2\lesssim \log n/\sqrt n$.

Finally we estimate $c_{t+1}$. We recall from \eqref{e:bbcoeffr}, the coefficient $c_{t+1}$ is the remainder of $\bmy_t^{\otimes (k-1)}$ after projecting on $\{\bmv_{j_1}\otimes \bmv_{j_2}\otimes \cdots \otimes \bmv_{j_{k-1}}\}_{\bmj \in \qq{1,r}^{k-1}}, \bmtau_0, \bmtau_1, \cdots, \bmtau_{t-1}$. It is bounded by the remainder of $\bmy_t^{\otimes (k-1)}$ after projecting on $\{\bmv_{j_1}\otimes \bmv_{j_2}\otimes \cdots \otimes \bmv_{j_{k-1}}\}_{\bmj \in \qq{1,r}^{k-1}}$,
\begin{align}
|c_{t+1}|\leq \|\bmy_t^{\otimes (k-1)}-a_t^{k-1}\bmv_t^{\otimes (k-1)}\|_2
=\|(a_t \bmv_t+b_t \bmw_t+c_t\bmxi_t)^{\otimes (k-1)}-a_t^{k-1}\bmv_t^{\otimes (k-1)}\|_2
.
\end{align}
The difference $(a_t \bmv_t+b_t \bmw_t+c_t\bmxi_t)^{\otimes (k-1)}-a_t^{k-1}\bmv_t^{\otimes (k-1)}$ is a sum of terms in the following form,
\begin{align}\label{e:ower}
\bm\eta_1\otimes \bm\eta_2\otimes \cdots\otimes \bm\eta_{k-1},
\end{align}
where $\bm\eta_1, \bm\eta_2, \cdots, \bm\eta_{k-1}\in \{a_t\bmv_t, b_t\bmw_t+c_t\bm\xi_t\}$, and at least one of them is $b_t\bmw_t+c_t\bm\xi_t$. We notice that by our induction hypothesis, $\|b_t\bmw_t+c_t\bm\xi_t\|_2\lesssim |b_t|\|\bmw_t\|_2+|c_t|\|\bm\xi_t\|_2\lesssim |b_t|+|c_t|$. For the $L_2$ norm of \eqref{e:ower}, each copy of  $a_t\bmv_t$ contributes $a_t$ and each copy of $b_t\bmw_t+c_t\bm\xi_t$ contributes a factor $|b_t|+|c_t|$. We conclude that
\begin{align}\label{e:ctboundr}
|c_{t+1}|\leq
\|(a_t \bmv_t+b_t \bmw_t+c_t\bmxi_t)^{\otimes (k-1)}-a_t^{k-1}\bmv_t^{\otimes (k-1)}\|_2
\lesssim \sum_{r=1}^{k-1}|a_t|^{k-1-r}(|b_t|+|c_t|)^r.
\end{align}
Combining with \eqref{e:abbr} that $|a_{t+1}|\asymp |\beta_1||a_t|^{k-1}$, we divide both sides of \eqref{e:ctboundr} by $|\beta_1||a_t|^{k-1}$,
\begin{align}\label{e:xtbbr}
\frac{|c_{t+1}|}{|a_{t+1}|}
&\lesssim \frac{1}{|\beta_1|}\sum_{r=1}^{k-1}\left(\frac{|b_t|}{|a_t|}+\frac{|c_t|}{|a_t|}\right)^r
\lesssim
\frac{1}{|\beta_1|}\sum_{r=1}^{k-1}\left(\frac{1}{|\beta_1|}+\frac{|c_t|}{|a_t|}\right)^r
\end{align}
There are three cases:
\begin{enumerate}
\item If $|c_t|/|a_t|\geq 1$, then
\begin{align}
\frac{|c_{t+1}|}{|a_{t+1}|}
\lesssim
\frac{1}{|\beta_1|}\sum_{r=1}^{k-1}\left(\frac{1}{|\beta_1|}+\frac{|c_t|}{|a_t|}\right)^r
\lesssim\frac{1}{|\beta_1|}\left(\frac{|c_t|}{|a_t|}\right)^{k-1}.
\end{align}
If $k=2$, then $|c_{t+1}|/|a_{t+1}|\lesssim (|c_t|/|a_t|)/n^{\varepsilon}$. If $k\geq 2$, by our induction hypothesis $|c_t|/|a_t|\lesssim \beta_1^{1/(k-2)}/n^\varepsilon$. Especially, $(|c_t|/|a_t|)^{k-2}/|\beta_1|\lesssim 1/n^{\varepsilon}$. We still get that $|c_{t+1}|/|a_{t+1}|\lesssim (|c_t|/|a_t|)/n^{\varepsilon}$.
\item If $1/|\beta_1|\lesssim |c_t|/|a_t|\leq 1$, then
\begin{align}
\frac{|c_{t+1}|}{|a_{t+1}|}
\lesssim
\frac{1}{|\beta_1|}\sum_{r=1}^{k-1}\left(\frac{1}{|\beta_1|}+\frac{|c_t|}{|a_t|}\right)^r
\lesssim\frac{1}{|\beta_1|}\left(\frac{|c_t|}{|a_t|}\right)\lesssim \frac{1}{n^\varepsilon}\left(\frac{|c_t|}{|a_t|}\right).
\end{align}
\item Finally for $ |c_t|/|a_t|\lesssim 1/|\beta_1|$, we will have
\begin{align}
\frac{|c_{t+1}|}{|a_{t+1}|}
\lesssim
\frac{1}{|\beta_1|}\sum_{r=1}^{k-1}\left(\frac{1}{|\beta_1|}+\frac{|c_t|}{|a_t|}\right)^r
\lesssim\frac{1}{|\beta_1|}\left(\frac{1}{|\beta_1|}\right)\lesssim \frac{1}{|\beta_1|^2}.
\end{align}
\end{enumerate}
In all these cases we have $|c_{t+1}|/|a_{t+1}|\lesssim\min\{\sqrt n, \bm1(k\geq 3)|\beta_1|^{1/(k-2)}\}/n^{\varepsilon}$.
This finishes the proof of the induction \eqref{e:indr}.

For \eqref{e:xibbr}, since $\bmtau_t$ is orthogonal to $\{\bmv_{j_1}\otimes \bmv_{j_2}\otimes \cdots \otimes \bmv_{j_{k-1}}\}_{\bmj \in \qq{1,r}^{k-1}}, \bmtau_0, \bmtau_1, \cdots, \bmtau_{t-1}$, Lemma \ref{l:decompose} implies that conditioning on $\bmxi_\bmj=\bmZ[\bmv_{j_1}\otimes \bmv_{j_2}\otimes \cdots\otimes \bmv_{j_{k-1}}]$ for index $\bmj=(j_1, j_2, \cdots, j_{k-1})\in \qq{1, r}^{k-1}$  and $\bmxi_{s+1}=\bmZ[\bmtau_s]$ for $0\leq s\leq t-1$, $\bmxi_{t+1}=\bmZ[\bmtau_t]$ is an independent Gaussian vector, with each entry $\cN(0,1/n)$. By the standard concentration inequality, it holds that  with probability $1-e^{c(\log n)^2}$, $\|\bm\xi_{t+1}\|_2=1+\OO(\log n/\sqrt n)$, $|\langle \bma, \bm\xi_{t+1}\rangle|$ and the projection of $\bm\xi_{t+1}$ on the span of $\{\bmv_1, \bmv_2,\cdots, \bmv_r, \{\bmxi_\bmj\}_{\bmj\in\qq{1,r}^{k-1}},\bmxi_1, \cdots, \bmxi_{t-1}\}$ is bounded by $\log n/\sqrt n$. This finishes the proof of the induction \eqref{e:xibbr}.
\end{proof}

Next, using \eqref{e:indr} and \eqref{e:xibbr} as input,
we prove that for
\begin{align}\label{e:tbbdr}
t\geq 1+\frac{1}{\varepsilon}\left(\frac{1}{2}+\frac{2\log|\beta_1|}{\log n}\right)+\frac{\log\log(\sqrt n|\beta_1|)}{\log(k-1)}
\end{align}
we have
\begin{align}\label{e:btformrcopy}
\bmy_t=\sum_{j=1}^ra_{tj} \bmv_j+\sum_\bmj b_{t\bmj}\bmxi_\bmj+b_{t1}\bmxi_1+\cdots+b_{tt-1}\bmxi_{t-1}+c_t\bmxi_t,
\end{align}
such that
\begin{align}\begin{split}\label{e:newboundr}
&|a_{tj}|\lesssim \left(\frac{\log n}{\sqrt n}\frac{1}{|\beta_1|}\right)^{k-1}|a_{tj_*}|, \quad j\neq j_*,\\
&b_{t(j_*,j_*,\cdots ,j_*)}=\frac{a_{tj_*}}{\beta_{j_*}}+\OO\left(\frac{\log n |a_t|}{|\beta_1|^2\sqrt n}\right),
|b_{(t+1) \bmj_*}|\lesssim \frac{\log n}{\sqrt n|\beta_1|^2}|a_{tj_*}|,\quad
\bmj_*=(j_*,j_*,\cdots, j_*),\\
&|b_{t1}|, |b_{t2}|,\cdots, |b_{t(t-1)}|\lesssim \frac{(\log n)^{1/2}|a_{t}|}{|\beta_1|^{3/2}n^{1/4}},\quad |c_t|\lesssim |a_t|/\beta_1^2
\end{split}\end{align}

Let $x_t=|c_t/a_t|\leq n^{-\varepsilon}|\beta|^{1/(k-2)}$, and $r_t=\max_{j\neq j_*}(\beta_j^{1/(k-2)}a_{tj})/(\beta_{j_*}^{1/(k-2)}a_{tj_*})$.
For $t=1$, our Assumption \ref{a:sumpr2} implies that
\begin{align}\begin{split}
\beta_j^{1/(k-2)}a_{1j}
&\leq (\beta_j\langle\bmu,\bmv_j\rangle^{k-2})^{(k-1)/(k-2)}\\
&\leq ((1-1/\kappa)\beta_{j_*}\langle\bmu,\bmv_{j_*}\rangle^{k-2})^{(k-1)/(k-2)}
\leq (1-1/\kappa)\beta_{j_*}^{1/(k-2)}a_{1j_*}.
\end{split}\end{align}
Thus we have that $r_1\leq (1-1/\kappa)$.
We recall from \eqref{e:tt1r}
\begin{align}\begin{split}
\beta_j^{1/(k-2)}a_{(t+1) j}
&=\left(\beta_j^{1/k-2}(a_{tj}+b_t\langle \bmw_t, \bmv_j\rangle+c_t\langle \bmxi_t, \bmv_j\rangle)\right)^{k-1}\\
&=\left(\beta_j^{1/k-2}(a_{tj}+\OO\left(|a_t|\frac{\log n(1/|\beta_1|+x_t)}{\sqrt n}\right)\right)^{k-1},
\end{split}\end{align}
where we used \eqref{e:indr} and \eqref{e:xibbr}.
Thus it follows that
\begin{align}\begin{split}\label{e:r_t}
r_{t+1}
&= \max_{j\neq j_*}\left( \frac{\beta_j^{1/k-2}(a_{tj}+\OO\left(|a_t|\log n(1/|\beta_1|+x_t)/\sqrt n\right)}{\beta_{j_*}^{1/k-2}(a_{t{j_*}}+\OO\left(|a_t|\log n(1/|\beta_1|+x_t)/\sqrt n\right)}\right)^{k-1}\\
&\leq \left( \frac{r_t+\OO\left(\log n(1/|\beta_1|+x_t)/\sqrt n\right)}{1+\OO\left(\log n(1/|\beta_1|+x_t)/\sqrt n\right)}\right)^{k-1}
\end{split}\end{align}
For $x_t$, \eqref{e:xtbbr} implies
\begin{align}\label{e:x_t}
x_{t+1}\lesssim \frac{1}{|\beta_1|}\sum_{r=1}^{k-1}\left(\frac{1}{|\beta_1|}+x_t\right)^r,
\end{align}
from the discussion after \eqref{e:xtbbr}, we have that either
$x_{t+1}\lesssim 1/|\beta_1|^2$, or $x_{t+1}\lesssim x_t/n^\varepsilon$.
Since $x_1=|c_1/a_1|\lesssim n^{1/2-\varepsilon}$, and $r_1\leq (1-1/\kappa)$ we conclude from \eqref{e:r_t} and \eqref{e:x_t} that
\begin{align}\label{e:cbbr}
x_t=|c_t/a_t|\lesssim 1/\beta_1^2,\quad r_t\lesssim (\log n/(|\beta_1|\sqrt n))^{k-1}, \end{align}
when
\begin{align}t\geq \frac{1}{\varepsilon}\left(\frac{1}{2}+\frac{2\log|\beta_1|}{\log n}\right)+\frac{\log\log(\sqrt n|\beta_1|)}{\log(k-1)}.
\end{align}

To derive the upper bound of $b_{t1}, b_{t2},\cdots, b_{t(t-1)}$, we use \eqref{e:sumbr}.
\begin{align}\begin{split}\label{e:sumb2r}
&\phantom{{}={}}\sum_{\bmj}b_{(t+1) \bmj}^2+b_{(t+1)1}^2+b_{(t+1)2}^2+\cdots +b_{(t+1)t}^2\\
&\leq
\|a_t \bmv_t+b_t \bmw_t+c_t{\rm{Proj}}_{{\rm Span}\{\bmv_1, \bmv_2,\cdots, \bmv_r, \{\bmxi_\bmj\}_{\bmj\in\qq{1,r}^{k-1}},\bmxi_1, \cdots, \bmxi_{t-1}\}}(\bmxi_t)\|_2^{2(k-1)}\\
&=\left(a_t^2+\OO\left(|a_t|\left(|b_t|+|c_t|\right)\frac{\log n}{\sqrt n}+\left(|b_t|+|c_t|\frac{\log n}{\sqrt n}\right)^2\right)\right)^{k-1},
\end{split}\end{align}
where we used \eqref{e:xibbr}. The first term $b_{(t+1) \bmj}$ is the projection of $\bmy_t^{\otimes (k-1)}$ on $\bmv_{j_1}\otimes \bmv_{j_2}\otimes\cdots \otimes \bmv_{j_{k-1}}$,
\begin{align}\label{e:bbjj}
b_{(t+1) \bmj}=\prod_{s=1}^{k-1}\langle a_t \bmv_t+b_t \bmw_t+c_t\bmxi_{t},\bmv_{j_s}\rangle
=\prod_{s=1}^{k-1}\left(a_{tj_s}+ \OO\left(\frac{\log n(|b_t|+|c_t|)}{\sqrt n}\right)\right),
\end{align}
and
\begin{align}\begin{split}\label{e:sumb3r}
\sum_{\bmj}b_{(t+1) \bmj}^2
&=\left(\sum_{s=1}^{k-1}|\langle a_t \bmv_t+b_t \bmw_t+c_t\bmxi_{t},\bmv_{j_s}\rangle|^{2}\right)^k\\
&=\left(a_t^2+\OO\left(|a_t|\left(|b_t|+|c_t|\right)\frac{\log n}{\sqrt n}+\left(|b_t|+|c_t|\frac{\log n}{\sqrt n}\right)^2\right)\right)^{k-1},
\end{split}\end{align}
where we used \eqref{e:xibb} that $|\langle \bmxi_\bmj, \bmv_j\rangle|, |\langle \bmxi_1, \bmv_j\rangle|,\cdots, |\langle \bmxi_t, \bmv_j\rangle|\lesssim \log n/\sqrt n$.
Now we can take difference of \eqref{e:sumb2r} and \eqref{e:sumb3r}, and use that $|b_t|\lesssim |a_t|/|\beta_1|$ from \eqref{e:indr} and $|c_t|\lesssim |a_t|/|\beta_1|^2$ from \eqref{e:cbbr},
\begin{align}\label{e:bttar}
 b_{(t+1)1}^2+b_{(t+1)2}^2+\cdots +b_{(t+1)t}^2
\lesssim a_t^{2(k-1)}\frac{\log n}{|\beta|\sqrt n}.
\end{align}
Using \eqref{e:bbjj} and \eqref{e:cbbr},
we get that
\begin{align}\begin{split}\label{e:bbbtr}
&b_{(t+1) \bmj_*}=a_{tj_*}^{k-1}\left(1+\OO\left(\frac{\log n}{\sqrt n|\beta_1|}\right)\right),\quad \bmj_*=(j_*, j_*,\cdots, j_*)\\
&|b_{(t+1) \bmj}|\lesssim \frac{\log n}{\sqrt n|\beta_1|}|a_{tj_*}|^{k-1}, \quad \bmj\neq\bmj_*.
\end{split}\end{align}
From \eqref{e:tt1r},  \eqref{e:abbr} and \eqref{e:cbbr}, we have that
\begin{align}\begin{split}\label{e:at+1r}
&a_{(t+1)j_*}=\beta_{j_*} b_{(t+1)\bmj_*}=\beta_{j_*}a_{tj_*}^{k-1}\left(1+\OO\left(\frac{\log n}{\sqrt n|\beta_1|}\right)\right),\\
&|a_{(t+1)j}|\lesssim \left(\frac{\log n}{\sqrt n|\beta_1|}\right)^{k-1}|a_{(t+1)j_*}|,\quad j\neq j_*.
\end{split}\end{align}
Using the above relation, we can simplify \eqref{e:bttar} and \eqref{e:bbbtr} as
\begin{align}
|b_{(t+1)1}|, |b_{(t+1)2}|, \cdots |b_{(t+1)t}|
\lesssim \frac{(\log n)^{1/2}|a_{t+1}|}{|\beta_1|^{3/2}n^{1/4}}.
\end{align}
and
\begin{align}\begin{split}
&b_{(t+1) \bmj_*}=\frac{a_{(t+1)j_*}}{\beta_{j_*}}\left(1+\OO\left(\frac{\log n}{\sqrt n|\beta_1|}\right)\right),\\
&|b_{(t+1) \bmj}|\lesssim \frac{\log n}{\sqrt n|\beta_1|^2}|a_{(t+1)j_*}|, \quad \bmj\neq\bmj_*.
\end{split}\end{align}
This finishes the proof of \eqref{e:newboundr}.

With the expression \eqref{e:newboundr}, we can process to prove our main results \eqref{e:cltr} and \eqref{e:betacltr}.
Thanks to \eqref{e:xibbr} and \eqref{e:btformrcopy}, for $t$ satisfies \eqref{e:tbbdr}, we have that with probability at least $1-\OO(e^{-c(\log n)^2})$
\begin{align}\label{e:ytnorm}
\|\bmy_t\|_2^2=a_{tj_*}^2\left(1+\frac{1}{\beta_{j_*}^2}+\frac{2\langle \bmv_{j_*}, \bmxi_{\bmj_*}\rangle}{\beta_{j_*}}+\OO\left(\frac{\log n}{\sqrt n}\left(\frac{\log n}{\sqrt n |\beta_1|}\right)^{k-1}+\frac{\log n}{\beta_1^2\sqrt n}+\frac{(\log n)^{3/2}}{|\beta_1|^{3/2}n^{3/4}}+\frac{1}{\beta_1^4}\right)\right)
\end{align}
where $\bmj_*=(j_*, j_*,\cdots, j_*)$. By rearranging it we get
\begin{align}\label{e:ytnorminvr}
1/\|\bmy_t\|_2
=a_{tj_*}^2\left(1-\frac{1}{2\beta_{j_*}^2}-\frac{2\langle \bmv_{j_*}, \bmxi_{\bmj_*}\rangle}{\beta_{j_*}}+\OO\left(\frac{\log n}{\sqrt n}\left(\frac{\log n}{\sqrt n |\beta_1|}\right)^{k-1}+\frac{\log n}{\beta_1^2\sqrt n}+\frac{(\log n)^{3/2}}{|\beta_1|^{3/2}n^{3/4}}+\frac{1}{\beta_1^4}\right)\right)
\end{align}
We can take the inner product $\langle \bma, \bmy_t\rangle$, and multiply \eqref{e:ytnorminvr}
\begin{align}\begin{split}
\langle \bma, \bmu_t\rangle
=\frac{\langle \bma, \bmy_t\rangle}{\|\bmy_t\|_2}
&=\sgn(a_{tj_*})\left(\left(1-\frac{1}{2\beta_{j_*}^2}\right)\langle\bma, \bmv_{j_*}\rangle+\frac{\langle \bma, \bmxi_{\bmj_*}\rangle-\langle \bma, \bmv_{j_*}\rangle\langle \bmv_{j_*}, \bmxi_{\bmj_*}\rangle}{\beta}\right)\\
&+\OO_\bP\left(\frac{\log n}{\sqrt n}\left(\frac{\log n}{\sqrt n |\beta_1|}\right)^{k-1}+\frac{\log n}{\beta_1^2\sqrt n}+\frac{(\log n)^{3/2}}{|\beta_1|^{3/2}n^{3/4}}+\frac{1}{\beta_1^4}\right),
\end{split}\end{align}
where we used \eqref{e:xibb} that with high probability $|\langle \bma, \bm\xi_{\bmj}\rangle|$, $|\langle \bma, \bm\xi_s\rangle|$ for $1\leq s\leq t$ are bounded by $\log n/\sqrt n$.
This finishes the proof of \eqref{e:cltr}. For $\widehat\beta$ in \eqref{e:betacltr}, we have that
\begin{align}\label{e:newbetar}
\bmX[\bmu_t^{\otimes k}]
=\frac{\bmX[\bmy_t^{\otimes k}]}{\|\bmy_t\|_2^k}
=\frac{\langle \bmy_t, \bmX[\bmy_t^{\otimes(k-1)}]\rangle}{\|\bmy_t\|_2^k}
=\frac{\langle \bmy_t, \bmy_{t+1}\rangle}{\|\bmy_t\|_2^k}
.
\end{align}

Thanks to \eqref{e:cbbr}, \eqref{e:at+1r} and \eqref{e:xibbr}, for $t$ satisfies \eqref{e:tbbdr},  with probability at least $1-\OO(e^{-c(\log n)^2})$, we can write the first term on the righthand side of \eqref{e:newbeta}, we have
\begin{align}
\bmy_{t+1}=\sum_j a_{(t+1)j} \bmv_j+\sum_{\bmj}b_{(t+1)\bmj}\bmxi_\bmj+b_{(t+1)1}\bmxi_1+\cdots+b_{(t+1)t}\bmxi_{t}+c_{t+1}\bmxi_{t+1},
\end{align}
where $|c_{t+1}|\lesssim |a_t|^{k-1}/\beta^2$,
\begin{align}\begin{split}
&a_{(t+1)j_*}=\beta_{j_*} b_{(t+1)\bmj_*}=\beta_{j_*}a_{tj_*}^{k-1}\left(1+\OO\left(\frac{\log n}{\sqrt n|\beta_1|}\right)\right),\\
&|a_{(t+1)j}|\lesssim \left(\frac{\log n}{\sqrt n|\beta_1|}\right)^{k-1}|a_{(t+1)j_*}|,\quad j\neq j_*
\end{split}\end{align}
and
\begin{align}\begin{split}
&b_{(t+1) \bmj_*}=\frac{a_{(t+1)j_*}}{\beta_{j_*}}\left(1+\OO\left(\frac{\log n}{\sqrt n|\beta_1|}\right)\right),\\
&|b_{(t+1) \bmj}|\lesssim \frac{\log n}{\sqrt n|\beta_1|^2}|a_{(t+1)j_*}|, \quad \bmj\neq\bmj_*,
\\
 & |b_{(t+1)1}|, |b_{(t+1)2}|+\cdots +|b_{(t+1)t}|
\lesssim a_t^{k-1}\frac{(\log n)^{1/2}}{|\beta|^{1/2}n^{1/4}}.
\end{split}\end{align}
From the discussion above, combining with \eqref{e:btformrcopy} and \eqref{e:newboundr} with straightforward computation, we have
\begin{align}\label{e:crossr}
\langle \bmy_t, \bmy_{t+1}\rangle=
\beta_{j_*} a_{tj_*}^k\left(1+\frac{1}{\beta_{j_*}^2}+\frac{(k+1)\langle \bm\xi_{\bmj_*}, \bmv_{j_*}\rangle}{\beta_{j_*}}+\OO\left(\frac{\log n}{\sqrt n}\left(\frac{\log n}{\sqrt n |\beta_1|}\right)^{k-1}+\frac{\log n}{\beta_1^2 \sqrt n}+\frac{(\log n)^{3/2}}{|\beta_1|^{3/2}n^{3/4}}\right)\right).
\end{align}
By plugging \eqref{e:ytnorminvr} and \eqref{e:crossr} into \eqref{e:newbetar}, we get
\begin{align}
\bmX[\bmu_t^{\otimes k}]
&=\sgn(a_{tj_*}^k)\left(\beta_{j_*}+\langle \bm\xi_{\bmj_*}, \bmv_{j_*}\rangle -\frac{k/2-1}{\beta_{j_*}}\right)\\
&+\OO\left(\frac{\log n}{\sqrt n}\left(\frac{\log n}{\sqrt n |\beta_1|}\right)^{k-1}+\frac{\log n}{|\beta_1| \sqrt n}+\frac{(\log n)^{3/2}}{|\beta_1|^{1/2}n^{3/4}}+\frac{1}{|\beta_1|^3}\right)
\end{align}

Since by our assumption, in Case 1 we have that $\beta_{j^*}>0$. Thanks to \eqref{e:at+1r} $a_{t+1j_*}=\beta a_{tj_*}^{k-1}(1+\oo(1))$, especially $a_{t+1j_*}$ and $a_{tj_*}$ are of the same sign. In the case $\langle \bmu, \bmv_{j_*}\rangle>0$, we have $a_{1j_*}=\beta\langle \bmu, \bmv_{j_*}\rangle^{k-1}>0$. We conclude that $a_{tj_*}>0$. Therefore $\sgn(\bmX[\bmu_t^{\otimes k}])=\sgn(a_{tj_*})^k=+$, and it follows that
\begin{align}
\bmX[\bmu_t^{\otimes k}]
&=\beta_{j_*}+\langle \bm\xi_{\bmj_*}, \bmv_{j_*}\rangle -\frac{k/2-1}{\beta_{j_*}}\\
&+\OO\left(\frac{\log n}{\sqrt n}\left(\frac{\log n}{\sqrt n |\beta_1|}\right)^{k-1}+\frac{\log n}{|\beta_1| \sqrt n}+\frac{(\log n)^{3/2}}{|\beta_1|^{1/2}n^{3/4}}+\frac{1}{|\beta_1|^3}\right)
\end{align}
This finishes the proof of \eqref{e:betacltr}. The Cases \ref{c:case2r}, \ref{c:case3r}, \ref{c:case4r}, by simply changing $(\beta_{j_*}, \bmv_{j_*})$ in the righthand side of \eqref{e:cltr} and \eqref{e:betacltr} to the corresponding limit.

\end{proof}

\subsection{Proof of Theorem \ref{t:randominit}}
\begin{proof}[Proof of Theorem \ref{t:randominit}]
We first prove  \eqref{e:defp}. If $\bmu$ is uniformly distributed over the unit sphere, then it has the same law as $\bm\eta/\|\bm\eta\|_2$, where $\bm\eta$ is an $n$-dim standard Gaussian vector, with each entry $\cN(0,1)$. With this notation
\begin{align}\label{e:newep}
|\beta_j\langle \bmu, \bmv_j\rangle^{k-2}|
=|\beta_j\langle \bm\eta, \bmv_j\rangle^{k-2}|/\|\bm\eta\|_2^{k-2},
\end{align}
and we can rewrite $\bP(i=\argmax_j |\beta_j\langle \bmu, \bmv_j\rangle^{k-2}|)$
as
\begin{align}
\bP(i=\argmax_j |\beta_j\langle \bmu, \bmv_j\rangle^{k-2}|)=
\bP(i=\argmax_j |\beta_j\langle \bm\eta, \bmv_j\rangle^{k-2}|).
\end{align}
Since $\bmv_1, \bmv_2,\cdots, \bmv_r$ are orthonormal vectors, $\langle\bmv_1, \bm\eta\rangle, \langle \bmv_2,\bm\eta\rangle,\cdots, \langle \bmv_r, \bm\eta\rangle$ are independent standard Gaussian random variables. Then we have
\begin{align}\begin{split}
p_i&=\bP(i=\argmax_j |\beta_j\langle \bm\eta, \bmv_j\rangle^{k-2}|)
=\bP( |\beta_i/\beta_\ell|^{1/{k-2}}\langle \bm\eta, \bmv_i\rangle|\geq |\langle \bm\eta, \bmv_\ell\rangle|, \text{ for all } i\neq \ell)\\
&=\int_0^\infty \sqrt{\frac{2}{\pi}}e^{-x^2/2}\left(\prod_{\ell\neq i}\int_0^{\left(\frac{|\beta_i|}{|\beta_\ell|}\right)^{\frac{1}{k-2}}x}\sqrt{\frac{2}{\pi}}e^{-y^2/2}\rd y\right)\rd x.
\end{split}\end{align}
This gives \eqref{e:defp}. Using the fact we can rewrite $\bmu$ as $\bm\eta/\|\bm\eta\|_2$, we have that with probability $1-\OO(1/\sqrt \kappa)$,
\begin{align}
1/\sqrt{\kappa n}\leq |\langle \bmu, \bmv_i\rangle|\leq \sqrt{\kappa/n},
\end{align}
for all $1\leq i\leq r$. Thus Assumption \eqref{a:sumpr} holds, and especially,
\begin{align}
\max_j |\beta_j\langle \bmu, \bmv_j\rangle^{k-2}|
\geq |\beta_1\langle \bmu, \bmv_1\rangle^{k-2}|
\geq |\beta_1(1/\sqrt{\kappa n})^{k-2}|\gtrsim n^\varepsilon.
\end{align}
Theorem \ref{t:randominit} then follows directly from Theorem \ref{t:mainr}.
\end{proof}

\subsection{Proof of Corollaries \ref{coro3}, \ref{coro4} and \ref{coro5}}

\begin{proof}[Proof of Corollary \ref{coro3}]
According to the definition of $\bmxi$ in \eqref{e:cltr} of Theorem \ref{t:mainr}, i.e.
$\bmxi=\bmZ[\bmv_{j_*}^{\otimes(k-1)}]$, is an $n$-dim vector, with each entry i.i.d. $\cN(0,1/n)$ Gaussian random variable. 
We  see that
$$
\langle \bm\xi, \bmv\rangle\stackrel{d}{=} \cN\left(0,1/n\right).
 $$
  Especially with high probability we will have that $|\langle \bm\xi, \bmv\rangle|\lesssim \log n/\sqrt n$.
Then we conclude from \eqref{e:betacltr}, with high probability it holds
\begin{align}\label{e:betabb}
\widehat \beta=\beta_{j_*}+\OO\left(\frac{1}{\beta_{j_*}}+\frac{\log n}{\sqrt n}\right).
\end{align} 
With the bound \eqref{e:betabb}, we can replace $\langle \bma, \bmv\rangle/(2\beta^2)$ on the righthand side of \eqref{e:cltr} by $\langle \bma, \bmv\rangle/(2\widehat\beta^2)$, which gives an error \begin{align}
\left|\frac{\langle \bma, \bmv\rangle}{2\beta_{j_*}^2}-\frac{\langle \bma, \bmv\rangle}{2\widehat\beta^2}\right|=\OO\left(|\langle \bma, \bmv\rangle|\left(\frac{1}{|\beta_{j_*}|^4}+\frac{\log n}{|\beta_{j_*}|^3\sqrt n}\right)\right).
\end{align}
Combining the above discussion together, 
we can rewrite \eqref{e:cltr}  as
\begin{align}\begin{split}\label{e:avvbb}
\langle \bma, \widehat\bmv\rangle-\left(1-\frac{1}{2\widehat\beta^2}\right)\langle\bma, \bmv\rangle
&=\frac{\langle \bma, \bmxi\rangle-\langle \bma, \bmv_{j_*}\rangle\langle \bmv_{j_*}, \bmxi\rangle}{\beta_{j_*}}\\
&+\OO_\bP\left(\frac{\log n}{\sqrt n}\left(\frac{\log n}{\sqrt n |\beta_1|}\right)^{k-1}+\frac{\log n}{|\beta_1|^2\sqrt n}+\frac{(\log n)^{3/2}}{|\beta_1|^{3/2}n^{3/4}}+\frac{1}{|\beta_1|^4}\right),
\end{split}\end{align}
with high probability, where we used that $|\beta_{j_*}|\gtrsim |\beta_1|$.

Again thanks to the definition of $\bmxi$ in \eqref{e:cltr}  of Theorem \ref{t:main}, i.e.
$\bmxi=\bmZ[\bmv^{\otimes(k-1)}]$, is an $n$-dim vector, with each entry i.i.d. $\cN(0,1/n)$ Gaussian random variable, we see that
\begin{align}
\langle \bma, \bmxi\rangle-\langle \bma, \bmv_{j_*}\rangle\langle \bmv_{j_*}, \bmxi\rangle
=\langle \bma -\langle \bma, \bmv_{j_*}\rangle \bmv_{j_*}, \bm\xi\rangle,
\end{align}
is a Gaussian random variable, with mean zero and variance
\begin{align}
\bE[\langle \bma -\langle \bma, \bmv_{j_*}\rangle \bmv_{j_*}, \bm\xi\rangle^2]
&=\frac{1}{n}\|\bma -\langle \bma, \bmv_{j_*}\rangle \bmv_{j_*}\|_2^2=\frac{1}{n}\langle \bma, (\bmI_n-\bmv_{j_*} \bmv_{j_*}^\top)\bma\rangle\\
&=\frac{1+\oo(1)}{n}\langle \bma, (\bmI_n-\widehat\bmv_{j_*} \widehat\bmv_{j_*}^\top)\bma\rangle.
\end{align}
 This together with \eqref{e:betabb}, \eqref{e:avvbb} as well as our assumption \eqref{e:avboundk} \begin{equation}
\frac{\sqrt{n}\widehat\beta}{\sqrt{\langle \bma, (\bmI_n-\widehat\bmv \widehat\bmv^\top)\bma\rangle}}\left[\big(1-\frac{1}{2\widehat\beta^2}\big)^{-1}\langle \bma, \widehat\bmv\rangle-\langle\bma, \bmv_{j_*}\rangle\right]
\xrightarrow{d} \cN(0,1).
\end{equation}
Under the same assumption, we have similar results for Cases \ref{c:case2}, \ref{c:case3}, \ref{c:case4}, by simply changing $(\beta_{j_*}, \bmv_{j_*})$
in the righthand side of \eqref{e:clt} and \eqref{e:betaclt} to the corresponding expression.
\end{proof}

\begin{proof}[Proof of Corollary \ref{coro4}]
For $k\geq 3$ and $|\beta_1|\geq n^{(k-2)/2+\varepsilon}$, the assumption \ref{e:avboundk} holds trivially. The claim \eqref{e:bba} follows from \eqref{coro:clt3}. For \eqref{e:bbc}, we recall that in   \eqref{e:betacltrr}, $\bmxi=\bmZ[\bmv_{i}^{\otimes(k-1)}]$, is an $n$-dim vector, with each entry i.i.d. $\cN(0,1/n)$ Gaussian random variable. 
We  see that
$$
\langle \bm\xi, \bmv_i\rangle\stackrel{d}{=} \cN\left(0,1/n\right).
 $$
  Especially with high probability we will have that $|\langle \bm\xi, \bmv\rangle|\lesssim \log n/\sqrt n$.
Then we conclude from \eqref{e:betacltrr}, with high probability it holds
\begin{align}\label{e:betabbb}
\widehat \beta=\beta_{i}+\OO\left(\frac{1}{\beta_{i}}+\frac{\log n}{\sqrt n}\right).
\end{align} 
With the bound \eqref{e:betabbb}, we can replace $(k/2-1)/\beta_i$ on the righthand side of \eqref{e:betacltrr} by $(k/2-1)/\widehat\beta$, which gives an error \begin{align}
\left|\frac{k/2-1}{\beta_i}-\frac{k/2-1}{\widehat\beta}\right|=\OO\left(\frac{1}{|\beta_{1}|^2}+\frac{\log n}{|\beta_{1}|\sqrt n}\right),
\end{align}
where  we used that $|\beta_{i}|\gtrsim |\beta_1|$.
Combining the above discussion together, 
we can rewrite \eqref{e:betacltrr}  as
\begin{align}\label{e:eet}
\beta_i=\widehat\beta+\frac{k/2-1}{\widehat\beta}-\langle \bm\xi, \bmv_{i}\rangle
+\OO_\bP\left(\frac{\log n}{\sqrt n}\left(\frac{\log n}{\sqrt n |\beta_1|}\right)^{k-1}+\frac{\log n}{|\beta_1| \sqrt n}+\frac{(\log n)^{3/2}}{|\beta_1|^{1/2}n^{3/4}}+\frac{1}{|\beta_1|^2}\right).
\end{align}
Since $\langle \bm\xi, \bmv_i\rangle\stackrel{d}{=} \cN\left(0,1/n\right)$, and the error term in \eqref{e:eet} is much smaller than $1/\sqrt n$. We conclude from \eqref{e:eet}
\begin{align}
\sqrt n \left(\beta_i-\widehat\beta+\frac{k/2-1}{\widehat\beta}\right)\xrightarrow{d} \cN(0,1).
\end{align}
This finishes the proof of \eqref{e:bbc}.

\end{proof}

\begin{proof}[Proof of Corollary \ref{coro5}]
Given the significance level $\alpha$, the asymptotic confidence intervals in Corollary \ref{coro5}  can be calculated from Corollary \ref{coro4} by bounding the absolute values of the left hand sides of \eqref{e:bba} and \eqref{e:bbc} at $z_{\alpha}$.
\end{proof}

\bibliographystyle{abbrv}
\bibliography{References}

\end{document}